\documentclass[11pt, reqno]{amsart}
\usepackage[T1]{fontenc}
\usepackage[american]{babel}
\usepackage[toc,page]{appendix}
\usepackage{amssymb,amsthm,amsmath} 
\usepackage{indentfirst}
\usepackage{color}%per scrivere con diversi colori
\usepackage{graphicx}%per inserire grafici
\usepackage{pdfpages}
\usepackage{mathtools}
\usepackage{amsthm}
\usepackage{tikz}
\usetikzlibrary{matrix}
\usepackage{enumerate}
\usepackage{tikz-cd}
\usepackage[framemethod=default]{mdframed}
\usepackage{hyperref}
\usepackage{verbatim}

\usepackage[numbers, sort]{natbib} 

\newtheorem{theorem}{Theorem}[section]
\newtheorem{lemma}[theorem]{Lemma}
\newtheorem{Assumption}[theorem]{Assumption}
\newtheorem{proposition}[theorem]{Proposition}

\theoremstyle{definition}

\theoremstyle{remark}
\newtheorem{remark}[theorem]{Note}

\newcommand{\cB}{\mathcal{B}}
\newcommand{\N}{\mathbb{N}}

\newcommand{\cP}{\mathcal{P}}

\newcommand{\pa}{\partial}

\newcommand{\be}{\begin{equation}}
	\newcommand{\ee}{\end{equation}}

\newcommand{\R}{\mathbb{R}}
\newcommand{\T}{\mathbb T}
\newcommand{\iin}{^{i,N}}
\newcommand{\jn}{^{j,N}}

\newcommand{\fL}{\mathfrak{L}}
\newcommand{\cF} {\mathcal F}

\newcommand{\Z} {\mathbb Z}
\newcommand{\mP}{\mathbb{P}}
\newcommand{\mE}{\mathbb{E}}
\newcommand{\cC}{\mathcal{C}}
\newcommand{\cG}{\mathcal{G}}
\newcommand{\cI}{\mathcal{I}}
\newcommand{\mX}{\mathbb{X}}
\newcommand{\mY}{\mathbb{Y}}

\newcommand{\mM}{\mathbb{M}}

\newcommand{\tT}{\tilde{T}}

\counterwithin{equation}{section}

\begin{document}

	\title{Well-posedness and Stationary solutions of McKean-Vlasov (S)PDEs}
    \author{L. Angeli$^{(1)}$}
    \address{$^{(1)}$Mathematics Department, Heriot-Watt University, and Maxwell Institute, Edinburgh, l.angeli@hw.ac.uk}

    \author{J. Barr\'{e}$^{(2)}$}
    \address{$^{(2)}$Institut Denis Poisson, Universit\'e d’Orl\'{e}ans, CNRS, Universit\'{e} de Tours, France and Institut Universitaire de France, julien.barre@univ-orleans.fr}

    \author{M. Kolodziejczyk$^{(3)}$}
    \address{$^{(3)}$Mathematics Department, Heriot-Watt University, and Maxwell Institute, Edinburgh, mk2006@hw.ac.uk}

    \author{M. Ottobre$^{(4)}$}
    \address{$^{(4)}$Mathematics Department, Heriot-Watt University, and Maxwell Institute, Edinburgh, m.ottobre@hw.ac.uk}

	\begin{abstract}
		This paper is composed of two parts. In the first part
		we consider McKean-Vlasov Partial Differential Equations (PDEs), obtained as thermodynamic  limits of interacting particle systems (i.e. in the limit $N\rightarrow \infty$, where $N$ is the number of particles). It is well-known that, even when the particle system has a unique invariant measure (stationary solution), the limiting PDE very often displays a phase transition: for certain choices of  (coefficients and) parameter values, the PDE has a unique stationary solution, but as  the value of the parameter varies multiple stationary states appear. In the first part of this paper, we add to this stream of literature and consider a specific instance of a McKean-Vlasov type equation, namely the Kuramoto model on the torus perturbed by a symmetric double-well potential, and show that this PDE  undergoes the type of phase transition just described, as the diffusion coefficient is varied.    In the second part of the paper, we consider a rather general class of McKean-Vlasov PDEs on the torus (which includes both the original Kuramoto model and the Kuramoto model in double well potential of part one) perturbed by (strong enough) infinite-dimensional additive noise. To the best of our knowledge, the resulting  Stochastic PDE, which we refer to as the {\em Stochastic McKean-Vlasov equation},   has not been studied before, so we first study its well-posedness.  We then show that the addition of noise to the PDE has the effect of restoring uniqueness of the stationary state in the sense that,  irrespective of the choice of coefficients and parameter values in the McKean-Vlasov PDE,  the Stochastic McKean-Vlasov PDE always admits at most one invariant measure.


	\end{abstract}
	\maketitle
	
	\section{Introduction}

	Consider the following system of $N$ interacting particles
	\be\label{initialPS}
	dX_t\iin= - V'(X_t\iin) + \frac{1}{N}\sum_{j=1}^N F'(X_t\jn-X_t\iin)  dt+ \sqrt{2\sigma} d\beta_t^i, 
	\ee
	where, for every $i=1\dots N$, $X_t\iin \in \T$ represents the position of the $i-$th particle on the torus $\T$ of length $2\pi$,  $\T:= \mathbb R/2\pi \mathbb Z$, the potentials $V$ and $F$,   $V,F: \T \rightarrow \R$, are, respectively the {\em environmental} and {\em inter-particle potential},  $'$ is derivative with respect to the argument of the function and the $\beta_t^i$'s are independent one-dimensional standard Brownian motions. 
	
	It is well known that, as $N \rightarrow \infty$, the particle system \eqref{initialPS} converges to the non-local PDE  
	\begin{equation}\label{PDEsimpleparticlesystem}
		\pa_t \rho_t(x)= \sigma \pa_{xx} \rho_t(x) +\pa_x \left[\left(V'(x) + (F'\ast \rho_t)(x)\right) \rho_t(x)\right]\,,     
	\end{equation}
	for the unknown $\rho_t(x)= \rho(t,x):\R_+ \times \T \rightarrow \R$,\footnote{Throughout the 
		paper, for any quantity, say $Y$, that depends on time, we use interchangeably the notation $Y_t$ or
		$Y(t)$ to denote time-dependence} in the sense that the empirical measure $\mu^N_t:= \frac1 N \sum
	\delta_{X_t^{i,N}}$ of the particle system \eqref{initialPS}, which is, for each $t>0$,  a random probability measure,  converges weakly to the (deterministic) function $\rho_t$, 
	provided this is true for the corresponding initial data, i.e provided $\mu_0^N$ converges to 
	$\rho_0$ \cite{KipnisLandim, Pulvirenti}. Another way of seeing this is the following: as 
	$N\rightarrow \infty$, the particles become independent ({\em propagation of chaos}) and, in the 
	limit, the motion of each of them is described by the following SDE
	\be\label{non-linearSDE}
	dX_t  = - \left(V'(X_t) + \int_{\T} F'(y-X_t)\rho_t(dy)
	\right) dt+ \sqrt{2\sigma} d\beta_t \,,  \quad X_t \in \T\, ,
	\ee
	where $\beta_t$ is a one-dimensional standard Brownian motion and $\rho_t$ is the law of $X_t$ at time $t$, so that the above evolution is non-linear in the sense that the process depends on its own law, i.e. it is {\em non-linear in the sense of McKean}. By It\^o's formula, the law $\rho_t$ of $X_t$ is a solution of the PDE \eqref{PDEsimpleparticlesystem} and invariant measures of the SDE \eqref{non-linearSDE} are precisely the stationary solutions of \eqref{PDEsimpleparticlesystem}. 
	
	In this paper we will consider (specific instances of) the PDE \eqref{PDEsimpleparticlesystem} as well as the following SPDE
	\begin{align}\label{SPDEintro}
		& \pa_t u = \pa_{xx}u+\pa_x \left[ V^{'}u+(F^{'}*u) u \right ]+ {Q}^{1/2} \pa_t W,\quad  (0,T) \times \T  \nonumber\\
		& u(t,0) =u(t,2\pi),\qquad t \in [0,T]  \\
		& u(0,x)  =u_0(x), \qquad  x \in \T,  \nonumber
	\end{align} 
	for the unknown $u(t,x):\R_+ \times \mathbb T \rightarrow \R$ (having omitted, as customary, dependence on the realization $\omega$). 
	The evolution \eqref{SPDEintro} is obtained from \eqref{PDEsimpleparticlesystem} by adding  infinite dimensional  noise, as  in  the above $W(t,x)$ is  cylindrical Wiener noise while $Q$ is a positive, symmetric and trace class operator (precise notation and assumptions in Section \ref{sec:notation}).   Evolutions of the type \eqref{PDEsimpleparticlesystem}  are often called  McKean-Vlasov PDEs  (or also  granular media or aggregation equation) and for this reason we refer to the SPDE \eqref{SPDEintro} as to {\em McKean-Vlasov  SPDE} or, more accurately,  {\em Stochastic McKean-Vlasov equation} (SMKV). The former name  could be misleading so we clarify that  the solution to \eqref{SPDEintro}, seen as a function-space-valued process,    is {\em not} an (infinite dimensional) McKean-Vlasov SDE, as the process does not depend on its own law. The investigation of infinite-dimensional McKean-Vlasov SDEs has been recently tackled in \cite{vlasovwellposedness, {vlasovwellposedness2}}. However, to the best of our knowledge, the McKean-Vlasov SPDE \eqref{SPDEintro} that we consider here has not been studied in the literature, so this paper constitutes a first work on the topic. We will give more detail on comparison between the infinite dimensional McKean-Vlasov SDEs of \cite{vlasovwellposedness, {vlasovwellposedness2}}  and the evolution \eqref{SPDEintro} in  Note \ref{note:wellposthm}.  
	
	\bigskip
	The evolution \eqref{PDEsimpleparticlesystem} and many of its variants have been extensively studied in the PDE, statistical physics, stochastic analysis and modelling literature. In particular well-posedness  for \eqref{PDEsimpleparticlesystem} and \eqref{non-linearSDE} have been studied in a number of works, see e.g. \cite{CarrilloMcCannVillani, ImkellerPeithmann}, for a PDE and probabilistic perspective, respectively. So we will not discuss this aspect in the present work.  Beyond an intrinsic theoretical interest, McKean-Vlasov evolutions emerge naturally as models in opinion formation, animal navigation, in the study of rating systems and of neural networks, to mention just a few application fields where this equation plays a central role, and for this reason they have attracted growing attention for a few decades now, see \cite{PareschiToscani, Meleard, Pulvirenti} and references therein for modelling aspects. 
	
	When $V \equiv 0$ and $F(x) = - K\cos x$, for some $K>0$,   $F$ acts as an attractive force between 
	particles and the parameter $K>0$ modulates the strength of the force; for these choices of $V$ and
	$F$,  system \eqref{initialPS} and the related PDE are often referred to as the {\em Kuramoto 
	model}  (or also  the mean field classical XY model), which has been subject of careful study, see 
	e.g.  \cite{BertiniGiacominPoquet, constantin, constantin2}.\footnote{We point out for completeness
	that in the physics literature the Kuramoto model also includes the effect of an intrinsic 
	oscillation frequency for each particle, see \cite{Kuramoto, Kuramoto2}.}  In particular, the 
	asymptotic behaviour of this model has been described in detail,  using various approaches, see 
	\cite{Vukadinovic, GiacominPakPel, CarrilloMcCannVillani},  and references therein and, more 
	recently, \cite{pavliotis}.  One of the phenomena of interest  is the following: while (for each 
	$N$ fixed) the Kuramoto particle system has a unique invariant measure (uniqueness being 
	straightforward in view of ellipticity), the Kuramoto PDE undergoes a phase transition. Namely, 
	there exists $K_c>0$ (depending on the noise strength $\sigma$) such that for $K<K_c$ the PDE has a
	unique stationary solution -- the uniform distribution on the torus --  while for $K>K_c$ the 
	equation has a whole manifold of stationary states. This can be intuitively understood as follows: 
	under the effect of the force  $F(x) = -K \cos x$, particles are attracted to each other; however, 
	as soon as $\sigma > 0$, there is a competition between such an attractive force and the effect of 
	the noise, which makes particles diffuse on the torus (and in this sense it can be seen as having a
	formally `repulsive' effect). If the noise is strong enough then the particles spread homogeneously
	around the torus; if this is not the case then a non-homogeneous steady state appears, concentrated
	say at zero; then,  by rotational symmetry, a whole manifold of steady states follows.  
	Indeed, the Kuramoto model enjoys a rotational symmetry which is key in the study of the dynamics: it is easy to see that, if $X_t^{i,N}$ is a solution of the Kuramoto particle system, then also $X_t^{i,N}+ \theta$ is a solution, for any $\theta \in [0, 2\pi]$; accordingly,   if $\rho_t$ solves the Kuramoto PDE, then also $\rho_t(x+\theta)$ solves the same PDE, for any $\theta \in [0,2\pi]$.

	For a general  $V\neq 0$ the rotational symmetry of the Kuramoto model  no longer holds; this is the case on which we focus  in the first part of this paper. In particular, in Section \ref{sec:sec4} we consider the McKean-Vlasov PDE \eqref{PDEsimpleparticlesystem}   when  
	\begin{equation}\label{doppia stella}
		V(x)=\cos(2x)\, ,\quad F(x)=-\cos x,\quad  x \in \T  \, , 
	\end{equation}	 
	and show that, in this setting, there exists a critical value of the noise, $\sigma_c$, such that when $\sigma > \sigma_c$ the PDE \eqref{PDEsimpleparticlesystem} has a unique stationary solution, whereas when $\sigma< \sigma_c$ (and small enough) there exist exactly three stationary solutions (see Theorem \ref{thm:mainthm3-inv-meas} for a precise statement), namely the homogeneous distribution and the other two concentrated at either minima of the double-well potential $V$.\footnote{In the discussion of the Kuramoto model, to be consistent with the cited literature,  we implicitly fixed the value of $\sigma$ and discussed the phase transition as $K$ varies. In this paper we (equivalently) fix the value of $K$ ($K=1$, see \eqref{doppia stella}) and study the behaviour as $\sigma$ varies.} This is due to the fact that, with respect to the case when $V=0$, the particles are subject not only to the competition between attractive force and noise, but also to the environmental potential, which introduces a tendency for the particles to converge towards the minima of the potential. 
	
	Overall, the above discussion should serve the purpose of showing that the behaviour of the PDE \eqref{PDEsimpleparticlesystem} can be rather complex, and in particular the number of stationary solutions of the PDE depends on the detailed properties of the potential $V$ and of the inter-particle force $F$, as well as on the value of $\sigma$. This is certainly not the only PDE that has a complicated set of stationary solutions, and indeed similar observations could be made e.g. for the Allen Cahn and Navier-Stokes equations, to mention just a few examples;  for such equations it has been observed that addition of (appropriately strong) noise to the PDE `restores' uniqueness of the equilibrium state, in the sense that the Stochastic Allen-Cahn and the Stochastic Navier Stokes equations have a unique invariant measure (stationary state), see \cite{NilsBerglundbook, {FIMas},{Kuksin}, Weinan, {FlandoliNavier}}.    In  the second part of this  paper we add to this stream of literature  and show that the SMKV equation \eqref{SPDEintro} admits at most one invariant measure, irrespective of the choice of $V, F$ and $\sigma$.  We will discuss more the technical aspects involved in proving this result below and in the next section, however we point out that  we work here under the assumption that the added noise is `strong enough', see Theorem \ref{theorem:existence and uniqueness inv measure} and comments afterwards. It is not a priori obvious what is the `minimum amount of noise' one can add to the PDE \eqref{PDEsimpleparticlesystem} so that the resulting SPDE has a unique invariant measure. This is a question that requires more sophisticated tools than those we use in this paper, such as those developed in \cite{MattinglyHairer},   and we will tackle such a question in future work. 
	
	The fact that the set of stationary solutions of the PDE is generally more complex than the set of stationary solutions  of the corresponding SPDE can be seen as the infinite dimensional analogue of what is well known to happen in finite dimension: to fix ideas, let $V(x)$ be a multi-well potential on the torus (but clearly the same is true in $\R$ for multi-well confining potentials); then the deterministic ODE 
	\be\label{ODEintro}
	dx_t = -V'(x_t) dt, \quad x_t \in \T, t \geq 0, 
	\ee
	has multiple steady states (as many as the critical points of $V$). However the Langevin equation
	$$
	dx_t= - V'(x_t) dt+ d\beta_t, 
	$$
	has a unique invariant measure, as the (elliptic, hence `strong enough') Brownian noise $\beta_t$ allows full exploration of state space. If we see the steady states of \eqref{ODEintro} as invariant measures (by considering Dirac deltas concentrated at the critical points), then we can say that the addition of noise has `restored' uniqueness of the invariant measure. 
	
	To summarise, this paper is divided in two parts: in the first part  (Section \ref{sec:sec4}) we 
	study the PDE \eqref{PDEsimpleparticlesystem} when $V$ and $F$ are as in \eqref{doppia stella} and
	show that the number of steady states depends on the strength of the noise $\sigma$. In the second
	part of the paper, we first prove the well-posedness in mild sense of the SPDE \eqref{SPDEintro}
	(Section \ref{Well-Posedness McKean_Vlasov}) and then we show that such an SPDE admits at most one invariant measure. In 
	order to do so, one needs to prove that the semigroup associated with the evolution 
	\eqref{SPDEintro} is irreducible and Strong Feller. We prove irreducibility in Section \ref{sec: 
		irreducibility} and Strong Feller property in Section \ref{sec:sec7}. The next section, Section 
	\ref{sec:notation},  contains precise statements of the main results and more thorough relation to
	literature.  The proofs of Section \ref{sec:sec4} are independent of the proofs  Section \ref{Well-Posedness McKean_Vlasov} and following sections.

	{\bf Further Motivation. }  As we have mentioned,  we are not aware of any works on either well-posedness or  ergodic properties of the SMKV equation \eqref{SPDEintro}, so this paper is, primarily,  a first contribution towards establishing properties of such an evolution.  In contrast,  there is a large literature on the following SPDE
	\begin{align}\label{SPDEPScomnoise}
		\pa_t u_t(x)  = \pa_x \left[\left(V'(x) + (F'\ast u_t)(x)\right) u_t(x) + (\sigma+\tilde\sigma) \pa_x u_t(x)\right] - \sqrt{2 \tilde\sigma} (\pa_x u_t)\pa_t \beta_t \,,
	\end{align}
	which can be viewed, for the purposes of this discussion,  as a different stochastic perturbation of the PDE \eqref{PDEsimpleparticlesystem}, \cite{Lacker} and references therein.
	The difference between \eqref{SPDEintro} and \eqref{SPDEPScomnoise} is that in the former the (infinite dimensional) noise is additive, while in the latter noise is multiplicative; more importantly, \eqref{SPDEPScomnoise} has transport (gradient) structure, while \eqref{SPDEintro} does not - fact that is source of many complications. The SPDE \eqref{SPDEPScomnoise} has attracted a lot of attention as it can be obtained in the limit $N\rightarrow \infty$ of the following particle system
	\be\label{PSwithcommonnoise}
	dX_t^{i,N}= - V'(X_t^{i,N}) - \frac{1}{N}\sum_{j=1}^N F'(X_t^{i,N}-X_t^{j,N})  dt+ \sqrt{2\sigma} \beta_t^i+ \sqrt{2\tilde\sigma} d\beta_t,  
	\ee
	where, crucially, the Brownian noise $\beta_t$  is the same for each particle, and  $\tilde\sigma>0$, \cite{Lacker}. 
	Because the noise same $\beta_t$ acts on all the particles, the limit of the particle system is no longer deterministic, and it is stochastic instead. In upcoming work we will investigate one possible interpretation  of \eqref{SPDEintro} in relation to interacting particle limits and some of the results of this work form a basis for future work in this direction.

	\section{Notation and Main Results}\label{sec:notation}
	In this section we first introduce some notation and recall some basic facts; we then state the main results of this paper and comment on them in turn. 
	
	\subsection{Notation}\label{subsec:notation} In what follows  $L^2(\T;\R)$ denotes the separable Hilbert space of $2\pi$-periodic real-valued  square-integrable functions, endowed  with the scalar product 
	$$\langle f,g \rangle_{L^2(\T;\R)}:=\int_{\T} f(x)g(x)\,dx,\qquad f,g \in L^2(\T;\R).$$
	We fix  $\{e_k\}_{k \in \mathbb{Z}}$ to be the following orthonormal Fourier basis of $L^2(\T;\R)$
	\begin{equation}\label{fourier}
		\begin{cases}
			e_k(x)= \frac{1}{\sqrt{\pi}}\sin(kx),\quad k>0, & \\ 
			e_0(x)=\frac{1}{\sqrt{2\pi}},\quad k=0, & \\ 
			e_k(x)=\frac{1}{\sqrt{\pi}}\cos(kx),\quad k<0 \, ,
		\end{cases}
	\end{equation}
	and for any  $f \in L^2(\T;\R)$, we denote by $f_k:=\langle f,e_k \rangle_{L^2(\T;\R)},\,k \in \Z$  the $k$'th Fourier coefficient of $f$, so that 
	\begin{equation}\label{serie}
		f=\sum \limits_{k \in \Z} f_k e_k,  \notag
	\end{equation}
	where the equality holds in $L^2(\T;\R)$.
	%As a particular choice of $Q$ in \eqref{spde Q} we set $Q:=\left( -B \right)^{-\gamma}$ where $B:=A-id$ and $\gamma$ is a real parameter ($id$ denotes the identity operator of $L^2(\T;\R)$). It is clear that $\left(-B \right)^{-\gamma}$ is of trace class if and only if $\gamma > \frac{1}{2}$. Indeed, since $$ \left( -B \right)^{-\gamma}e_k=(1+k^2)^{-\gamma}e_k,\quadk \in \Z,$$ it follows that $$\text{Tr}\left( \left( -B \right)^{-\gamma} \right)=\sum \limits_{k \in \Z} (1+k^2)^{-\gamma}$$ which is finite if and only if $\gamma > \frac{1}{2}$.  
	
	We denote by $A$ the one-dimensional Laplacian, i.e. the unbounded linear operator $A:\mathcal{D}(A) \subset L^2(\T;\R) \to L^2(\T;\R)$, $A=\pa_{xx}$,  which acts on the elements of the  basis \eqref{fourier} as 
	$$Ae_k=-k^2e_k,\quad k \in \Z.$$
	With this notation we rewrite the problem \eqref{SPDEintro} as 
	\begin{equation}\label{spde Q}
		\begin{cases}
			\pa_t u=Au+\pa_x \left[ V^{'}u+(F^{'}*u) u \right ]+ {Q}^{1/2} \pa_t W,\quad  (0,T) \times \T , & \\
			u(t,0)=u(t,2\pi),\quad t \in [0,T] , & \\
			u(0,x)=u_0(x), \quad  x \in \T, & \\
		\end{cases}
	\end{equation} 
	for the unknown $ u=u(t)\in L^2(\T;\R)$ for every $t>0$ and initial datum $u_0 \in L^2(\T;\R)$. As a {\em standing assumption}, throughout the operator $Q:L^2(\T;\R) \rightarrow L^2(\T;\R)$ is a  positive, symmetric and trace-class operator such that
	\be\label{Qactsonbasis}
	Q e_k=\lambda_k^2 e_k\, , \quad k \in\Z. 
	\ee
	Because $Q$ is trace class, the eigenvalues of $Q$ are summable, namely
	\begin{equation}\label{eqn:Qtrace-class}
		\sum_{k\in\Z}\lambda_k^2<+\infty \,.
	\end{equation}
	Moreover,  $W$ is an $L^2(\T;\R)$-valued cylindrical Wiener process defined over a filtered probability space $(\Omega,\mathcal{F},\mathcal{F}_t,\mathbb{P})$. That is, $W$ can be represented as
	$$W(t,x)=\sum \limits_{k \in \Z} e_k(x) \beta_t^k,$$
	where $\{e_k\}_{k \in \Z}$ is the orthonormal basis given in \eqref{fourier} and $\{\beta_k\}_{k \in \Z}$ is a family of standard real-valued independent Brownian motions. 
	
	We will work with mild solutions of \eqref{spde Q}, so we recall that a continuous $L^2(\T;\R)$-valued stochastic process $u(t)$, $t \in [0,T]$, is said to be a {\em mild solution} to \eqref{spde Q} if the following holds
	\begin{equation} \label{mild solution}
		\begin{split}
			u(t)& =e^{tA}u_0+\int_0^t e^{(t-s)A}\partial_x \left[  V^{'}u(s) \right]\,ds\\
			& +\int_0^t e^{(t-s)A}\partial_x \left[(F^{'} \ast u)(s)u(s) \right]\,ds+W_{A}(t),\qquad \text{$t \in [0,T]$, $\mathbb{P}$-a.s.}, 
		\end{split}
	\end{equation}
	where in the above $W_A$ denotes the {\em stochastic convolution}, namely 
	$$W_{A}(t):= \int_0^t e^{(t-s)A}{Q}^{\frac{1}{2}}\,dW(s),\qquad t \geq 0.$$	
	The stochastic convolution is differentiable and, for each $t>0$,  $W_A(t)$ belongs to $H^1(\T; \R)$, see Appendix \ref{appendix: proofs of sec 7}. 
	We will denote by  $\{\cP_t\}_{ \geq 0}$ the semigroup associated with the evolution \eqref{spde Q}, namely 
	\begin{equation}\label{semigroup non linear spde}
		\left ( \cP_{t}\psi \right )(u_0):=\mathbb{E} \left( \psi(u(t;u_0))\right),\, u_0 \in L^2(\T;\R),\,\psi \in \mathcal{B}_b(L^2(\T;\R);\R),
	\end{equation}
	where $\mathcal{B}_b(L^2(\T;\R);\R)$ is the class of real-valued bounded Borel measurable functions on $L^2(\T;\R)$ and $u(t;u_0)$ denotes the mild solution to \eqref{spde Q} with initial datum $u_0$. We will use such a notation every time we want to emphasize the dependence of the solution on the initial datum.

	Finally,  we denote by $\{ e^{tA} \}_{t \geq 0}$ the heat semigroup on $\T$ which acts on $f \in L^2(\T;\R)$ as follows
	\begin{equation}
		e^{tA}f=\sum \limits_{k \in \Z} e^{-tk^2}f_k e_k,\quad t \geq 0,\, f \in L^2(\T;\R).\notag
	\end{equation}
	Equivalently, $e^{tA}f$, $t \geq 0$, can be expressed as the convolution with respect to the space variable on $\T$ between the periodic heat kernel $G_{t}^{per}$ (defined in Appendix \ref{estimate heat kernel}, equation \eqref{heat kernel on the torus}) and $f$, i.e.
	\begin{equation}\label{heat semigroup via convoluzione}
		e^{tA}f=G_t^{per} * f,\quad  t \geq 0.
	\end{equation}

	\subsection{Statement of Main Results}
	As explained in the introduction, in the first part of this paper we study  the number of stationary solutions of the PDE \eqref{PDEsimpleparticlesystem}, when $V$ and $F$ are as in \eqref{doppia stella}. The  main result of this first part is  Theorem \ref{thm:mainthm3-inv-meas} below. In order to state and explain this result, let us start by recalling that stationary solutions of the PDE \eqref{PDEsimpleparticlesystem} can be characterised as solutions of an appropriate fixed point problem; namely, the following holds. 
	\begin{lemma} \label{stationary proposition}
		Consider the stationary problem associated to the evolution \eqref{PDEsimpleparticlesystem}, i.e.
		\begin{align}\label{stationary equation}
		 \sigma \pa_{xx} \rho(x) +\pa_x \left[\left(V'(x) + (F'\ast \rho)(x)\right) \rho(x)\right]=0 \,,   
		\end{align}
		with $V,F$ any two functions in $C^{\infty}(\T;\R)$. If $\rho \in \mathcal{P}_{ac}(\T) \cap H^1(\T;\R)$ is a weak solution to \eqref{stationary equation}, then $\rho$ is smooth, i.e. $\rho \in \cP_{ac}(\T) \cap C^{\infty}(\T;\R)$ and solves the following fixed point equation
		\begin{equation}\label{fixed point}
			\rho(x)= \frac{1}{Z_{\sigma}} e^{-\frac{1}{\sigma} \big(V(x)+F\ast \rho(x)\big)},\quad x \in \T,\,
		\end{equation}
		where $Z_{\sigma}$ is the normalization constant so that $\|\rho\|_{L^1(\T;\R)}=1$.
		Conversely, any probability measure whose density satisfies \eqref{fixed point} is smooth and it is a solution to \eqref{stationary equation}.
	\end{lemma}
	The proof of the above lemma is standard, but we could not find it in the literature for the exact setup we are considering here, so we briefly sketch it in Appendix \ref{uniqueness sigmageq1}.

	We are interested in solutions of the problem \eqref{stationary equation} when $V$ and $F$ are as in \eqref{doppia stella}, i.e. in solutions of the following problem
	\begin{equation}\label{PDEFV}
		\pa_{xx}\rho(x)+\pa_x \left[ \left( -2\sin(2x)+ \int_{\T} \sin(x-y)\rho(y)\,dy  \right) \rho(x)\right ]=0, \quad x \in \T,  
	\end{equation}
	so we  further specify the fixed point equation \eqref{fixed point} for such choices of $V$ and $F$. To this end, by  the addition formula for the cosine we have
	\begin{equation}\label{f2par}
		\begin{split}
			(F \ast \rho)(x) = -\int_{\T} \cos(x-y) \rho(y)\,dy  = -m_1 \cos x -m_2 \sin x ,     
		\end{split}
	\end{equation}
	having set 
	\begin{equation} \label{constant M1 and M2}
		m_1:=\int_{\T} \cos y  \rho(y)\,dy\, , \quad
		m_2:=\int_{\T} \sin y \rho(y)\,dy.
	\end{equation}
	Hence, for our choice of $V$ and $F$ we can  rewrite the fixed point problem \eqref{fixed point} as follows 
	\begin{equation} \label{fixed point equation}
		\rho_{m_1,m_2}(x)= \frac{1}{Z_{\sigma}(m_1,m_2)} e^{-\frac{1}{\sigma} \left(\cos(2x)-m_1\cos x -m_2\sin x   \right)},\quad x \in \T,   
	\end{equation}
	where 
	$$Z_{\sigma}(m_1,m_2):=\int_{\T} e^{-\frac{1}{\sigma} \left(\cos(2x)- m_1 \cos x-m_2 \sin x  \right)}\,dx$$ 
	is the normalization constant. By multiplying both sides of \eqref{fixed point equation} by $\cos x$ ($\sin x$, respectively), we obtain  
	\begin{eqnarray} \label{m1}
		&& m_1=\int_{\T} \frac{\cos x}{Z_{\sigma}(m_1,m_2)} e^{-\frac{1}{\sigma} \left(\cos(2x)-m_1\cos x-m_2\sin x  \right)}\,dx, \\
		&& m_2=\int_{\T}\frac{\sin x}{Z_{\sigma}(m_1,m_2)} e^{-\frac{1}{\sigma} \left(\cos(2x)-m_1\cos x-m_2\sin x  \right)}\,dx. \label{m2}
	\end{eqnarray}
	Then, if we define the map 
	\begin{equation} \label{fixed point map}
		\begin{split}
			g_{\sigma}:\R^2 & \to \R^2 \\ 
			(m_1,m_2) & \to g_{\sigma}(m_1,m_2):=(g_{1,\sigma}(m_1,m_2),g_{2,\sigma}(m_1,m_2))    ,
		\end{split}
	\end{equation}
	where 
	\begin{eqnarray} \label{g1}
		&& g_{1,\sigma}(m_1,m_2):= \int_{\T} \frac{\cos x}{Z_{\sigma}(m_1,m_2)} e^{-\frac{1}{\sigma} \left(\cos(2x)-m_1\cos x-m_2\sin x  \right)}\,dx,  \\
		&& g_{2,\sigma}(m_1,m_2):= \int_{\T}\frac{\sin x}{Z_{\sigma}(m_1,m_2)} e^{-\frac{1}{\sigma} \left(\cos(2x)-m_1\cos x-m_2\sin x  \right)}\,dx, \label{g2}
	\end{eqnarray}
	it follows that the density $\rho_{m_1,m_2}$ is stationary solution of \eqref{PDEFV} if and only if $(m_1,m_2)$ is a fixed point of the map $g_{\sigma}$.  We have therefore reformulated the problem of finding solutions of equation \eqref{PDEFV}  as a two parameter problem of finding fixed points of the map $g_{\sigma}$. This is the basic approach that allows one to prove the following result.

	\begin{theorem}\label{thm:mainthm3-inv-meas}
		There exists $\sigma_c>0$ such that if $\sigma$ is either sufficiently small (i.e. $\sigma \ll 1$) or $ \frac{1}{2} \leq \sigma < \sigma_c$ the stationary problem \eqref{PDEFV} has exactly three solutions, while for $\sigma>\sigma_c$ it has exactly one solution.
		Furthermore, the density
		\begin{equation}\label{centred stat sol}
			\rho_{0,0}(x)=\frac{1}{Z_{\sigma}(0,0)}e^{-\frac{\cos(2x)}{\sigma}},\quad x \in \T , 
		\end{equation}
		is always a solution of \eqref{PDEFV}, irrespective of the value of $\sigma$.  The two additional stationary solutions for $\sigma \ll 1$ and  $ \frac{1}{2} \leq \sigma < \sigma_c$ are centered around the minima of the double well potential $V$, i.e. around $x=\pm \frac{\pi}{2}$. Finally,  the critical value $\sigma_c$ can be explicitly characterized as the (unique) zero of the  function $f_c:(0,+\infty) \to \R$ defined as 
		\begin{equation}\label{function f}
			f_c(\sigma)=\frac{1}{\sigma} -2 +\frac{1}{\sigma}r_0 \left ( \frac{1}{\sigma} \right ),\quad\sigma >0, 
		\end{equation}
		where $r_0$ is the modified Bessel function, see \eqref{Besselr} for a definition. 
		Analytical computations show that  $\sigma_c \simeq 0.7709$.
	\end{theorem}
 The proof of Theorem \ref{thm:mainthm3-inv-meas} can be found in Section \ref{sec:sec4}.
	\begin{remark} Results in the spirit of the above theorem have been known for a long time, see e.g. \cite{Dawson, constantin,hermann, tugaut, Dawson, pavliotis, GiacominPakPel, DuTu} and references therein and indeed  parts of the (long) proof of the above theorem is inspired by \cite{tugaut, pavliotis}. More precisely, 
		\begin{itemize}
			\item As we mentioned in the introduction,  the setup of Theorem \ref{thm:mainthm3-inv-meas} can be seen  as being a non-rotationally invariant modification of the Kuramoto model considered in \cite{pavliotis, GiacominPakPel, BertiniGiacominPoquet} (or of the Smoluchowski equation of \cite{constantin,constantin2}). Because of this break of rotational invariance, most of the arguments used in   \cite{pavliotis, GiacominPakPel, BertiniGiacominPoquet, constantin, constantin2} cannot be adapted in the current setup. 
		\item Other results similar to Theorem \ref{thm:mainthm3-inv-meas} are those in \cite{hermann, Dawson}: besides the fact that the state space considered in  \cite{hermann, Dawson} is $\R$ as opposed to the torus,  also in \cite{hermann, Dawson} the authors consider a double -well potential and an attractive inter-particle force and they reduce the problem of finding stationary solutions of the PDE they consider to a fixed point problem; however, because of the exact analytic form of the attractive force they consider, they can reduce the fixed point problem \eqref{fixed point} to a one-parameter fixed point problem. In our case, because of our choice of $F$ \eqref{doppia stella}, we end up with a two-parameter fixed point problem (see \eqref{f2par}), which cannot be a priori further reduced to a one parameter problem, again because of lack of rotational symmetry. 
		\item We further conjecture, based on numerical evidence,  that \eqref{PDEsimpleparticlesystem} has exactly three stationary solutions for any $0<\sigma<\sigma_c$, but we have been able to prove it only for $\sigma \ll 1$ and $\frac12 \leq \sigma<\sigma_c$.
		\item Similar results could be obtained for multi-well potentials (i.e. by considering $V(x)=\cos (mx)$). This is done (for $m=4$) in \cite{Martinthesis}. 
		\end{itemize}
	\end{remark}
	
	\begin{comment}
	\emph{In particular, the setup of theorem \ref{thm:mainthm3-inv-meas} can be seen as being either the analogous on the torus of the interacting particle system and PDE in $\R$ considered in \cite{hermann} {\color{red}Julien: if we say this then we should also explain why we need to redo the proofs and why these proofs are not a straightforward readaptation of the proofs in $\R$. Well...they partly are}, or, as we mentioned in the introduction, as being a non-rotationally invariant modification of the  Kuramoto model considered in \cite{pavliotis, GiacominPakPel, BertiniGiacominPoquet}. With respect to all these works, we are much more {\color{red} Julien: is this an overstatement?} {\color{blue} A bit of an overstatement: some of these works like \cite{tugaut,pavliotis} are precise about $\sigma_c$, so I am not sure if insisting on this... A proposition below, emphasizing another difference of this theorem with most of the literature; the lack of translation invariance actually creates difficulties with respect to the literature.} precise regarding the critical value $\sigma_c$ where the bifurcation occurs (the typical result being that for $\sigma$ small enough there are multiple stationary solutions and for $\sigma$ large enough there is only one).  }  
	\end{comment}

	Let us now move on to the second part of the paper, where we study the SPDE \eqref{spde Q}. We clarify that, when studying \eqref{spde Q}, we always consider $V,F$ to be arbitrary coefficients in $C^{\infty}(\T;\R)$, i.e. we no longer restrict to the choice \eqref{doppia stella}. We first state the main results of part two, Theorem \ref{mild solution global existence} and Theorem \ref{theorem:existence and uniqueness inv measure} below, and then comment on them. 
	
	\begin{theorem}\label{mild solution global existence}
		Let $Q$ satisfy the standing assumption \eqref{Qactsonbasis}-\eqref{eqn:Qtrace-class} and let $V, F \in C^{\infty}(\T;\R)$. Then,  for any $T>0$ (independent of $\omega \in \Omega$) and initial datum $u_0 \in L^2(\T;\R)$ there exists a unique mild solution (in the sense clarified in Subsection \ref{subsec:notation}) $u$ to equation \eqref{spde Q}; such a solution is  $\mP$-a.s.  in  $C \left( [0,T];L^2(\T;\R) \right )$.
	\end{theorem}
	
	The proof of Theorem \ref{mild solution global existence} is in Section \ref{Well-Posedness McKean_Vlasov}.

	\begin{Assumption}\label{ass:QstrongFeller}
		The covariance operator  $Q$ satisfies the standing assumption \eqref{Qactsonbasis}-\eqref{eqn:Qtrace-class} and, moreover, the  eigenvalues $\{\lambda_k^2\}_{k \in \mathbb Z}$ of $Q$ satisfy the following growth condition: there exist $c>0$ and  {$\gamma<1$} such that
		\begin{equation}\label{cond_lambda}
			\lambda_k^2\geq c(1+k^2)^{-\gamma}, \quad  k\in\Z.
			%,\,  \gamma<1.  
		\end{equation}
	\end{Assumption}

	\begin{theorem}\label{theorem:existence and uniqueness inv measure}
		Let $ V$ and  $F$ be as in the statement of Theorem \ref{mild solution global existence} and let $\mathcal P_t$ be the semigroup associated with the SMKV equation \eqref{spde Q}, see \eqref{semigroup non linear spde}. 
		\begin{description}
		    \item[i)]If $Q$ satisfies the standing assumption \eqref{Qactsonbasis}-\eqref{eqn:Qtrace-class} then the semigroup is irreducible.
		    \item[ii)] If, furthermore, $Q$ satisfies Assumption \ref{ass:QstrongFeller}, then $\mathcal P_t$ is Strong Feller as well, hence the dynamics \eqref{spde Q} admits at most one invariant measure. 
		\end{description} 
	\end{theorem}
	
	The proof of part i) can be found in Section \ref{sec: irreducibility}, the proof of ii) in Section \ref{sec:sec7}.

	\begin{remark}\label{note:wellposthm} Some comments  on Theorem \ref{mild solution global existence} and Theorem \ref{theorem:existence and uniqueness inv measure}. 
		\begin{itemize}
			\item The scheme of proof of Theorem \ref{mild solution global existence}   uses the  trick of reducing the SPDE at hand to a PDE with random coefficients by ``subtracting the noise". In particular, we mostly use a combination of the arguments of e.g. \cite{burgers} developed for the stochastic Burgers equation and of those developed in the  PDE literature for the evolution \eqref{PDEsimpleparticlesystem}, with particular reference to \cite{pav-car}.  Let us note that, while the PDE \eqref{PDEsimpleparticlesystem} (which is in gradient form) preserves total mass and positivity,  the SPDE \eqref{spde Q} does not enjoy any of these properties; this is one of the main reasons why in this paper the PDE arguments typically used for \eqref{PDEsimpleparticlesystem} (which heavily rely on such properties) could be  used  very sparingly and we rely instead on the methods for Stochastic Burgers'. 
			 
			  The similarity between the nonlinearity we consider and the Burgers' nonlinearity may become apparent once we observe that, taking the convolution with $F'$  (which is a smooth function on the torus, hence bounded),  the nonlinear term $\pa_x ((F'\ast u (x)) u(x))$ is, substantially, the derivative of an expression which is quadratic in the unknown $u$.  	 
			\item As mentioned in the introduction, the works \cite{vlasovwellposedness, {vlasovwellposedness2}} deal with strong well-posedness of infinite-dimensional McKean-Vlasov SDEs (via Galerkin approximation). Such equations  are different from \eqref{spde Q}, as in  \cite{vlasovwellposedness, {vlasovwellposedness2}} the solution of the process depends on the law of the process itself, whereas this is not the case for \eqref{spde Q}. In particular the evolution  \eqref{spde Q} enjoys both the Markov Property and the Markov family property (see \cite{Lacker}), while the evolutions in \cite{vlasovwellposedness, {vlasovwellposedness2}} satisfy the Markov property only. 
			Moreover \cite{vlasovwellposedness} contains some very nice results on averaging for McKean-Vlasov type SPDEs; producing such a result requires studying the ergodic properties of a linearization of \eqref{spde Q} (namely, in our context, the equation obtained from \eqref{spde Q} by replacing $F'\ast u$ with $F'\ast \nu$, for a given $\nu$), but not of \eqref{spde Q} itself. So Theorem \ref{theorem:existence and uniqueness inv measure} above constitutes, to the best of our knowledge, a first attempt at a partial description of the ergodic behaviour of \eqref{spde Q}. 
			\item The `basic' method to prove the Strong-Feller property requires the covariance operator $Q$ to have bounded inverse, see \cite{peszat, manca} and references therein.   Our proof of the Strong-Feller property relies on the use of Bismut-Elworthy-Li type of formulas (see Section \ref{sec:sec7}, formula \eqref{Zeta} and Note \ref{nota doppio pallino} in particular), where the inverse of the operator $Q$ appears. While we do not require $Q$ to have bounded inverse, we still need to control the growth of $Q^{-1}$, hence Assumption \ref{ass:QstrongFeller}. In other words, we need to require that the eigenvalues of $Q$ do not decay too fast, i.e. that the noise is `strong enough'. This is not unexpected, see e.g. \cite{cerrai} (we will make more comparisons with \cite{cerrai} in Section \ref{sec:sec7}).  It is worth mentioning that by using the same technique of Section \ref{sec:sec7} a strong Feller result can be obtained in the weaker noise setting as well, at the cost of changing space. Namely, it can be shown that when in Assumption \ref{ass:QstrongFeller} we take $\gamma \geq 1$ then the semigroup generated by \eqref{spde Q} is strong Feller in the fractional Sobolev space $H^{\gamma}(\T;\R)$(see \cite{Oh} for its definition), but we don't do this here for brevity.  
			\item The method of proof we use to show the Strong-Feller property relies on proving first that such a property holds for a class of equations with Lipshitz non-linearity, where the nonlinearity depends on the derivative of the solution (and for this reason it is different from the nonlinearities considered in \cite{cerrai, manca}); this general result is contained in Subsection \ref{sec: sec 7.2}. 
			\item In this paper we do not cover the purely cylindrical noise case, i.e. $Q=I$, $I$ being the identity operator. While our proofs of irreducibility and Strong Feller property would still hold if $Q=I$, our well-posedness proof requires some smoothness. We leave this further extension for future work. 
			\item Finally and most importantly, in this paper we do not study the {\em existence} of the invariant measure for \eqref{spde Q}.  It turns out that this is not a straightforward task so,  to contain the length of this paper,  we will do so in separate  forthcoming work. \end{itemize}
	\end{remark}

	\section{Proof of Theorem \ref{thm:mainthm3-inv-meas}}\label{sec:sec4}
	
	In this section we prove Theorem \ref{thm:mainthm3-inv-meas}. The main argument of the proof is described below and it is divided into four steps. The first three steps are proved in Subsection \ref{dove_sono} to Subsection \ref{Exactly three invmeas}, respectively, step four is proved in Appendix \ref{uniqueness sigmageq1} (as the ideas are completely analogous to those in Subsection \ref{More invmeas}). The proof that the critical value $\sigma_c$ is a zero of the function $f_c$ \eqref{function f} and the consequent approximation for the numerical value of $\sigma_c$ can be found in the proof of Theorem \ref{teorema segno zeta}.

	\begin{proof}[Proof of Theorem \ref{thm:mainthm3-inv-meas}]
		From \eqref{fixed point map}, \eqref{g1}  and \eqref{g2} it follows that $g_{\sigma}(0,0)=(0,0)$. Hence, $(m_1,m_2)=(0,0)$ is a fixed point of $g_{\sigma}$ for any $\sigma$, so that \eqref{centred stat sol} is a stationary solution to \eqref{PDEsimpleparticlesystem} for every $\sigma >0$ (note that $Z_{\sigma}(0,0)$ coincides exactly with the normalization constant $Z_{\sigma}$ appearing in \eqref{fixed point}).\\
		Second, exploiting the symmetries of $\sin$ and $\cos$ functions, it is also 
		easy to see that the axes $\mM_1:= \left \{(m_1,0) \right \}_{m_1 \in \R}$ and $\mM_2:= \left \{ (0,m_2) \right \}_{m_2 \in \R}$ are invariant for $g_\sigma$ i.e. $g_{\sigma}(m_1,0) \in \mM_1$ and $g_{\sigma}(0,m_2) \in \mM_2$ for all $m_1,m_2 \in \R$.
		
		We now divide the proof into four steps. 
		In Step 1) we show that all fixed points of $g_\sigma$ lie on either $\mM_1$ or $\mM_2$. This 
		allows us to reduce the search of fixed points to two one-parameter fixed point problems, one 
		on $\mM_1$ and the other on $\mM_2$. In Step 2) we consider the fixed point problem on $\mM_2$
		and show that there exists $\sigma_c \in(0,1)$ such that for $0<\sigma <\sigma_c$ there are 
		exactly two additional stationary states other than \eqref{centred stat sol}, while, for 
		$\sigma > \sigma_c$ those two additional solutions collapse into \eqref{centred stat sol} which lies on $\mM_2$, so 
		for $\sigma >\sigma_c$ there is exactly one steady state of \eqref{PDEFV} which lies on 
		$\mM_2$.
		At this point what one would want to show is that the only fixed point on $\mM_1$ is (0,0), 
		irrespective of the value of $\sigma$. However, we are able to prove this fact only for 
		$0< \sigma \ll 1$ (Step 3) and for $ \sigma \geq \frac{1}{2}$ (Step 4). Hence, from steps 1 to 
		4, we deduce the uniqueness for $\sigma >\sigma_c$ and the existence of exactly three stationary 
		states for $\frac{1}{2} \leq \sigma < \sigma_c $ and $0 < \sigma \ll 1$. 
		
		Steps 1, 2 and 3 are proven in Subsections \ref{dove_sono}, \ref{More invmeas}, \ref{Exactly three invmeas}, respectively, and the details for Step 4 are referred to Subsection \ref{uniqueness h} of Appendix \ref{uniqueness sigmageq1}.
	\end{proof}
	
	\subsection{Proof of Step 1}\label{dove_sono}
	In this section we prove that the fixed points of $g_{\sigma}$ belong to either the axis $\mM_1$ or the axis $\mM_2$. More precisely, the following statement holds.
	\begin{proposition} \label{dove sono i punti fissi}
		Let $(m_1,m_2) \in \R^2$ be a fixed point of the map $g_{\sigma}$ different from (0,0); then either $m_1=0$ or $m_2=0$.
	\end{proposition}
	\begin{proof}
		If $(m_1,m_2) \in \R^2$ is a fixed point of $g_{\sigma}$ then we know $(m_1,m_2)$ satisfies the fixed point equations \eqref{m1} and \eqref{m2}.
		Next, we write $(m_1,m_2) \in \R^2$ in polar coordinates, namely we let $(M,\varphi) \in \R_{+} \times [0,2\pi)$; that is,  
		\[
		M=(m_1^2+m_2^2)^{1/2},~m_1=M\cos\varphi,~m_2=M\sin\varphi,\qquad \varphi \in [0,2\pi),  M \in \R_{+}.
		\]
		From the above and from the addition formula for the cosine we can rewrite the fixed point equations \eqref{m1} and \eqref{m2} as 
		\begin{eqnarray}
			m_1 = \frac{1}{Z_{\sigma}(m_1,m_2)}\int_{\T}\cos x\, e^{-\frac{1}{\sigma}(\cos(2x)-M\cos(x-\varphi))}\,dx , \label{polar coordinates fixed point 1} \\
			m_2 = \frac{1}{Z_{\sigma}(m_1,m_2)} \int_{\T}\sin x \,e^{-\frac{1}{\sigma}(\cos(2x)-M\cos(x-\varphi))}\,dx, \label{polar coordinates fixed point 2}
		\end{eqnarray}
		respectively.
		Define  
		\begin{equation}\label{eq:integral}
			I(M,\varphi):=\int_{\T} \sin(x-\varphi) e^{-\frac{1}{\sigma}(\cos(2x)-M\cos(x-\varphi))}\,dx.
		\end{equation}
		If $(m_1,m_2)$ is a fixed point of $g_{\sigma}$ then it follows that $I(M,\varphi)=0$. Indeed, from the addition formula for the sine and from \eqref{polar coordinates fixed point 1} and \eqref{polar coordinates fixed point 2} we have 
		\begin{align}
			I(M,\varphi)& =\cos \varphi \int_{0}^{2\pi} \sin x \, e^{-\frac{1}{\sigma}(\cos(2x)-M\cos(x-\varphi)}\,dx \notag\\
			&-\sin \varphi  \int_{0}^{2\pi} \cos x \,e^{-\frac{1}{\sigma}(\cos(2x)-M\cos(x-\varphi)}\,dx  \notag\\
			& = Z_{\sigma}(m_1,m_2)\left( m_2\cos \varphi-m_1 \sin \varphi \right) \notag \\
			& =Z_{\sigma}(m_1,m_2)M \left(\sin \varphi \cos \varphi-\cos \varphi \sin \varphi \right)=0.\notag
		\end{align}
		Our goal is to show that the equation $I(M,\varphi)=0$ implies either $M=0$ or $\varphi\in \left\{0,\pi/2,\pi,3\pi/2\right\}$.
		By applying the change of variable to $x \to x-\varphi,\,x \in \T$ and using periodicity we obtain
		\begin{equation}
			I(M,\varphi)=\int_{-\pi}^{\pi} \sin x \,e^{-\frac{1}{\sigma}\cos(2x+2\varphi)}e^{\frac{1}{\sigma}M\cos x}\,dx. \notag
		\end{equation}
		We split the above integral into four parts: $I_{-1}(M,\varphi)$, $I_{0}(M,\varphi)$, $I_1(M,\varphi)$ and $I_2(M,\varphi)$, namely 
		\begin{equation}
			I(M,\varphi)=\sum \limits_{i=-1}^2 I_i(M,\varphi):=\sum \limits_{i=-1}^2 \quad\, \int_{(i-1)\frac{\pi}{2}}^{i\frac{\pi}{2}} \sin x\, e^{-\frac{1}{\sigma}\cos(2x+2\varphi)}\,e^{\frac{1}{\sigma} M \cos x}\,dx. \notag
		\end{equation}
		We now apply respectively the changes of variable $x \to x+\pi$, $x \to -x$ and $x \to \pi-x$ to $I_{-1}(M,\varphi),I_0(M,\varphi)$ and $I_2(M,\varphi)$, respectively, and then add together the resulting four expressions we obtain the following expression arriving at 
		\begin{align}\notag
			 I(M,\varphi)  =\int_0^{\frac{\pi}{2}} &\sin x  \bigg[
			-e^{-\frac{1}{\sigma}\cos(2x+2\varphi)}e^{-\frac{1}{\sigma}M\cos x}
			-e^{-\frac{1}{\sigma}\cos(2x-2\varphi)}e^{\frac{1}{\sigma}M\cos x} \\
			&
			e^{-\frac{1}{\sigma}\cos(2x+2\varphi)}e^{\frac{1}{\sigma} M\cos x}+e^{-\frac{1}{\sigma}\cos(2x-2\varphi)}e^{-\frac{1}{\sigma}M\cos x}  
			\bigg] \,dx \notag \\
			 =\int_{0}^{\frac{\pi}{2}}&\sin x \left[e^{-\frac{1}{\sigma}\cos(2x+2\varphi)}-e^{-\frac{1}{\sigma}\cos(2x-2\varphi)}\right] \left[e^{\frac{1}{\sigma}M\cos x}-e^{-\frac{1}{\sigma}M\cos x}\right]\,dx .\label{eq: rearranged}
		\end{align}
		Our first aim is to prove that $I(M, \varphi)=0$ cannot hold when $M>0$ and $\varphi \in \left ( 0, \frac{\pi}{2} \right )$. As a consequence we deduce that there are no fixed points of the map $g_{\sigma}$ in the first quadrant of $\R^2$ (i.e. when $m_1>0$ and $m_2>0$). We will then repeat the procedure on the other quadrants, in turn. In the first place, we note
		\begin{equation} \label{inequality cosine}
			\cos(2x+2\varphi)<\cos(2x-2\varphi),\quad\text{for all $x,\varphi \in \left (0,\frac{\pi}{2}\right )$}.
		\end{equation}
		Indeed, from the addition formula for the cosine, equation \eqref{inequality cosine} is equivalent to 
		\begin{equation}
			\cos(2x)\cos(2\varphi)-\sin(2x)\sin(2\varphi) < \cos(2x)\cos(2\varphi)+\sin(2x)\sin(2\varphi),\notag
		\end{equation}
		which is in turn equivalent to
		\begin{equation}\label{inequality sine}
			\sin(2x) \sin(2\varphi)>0.
		\end{equation}
		Clearly, \eqref{inequality sine} holds for all $x,\varphi \in \left (0,\frac{\pi}{2}\right )$.	Hence,  
		\begin{equation}
			e^{-\frac{1}{\sigma}\cos(2x+2\varphi)}>e^{-\frac{1}{\sigma}\cos(2x-2\varphi)},\quad\text{for all $x,\varphi \in \left (0,\frac{\pi}{2}\right ).$}
		\end{equation}
		Furthermore, since 
		$$\cos x > 0 > -\cos x,\quad x \in \left (0,\frac{\pi}{2}\right ),$$
		we deduce
		\begin{equation}
			e^{\frac{1}{\sigma}M\cos x}>1>e^{-\frac{1}{\sigma}M\cos x},\quad\text{for all $x \in \left (0,\frac{\pi}{2}\right )$}.
		\end{equation}
		Using \eqref{eq: rearranged} we then conclude that $I(M,\varphi)>0$ for all $\varphi \in \left (0,\frac{\pi}{2}\right )$ and $M>0$.\\
		Similarly, if we let $M>0$, $\varphi \in \left (\frac{\pi}{2},\pi \right )$ then 
		\begin{equation} \label{inequality cosine 2}
			\cos(2x+2\varphi)>\cos(2x-2\varphi),\quad\text{for all $x,\varphi \in \left(\frac{\pi}{2},\pi \right )$},
		\end{equation}
		which is equivalent to 
		\begin{equation}\label{inequality sine 2}
			\sin(2x) \sin(2\varphi)<0.
		\end{equation}
		The above holds for all $x \in (0,\frac{\pi}{2})$ and for all $\varphi \in (\frac{\pi}{2},\pi)$.
		Hence, 
		\begin{equation}
			e^{-\frac{1}{\sigma}\cos(2x+2\varphi)}<e^{-\frac{1}{\sigma}\cos(2x-2\varphi)},\quad\text{for all $x,\varphi \in \left (\frac{\pi}{2},\pi \right ).$}
		\end{equation}
		\begin{equation}
			e^{\frac{1}{\sigma}M\cos x}>1>e^{-\frac{1}{\sigma}M \cos x},\quad\text{for all $x \in \left (0,\frac{\pi}{2} \right )$},
		\end{equation}
		we have that $I(M,\varphi)<0$ for all $\varphi \in \left (\frac{\pi}{2},\pi \right )$ and $M>0$. The remaining cases i.e. $M>0$, $\varphi \in \left (\pi,\frac{3\pi}{2} \right )$ and $M>0$, $\varphi \in \left (\frac{3\pi}{2},2\pi \right )$ can be dealt with analogously. 
	\end{proof}
	
	\subsection{Proof of Step 2}\label{More invmeas}
	We recall that the axis $\mM_2$ is invariant for $g_{\sigma}$, hence we can define the map $\bar{g}_{\sigma}:\R \to \R$ to be the restriction of $g_{\sigma}:\R^2 \to \R^2$ to $\mM_2$; namely
	\begin{equation} \label{one dimension fixed point map}
		\bar{g}_{\sigma}(m):= \frac{1}{Z_{\sigma}(m)} \int_{\T} \sin x \,e^{-\frac{1}{\sigma} \left(\cos(2x)-m\sin x \right)}\,dx,\,\,\, Z_{\sigma}(m):=\int_{\T} e^{-\frac{1}{\sigma}\left(\cos(2x)-m\sin x \right)}\,dx. 
	\end{equation}
	Then $m$ is a fixed point of $\bar{g}_{\sigma}$ if and only if $(0,m)$ is a fixed point of $g_{\sigma}$. 
	In what follows we prove that the map \eqref{one dimension fixed point map} has the following property: there exists $\sigma_c>0$ such that, when $\sigma> \sigma_c$, $\bar{g}_{\sigma}$ admits a unique fixed point, while, when $\sigma < \sigma_c$, $\bar{g}_{\sigma}$ admits exactly three fixed points.
	The idea of the proof is inspired by \cite[Theorem 2.1]{tugaut}. Thus, let us introduce the map $\zeta_{\sigma}:\R \to \R$ defined as 
	\begin{equation} \label{zeta}
		\zeta_{\sigma}(m):= \int_{\T} (\sin x-m)e^{-\frac{1}{\sigma}\cos(2x)+\frac{m}{\sigma}\sin x}\,dx,\quad m \in \R,
	\end{equation}
	and note that $m \in \R$ is a fixed point of $\bar{g}_{\sigma}$ if and only if $m$ is a zero of $\zeta_{\sigma}$ i.e. $\zeta_{\sigma}(m)=0$. {Furthermore, $\zeta_{\sigma}(1)<0$.} The statement is then a consequence of Proposition \ref{uniqueness related to first derivative} and Theorem \ref{teorema segno zeta} below, which we first state and then prove in turn. 
	
	\begin{proposition}\label{uniqueness related to first derivative} 
		
		The map $\zeta_{\sigma}:\R \to \R$ defined in \eqref{zeta} is odd and it admits either exactly one zero or exactly three zeroes, depending on the value of $\sigma$. Furthermore, the following holds:
		\begin{itemize}
			\item If $\zeta_{\sigma}^{'}(0) < 0$ then $\zeta_{\sigma}$ is strictly decreasing on $[0,+\infty)$. Since $\zeta_{\sigma}(0)=0$ then we deduce that $\zeta_{\sigma}$ does not vanish on $(0,+\infty)$. The function $\zeta_{\sigma}$ is odd so that we can easily deduce  that it admits a unique zero on $\R$.
			\item If $\zeta_{\sigma}^{'}(0)> 0$ then there exists $\bar{m}>0$ such that $\zeta_{\sigma}$ is strictly increasing in $[0,\bar{m})$ and then strictly decreasing for $m \geq \bar{m}$. Since $\zeta_{\sigma}(0)=0$ {and $\zeta_{\sigma}(1)<0$}, we deduce that $\zeta_{\sigma}$ has a unique zero on $(0,\textcolor{blue}{1})$. The function $\zeta_{\sigma}$ is odd, hence it admits exactly three zeroes on $\R$.
		\end{itemize}
	\end{proposition}

	\begin{theorem}\label{teorema segno zeta}
		There exists $\sigma_c >0$ such that the following holds
		\begin{equation}\label{segno zeta preciso}
			\begin{cases}
				& \zeta_{\sigma}^{'}(0) < 0,\quad \sigma > \sigma_c, \\
				& \zeta_{\sigma}^{'}(0) > 0,\quad \sigma < \sigma_c. 
			\end{cases}
		\end{equation}
		In particular, due to Proposition \ref{uniqueness related to first derivative}, we have the following:
		\begin{itemize}
			\item If $\sigma > \sigma_c$ then $\zeta_{\sigma}$ admits a unique zero ($m=0$).
			\item If $\sigma < \sigma_c$ then $\zeta_{\sigma}$ admits exactly three zeroes (one of which is $m=0$).
		\end{itemize}
		Furthermore, an analytical approximation of $\sigma_c$ is given by $\sigma_c \simeq 0.7709$.
	\end{theorem}
	Before proving the above statements we state and prove some technical lemmata, which will be needed in the proofs of the above main results. More precisely, we first prove Lemma \ref{series expansion} and Lemma \ref{decrescente} after that we move on to proving Proposition \ref{uniqueness related to first derivative} and Theorem \ref{teorema segno zeta}.
	To this end, let us introduce the sequence $\{s_k\}_{k \in \N}\subset \R$ defined as 
	\begin{equation}\label{sk}
		s_k:=\int_{\T} (\sin x )^k e^{-\frac{1}{\sigma}\cos(2x)}\,dx,\quad k \in \N.
	\end{equation}
	Since $\sin \, x$ is an odd function, $\cos(2x)$ is an even function and $ (\sin \, x)^2 < 1$ for all $x \in \T \setminus \left \{ -\frac{\pi}{2}, \frac{\pi}{2} \right \} $, the following properties hold for $\{s_k\}_{k \in \N}$:
	\begin{eqnarray}
		&& s_{2k+1}=0,\quad\text{for all $k \in \N$}, \label{zero on odd numbers} \\
		&& s_{2k+2} < s_{2k},\quad\text{for all $k \in \N$} \label{decreases on even numbers}.
	\end{eqnarray}
	\begin{lemma}\label{series expansion}
		For every $\sigma>0$, the function $\zeta_{\sigma}:\R \to \R$ admits the following series expansion 
		\begin{equation}\label{series expansion zeta}
			\zeta_{\sigma}(m)= \sum \limits_{k \in \N} \frac{1}{(2k)!} \left( \frac{m}{\sigma}\right )^{2k+1} s_{2k}
			\Upsilon_k(\sigma),\quad m \in \R,
		\end{equation}
		where $\{\Upsilon_k(\sigma)\}_{k \in \N}$ is the sequence defined as 
		\begin{equation}\label{gamma}
			\Upsilon_k(\sigma):=\frac{s_{2k+2}}{(2k+1)s_{2k}}-\sigma,\quad k \in \N,\, \sigma >0.
		\end{equation}
	\end{lemma}
	\begin{proof}
		From the power series expansion of $e^x$, $e^x=\sum \limits_{k \in \N} \frac{x^k}{k!},\quad x \in \R$, we obtain 
		\begin{equation}
			\begin{split}
				\zeta_{\sigma}(m) & =\int_{\T} (\sin \, x-m)e^{-\frac{1}{\sigma}\cos(2x)+\frac{m}{\sigma} \sin x }\,dx \\
				& =\int_{\T} (\sin x - m)e^{-\frac{1}{\sigma}\cos(2x)}\sum \limits_{k \in \N} \frac{1}{k!}\left( 
				\frac{m}{\sigma}\right )^k \left( \sin x  \right)^k \,dx \\
				&= \sum \limits_{k \in \N} \frac{1}{k!} \left( \frac{m}{\sigma}\right )^k \left( s_{k+1}-ms_k\right )=\sum 
				\limits_{k \in \N} \frac{1}{k!} \left( \frac{m}{\sigma}\right )^k s_{k+1}-\sum \limits_{k \in \N} 
				\frac{\sigma}{k!} \left( \frac{m}{\sigma}\right )^{k+1} s_k.\notag
			\end{split}
		\end{equation}
		From \eqref{zero on odd numbers} we then deduce 
		\begin{equation}
			\begin{split}
				\zeta_{\sigma}(m) & = \sum \limits_{k \in \N} \frac{1}{(2k+1)!} \left( \frac{m}{\sigma}\right )^{2k+1} s_{2k+2}-\sum \limits_{k \in \N} \frac{\sigma}{(2k)!} \left( \frac{m}{\sigma}\right )^{2k+1} s_{2k} \\
				& = \sum \limits_{k \in \N} \frac{1}{(2k)!} \left( \frac{m}{\sigma}\right )^{2k+1} s_{2k}\left( \frac{s_{2k+2}}{(2k+1)s_{2k}}-\sigma \right). \notag
			\end{split}
		\end{equation}
		By recalling the definition \eqref{gamma} we deduce the assertion.
	\end{proof}
	Moreover, the following property of the sequence $\{\Upsilon_k(\sigma)\}_{k \in \N}$ holds.
	\begin{lemma}\label{decrescente}
		For any $\sigma>0$ the sequence $\{\Upsilon_k(\sigma)\}_{k \in \N}$ is strictly decreasing (in $k$).
	\end{lemma}
	\begin{proof}
		From the definition of $\{\Upsilon_k(\sigma)\}_{k \in \N}$ (see \eqref{gamma}) it suffices to prove that the sequence $\left \{ \frac{s_{2k+2}}{(2k+1)s_{2k}} \right \}_{k \in \N}$ is decreasing. 
		Thus, from the definition of $s_{2k}$ and by integrating by parts we obtain 
		\begin{align} \notag
			s_{2k} & =  \int_{\T} \cos x \,\frac{d}{dx} (\sin x)^{2k-1}e^{-\frac{1}{\sigma}\cos(2x)}\,dx \\ \notag
			& = (2k-1)\int_{\T} \cos^2 x \,(\sin x)^{2k-2}e^{-\frac{1}{\sigma}\cos(2x)}\,dx \\ \notag
			& +\frac{4}{\sigma} \int_{\T} \cos^2 x \, (\sin x)^{2k} e^{-\frac{1}{\sigma}\cos(2x)}\,dx\\ \notag
			& = (2k-1)s_{2k-2}-(2k-1)s_{2k}+\frac{4}{\sigma}s_{2k}-\frac{4}{\sigma}s_{2k+2}. \notag 
		\end{align}
		Rearranging, we obtain 
		\begin{equation}
			\frac{(2k+1)s_{2k}}{s_{2k+2}}=\frac{(2k-1)s_{2k-2}}{s_{2k+2}}+\frac{s_{2k}}{s_{2k+2}}+\frac{4}{\sigma}\frac{s_{2k}}{s_{2k+2}}-\frac{4}{\sigma}.\notag
		\end{equation}
		From \eqref{decreases on even numbers} we deduce  $\frac{s_{2k+2}}{s_{2k}} < 1$ for all $k \in \N$; using this fact in the above expression gives  
		\begin{equation}
			\begin{split}
				\frac{(2k+1)s_{2k}}{s_{2k+2}}\geq\frac{(2k-1)s_{2k-2}}{s_{2k+2}}=\frac{(2k-1)s_{2k-2}}{s_{2k}}\frac{s_{2k}}{s_{2k+2}} > \frac{(2k-1)s_{2k-2}}{s_{2k}},\notag
			\end{split}
		\end{equation}
		and this concludes the proof.
	\end{proof}
	We can now move on to proving Proposition \ref{uniqueness related to first derivative}.
	\begin{proof}[Proof of Proposition \ref{uniqueness related to first derivative}]
		From \eqref{series expansion zeta}, $\zeta_{\sigma}$ is an odd function. Moreover, 
		since $0 \leq \frac{s_{2k+2}}{(2k+1)s_{2k}} \leq \frac{1}{2k+1}$ for all $k \in \N$, we deduce 
		from \eqref{gamma} that $\Upsilon_k(\sigma) \downarrow -\sigma$ as $k \uparrow +\infty$ for 
		any fixed $\sigma > 0$. Hence, if we set $k_{\sigma}:= \min \left \{ k \in \N : 
		\Upsilon_k(\sigma) \leq 0 \right \}$, since $\{\Upsilon_{k}(\sigma)\}_{k \in \N}$ is decreasing
		for any given $\sigma > 0$ we deduce  that $\Upsilon_k(\sigma) > 0$ for $k \leq k_{\sigma}-1$ 
		and $\Upsilon_k(\sigma) \leq 0$ for $k \geq k_{\sigma}$. From this and the fact that 
		$\zeta_{\sigma}^{2k+1}(0)=\frac{(2k+1)}{\sigma^{2k+1}} s_{2k}\Upsilon_{k}(\sigma)$, $k \in 
		\N$, we obtain that the following power series representation of $\zeta_{\sigma}$ holds 
		\begin{equation}
			\zeta_{\sigma}(m)= \sum \limits_{0 \leq k \leq k_{\sigma}-1} 
			\frac{|\zeta_{\sigma}^{(2k+1)}(0)|}{(2k+1)!} m^{2k+1} -\sum \limits_{k \geq k_{\sigma}} 
			\frac{|\zeta_{\sigma}^{(2k+1)}(0)|}{(2k+1)!} m^{2k+1}. \notag
		\end{equation}
		Taking out $m^{2k_{\sigma}+1}$ we have
		\begin{equation}\label{(B)}
			\zeta_{\sigma}(m)= m^{2k_{\sigma}+1} \left \{ \sum \limits_{k \leq k_{\sigma}-1} 
			\frac{|\zeta_{\sigma}^{(2k+1)}(0)|}{(2k+1)!} m^{2k-2k_{\sigma}} -\sum \limits_{k \geq 
				k_{\sigma}} \frac{|\zeta_{\sigma}^{(2k+1)}(0)|}{(2k+1)!} m^{2k-2k_{\sigma}} \right \}. 
		\end{equation} 
		We can now conclude as in \cite[cfr. Step 4, Theorem 2.1]{tugaut}. Indeed, since the function 
		$m \to m^{2k-2k_{\sigma}}$ (resp. $m \to -m^{2k-2k_{\sigma}}$) is strictly decreasing for all 
		$k \leq k_{\sigma}-1$ (resp. for all $k \geq k_{\sigma}$) then we deduce that the factor 
		between brackets in \eqref{(B)} is strictly decreasing for $m>0$. Hence, again from 
		\eqref{(B)} we deduce  that if $m>0$ is a root of $\zeta_{\sigma}$ then such an $m$ must be a 
		root of the factor between brackets in \eqref{(B)}(which we know it admits at most one root 
		because it is strictly decreasing for $m>0$). Hence, $\zeta_{\sigma}$ admits at most one zero
		on $(0,+\infty)$. Furthermore, $\zeta_{\sigma}$ is odd so that it admits exactly either one or
		three zeroes on $\R$. {Once this is in place, we have to determine which situation occurs.
		From \eqref{(B)} we have 
		\begin{equation}\label{(B')}
			\zeta_{\sigma}^{'}(m)= m^{2k_{\sigma}} \left \{ \sum \limits_{0 \leq k \leq k_{\sigma}-1} 
			\frac{|\zeta_{\sigma}^{(2k+1)}(0)|}{(2k)!} m^{2k-2k_{\sigma}} -\sum \limits_{k \geq 
				k_{\sigma}} \frac{|\zeta_{\sigma}^{(2k+1)}(0)|}{(2k)!} m^{2k-2k_{\sigma}} \right \}. 
		\end{equation} 
		By the same argument applied to 
		\eqref{(B)} for $\zeta_{\sigma}$ we deduce that the factor between brackets in \eqref{(B')} is strictly decreasing. 
		If $\zeta_{\sigma}^{'}(0)<0$, then $k_\sigma=0$ and we conclude that 
		$\zeta_{\sigma}^{'}(m)$ is strictly decreasing; hence $\zeta_{\sigma}^{'}(m) < 0$ for all $m \geq 0$ so that 
		$\zeta_{\sigma}$ admits a unique zero on $\R$.
		If on the other hand $\zeta_{\sigma}^{'}(0)>0$, then, using $\zeta_{\sigma}(1)<0$, we know that $\zeta_{\sigma}$ admits at least one root on $(0,1)$. Then it admits exactly one root on $(0,+\infty)$ and the proof is thus concluded.} 
	\end{proof}
	
	In order to prove Theorem \ref{teorema segno zeta} we need to state the following asymptotic expansion results.
	\begin{lemma} \label{asymptotic expansion}
		Let $U$ and $G$ be two $C^{\infty}(\T;\R)$-continuous functions. Let us define $U_{m}=U+m \cdot G$ where $m$ is a parameter belonging to some compact interval $I$ of $\R$. Moreover, assume that $U_{m}$ admits a unique global minimum at $x_m$, such that $U_m^{''}(x_m) >0$ and $0<x_m<2\pi$. Then, for any $f \in C^3(\T;\R)$, the following asymptotic result holds (as $\sigma>0$ tends to 0):
		\begin{equation}\label{expansion f}
			\int_{\T} f(x)e^{-\frac{1}{\sigma}U_m(x)}\,dx=\sqrt{\frac{2\pi \sigma}{\mathcal{U}_2}} e^{-\frac{1}{\sigma}U_m(x_m)} \left(f(x_m)+\gamma_f \sigma + o_m(\sigma)  \right),
		\end{equation}
		with 
		\begin{equation} \label{gamma constant}
			\gamma_f:= f(x_m) \left ( \frac{5\mathcal{U}_3^2}{24\mathcal{U}_2^3}-\frac{\mathcal{U}_4}{8\mathcal{U}_2^2} \right) -f^{'}(x_m) \frac{\mathcal{U}_3}{2\mathcal{U}_2^2}+\frac{f^{''}(x_m)}{2 \mathcal{U}_2},
		\end{equation}
		where $\mathcal{U}_k:=U_m^{(k)}(x_m)$ for $k \in \N$, the notation $o_m(\sigma)$ is intended to mean that $\frac{o_m(\sigma)}{\sigma} \to 0$ as $\sigma \to 0$ and the convergence holds uniformly in $m \in I$.
	\end{lemma}
	\begin{proof}[Proof of Prop. \ref{asymptotic expansion}]
		This is obtained by Laplace method and we refer the reader to \cite[Lemma A.3, Step 1-Step 2.2.]{hermann} for further details.
	\end{proof}
	From Lemma \ref{asymptotic expansion} we obtain the following lemma, the proof of which is postponed to Appendix \ref{uniqueness sigmageq1}.
	\begin{lemma} \label{lemma asymptotic expansion}
		The following two asymptotic expansions hold
		\begin{equation} \label{expansion 1}
			s_0= \sqrt{\frac{\pi \sigma}{2}} e^{\frac{1}{\sigma}} \left(2+o(1)  \right),
		\end{equation}
		\begin{equation} \label{expansion 2}
			s_2=\sqrt{\frac{\pi \sigma}{2}} e^{\frac{1}{\sigma}} \left(2+o(1)  \right),
		\end{equation}
		where the notation $o(1)$ means that $o(1) \to 0$ as $\sigma \downarrow 0$ and we recall that the coefficients $s_k$ have been defined in \eqref{sk}.
	\end{lemma}
	\begin{proof}[Proof of Theorem \ref{teorema segno zeta}]
		To establish which case holds, i.e. whether $\zeta_{\sigma}^{'}(0)>0$ or $\zeta_{\sigma}^{'}(0) < 0$, we study the first derivative of $\zeta_{\sigma}$ at $m=0$. A straightforward calculation shows that 
		\begin{equation}\label{first derivative iniziale}
			\zeta_{\sigma}^{'}(0)= \frac{s_0}{\sigma} \left ( \frac{s_2}{s_0}-\sigma \right )=\frac{1}{\sigma} \left ( s_2-\sigma s_0 \right ).
		\end{equation}
		We note that since $s_2 \leq s_0$, it is clear from \eqref{first derivative iniziale} that for $\sigma> 1$, $\zeta_{\sigma}^{'}(0)<0$. While, from Lemma \ref{lemma asymptotic expansion} and \eqref{first derivative iniziale}, it is also clear that for $\sigma\ll 1$, $\zeta_{\sigma}^{'}(0)>0$. Since the map $\sigma \to \zeta_{\sigma}^{'}(0)$ is continuous we deduce that there exists a root $\sigma_c>0$ of the map $\sigma \to \zeta_{\sigma}^{'}(0)$. The uniqueness of such a $\sigma_c$ remains to be proven. To this end, we first note that by using the identity $\cos(2x)=1-2\sin^2 x$, the factor $s_2-\sigma s_0$  can be rearranged into the following form: 
		\begin{equation}\label{first derivaitve rearranged}
			\begin{split}
				s_2-\sigma s_0 & = \int_{\T}  \sin^2 x \, e^{-\frac{1}{\sigma}\cos(2x)}\,dx-\sigma \int_{\T} e^{-\frac{1}{\sigma}\cos(2x)}\,dx \\
				& =\left( \frac{1}{2}-\sigma \right )\int_{\T} e^{-\frac{1}{\sigma}\cos(2x)}\,dx-\frac{1}{2}\int_{\T}\cos(2x)e^{-\frac{1}{\sigma}\cos(2x)}\,dx. \\
			\end{split}
		\end{equation}
		Once this is in place let us introduce the family of functions defined as 
		\begin{equation*}\label{besselI}
			I_n(z):=\int_{\T} \cos(2nx)e^{z\cos(2x)}\,dx,\quad z \in \R,\, n \in \N,
		\end{equation*}
		and, consequently, 
		\begin{equation}\label{Besselr}
			r_n(z):=\frac{I_{n+1}(z)}{I_{n}(z)},\quad  z \in \R,\, n \in \N,
		\end{equation}
		the family of functions $\{I_n\}_{n \in \N}$ is commonly referred to as modified Bessel functions of first kind.
		It is well-known that $I_0$ is an even function while $r_0$ is an odd function and, moreover, $r_0(z)>0$, $r_0^{'}(z)>0$ for $z>0$ (see \cite[(15)]{Amo74}). \\
		By using \eqref{first derivative iniziale} and \eqref{first derivaitve rearranged}, we obtain the following expression for $\zeta_{\sigma}^{'}(0)$: 
		\begin{equation}\label{first derivative bessel}
			\zeta_{\sigma}^{'}(0)= \frac{1}{2} I_0 \left (-\frac{1}{\sigma} \right ) \left ( \frac{1}{\sigma} -2 -\frac{1}{\sigma}r_0 \left ( -\frac{1}{\sigma} \right )\right )=\frac{1}{2} I_0 \left (\frac{1}{\sigma} \right ) \left ( \frac{1}{\sigma} -2 +\frac{1}{\sigma}r_0 \left ( \frac{1}{\sigma} \right )\right ). \notag
		\end{equation}
		Since $I_0\left (\frac{1}{\sigma} \right )>0$ for all $\sigma>0$, to prove the uniqueness of $\sigma_c$ it suffices to study the set of zeroes of the function $f_c$ defined in \eqref{function f}.  
		By taking the first derivative of $f_c$ we obtain 
		\begin{equation}\label{derivative f}
			f_c^{'}(\sigma)=-\frac{1}{\sigma^2}-\frac{1}{\sigma^2}r_0\left( \frac{1}{\sigma}\right )-\frac{1}{\sigma^3}r_0^{'}\left (\frac{1}{\sigma} \right )<0,\quad\sigma >0.
		\end{equation}
		Hence, $f_c$ is a strictly decreasing function and, therefore, the critical value $\sigma_c>0$ must be the unique root of $f_c$.
		Lastly, since $\sigma_c$ is the unique zero of $f_c$, an approximation to its value can be obtained e.g. via the bisection method. This is the procedure that led to the value $\sigma_c \simeq 0.7709$. This concludes the proof.
	\end{proof}
	
	\subsection{Proof of Step 3}\label{Exactly three invmeas} 
	In this subsection we prove that when $0 < \sigma \ll 1$ the unique fixed point of the map $g_{\sigma}:\R^2 \to \R^2$ restricted to $\mM_1$ is the origin.
	To this end, we introduce the map $h_{\sigma}:\R \to \R$, which is the restriction of $g_\sigma$ to $\mM_1$, namely 
	\begin{equation} \label{map h}
		h_{\sigma}(m):= \frac{1}{\hat{Z}_{\sigma}(m)} \int_{\T} \cos x \,e^{-\frac{1}{\sigma}(\cos(2x)-m \cos x)}
		\,dx,\,\,\hat{Z}_{\sigma}(m):=\int_{\T} e^{-\frac{1}{\sigma}(\cos(2x)-m \cos x)}\,dx.
	\end{equation}
	Then $m \in \R$ is a fixed point of $h_{\sigma}$ if and only if $(m,0)$ is a fixed point of $g_{\sigma}$. We already know that the map $h_{\sigma}$ has a fixed point, as $h_{\sigma}(0)=0$. Our goal is to show that when $\sigma \ll 1$ the map $h_{\sigma}$ does not admit any further fixed point other than $m=0$. 
	\begin{theorem}\label{punti fissi M2=0}
		When $0 < \sigma \ll 1$ the map $h_{\sigma}:\R \to \R$ defined in \eqref{map h} admits a unique fixed point given by $m=0$.
	\end{theorem}
	The proof of Theorem \ref{punti fissi M2=0} relies on the asymptotic expansions in Lemma \ref{stime asintotiche h} below. So we first state Lemma \ref{stime asintotiche h} and then prove Theorem \ref{punti fissi M2=0}. The proof of Lemma \ref{stime asintotiche h} is in Appendix \ref{uniqueness sigmageq1}.
	\begin{lemma} \label{stime asintotiche h}
		The following asymptotic expansions (for $\sigma$ small) hold, uniformly over $m \in [0,1]$:
		\begin{equation} \label{expansion 2 h}
			\begin{split}
				& \int_{\T} e^{-\frac{1}{\sigma}(\cos(2x)-m\cos x)}\,dx = \sqrt{\frac{2\pi \sigma}{4-\frac{m^2}{4}}} e^{\frac{1}{\sigma} \left( 1+ \frac{m^2}{8} \right)} 
				\left(2+\left ( c(m)+\frac{4}{(4-\frac{m^2}{2})^2} \right)\sigma + o_m(\sigma) \right) , 
			\end{split}
		\end{equation} 
		\begin{equation}\label{expansion 3 h}
			\begin{split}
				& \int_{\T} \cos x \,e^{-\frac{1}{\sigma}(\cos(2x)-m\cos x)}\,dx = \sqrt{\frac{2\pi \sigma}{4-\frac{m^2}{4}}} e^{\frac{1}{\sigma} \left( 1+ \frac{m^2}{8} \right)} \left( \frac{m}{2}+ \bar{c}(m)\sigma + o_m(\sigma) \right) ,
			\end{split}
		\end{equation} 
		\begin{equation} \label{expansion 4 h}
			\begin{split}
				& \int_{\T} \cos^2 x \, e^{-\frac{1}{\sigma}(\cos(2x)-m\cos x)} \,dx=\sqrt{\frac{\pi \sigma}{2}} e^{\frac{1}{\sigma}\left( 1+ \frac{m^2}{8} \right)} \left(\frac{m^2}{8}+\left( \hat{c}(m)+\frac{2}{4-\frac{m^2}{4}} \right)\sigma + o_m(\sigma)  \right),
			\end{split}
		\end{equation}
		where $\frac{o_m(\sigma)}{\sigma} \to 0$ as $\sigma \downarrow 0$  and 
		$c,\bar{c},\hat{c}:[0,1] \to \R$ are continuous functions such that $c(m),\bar{c}(m), \hat{c}(m) \to 0$ as $m \downarrow 0$. In particular, we have $c(m)=\bar{c}(m)=\hat{c}(m)=O(m)$ as $m \in [0,1]$. \footnote{We recall that a function $f:[0,1] \to \R$ satisfies $f=O(m)$ with $m \in [0,1]$ if there exists a constant, say $D \in \R$ such that $|f(m)| \leq D \cdot m$, for all $m \in [0,1]$.}
	\end{lemma}
	\begin{proof}[Proof of Theorem \ref{punti fissi M2=0}]
		Note that since $|h_{\sigma}(m)| \leq 1$ for all $m \in \R$, the fixed points of $h_{\sigma}$ are in the interval $[-1,1]$. Moreover, since $h_{\sigma}$ is a $C^{\infty}(\R;\R)$-odd function (continuity and differentiability is meant with respect to $m \in \R$), it is enough to restrict $h_{\sigma}$ to the interval $[0,1]$ and prove that $m=0$ is the unique fixed point for $\sigma>0$ sufficiently small. We begin with computing the first derivative of $h_{\sigma}$:
		\begin{equation}\label{first derivative h}
			\begin{split}
				h_{\sigma}^{'}(m)& =-\frac{1}{\sigma (Z_{\sigma}(m))^2} \left( \int_{\T} \cos x\, e^{-\frac{1}{\sigma} \left(\cos(2x)-m\cos x \right)}\,dx \right)^2 \\
				& +\frac{1}{\sigma Z_{\sigma}(m)} \int_{\T} \cos^2 x \,e^{-\frac{1}{\sigma} \left( \cos(2x)-m\cos x \right)} \,dx.
			\end{split}
		\end{equation}
		In the first part of the proof we are going to prove that if we fix a $\delta>0$ small enough then it follows that 
		\begin{equation}\label{triangolo}
			\text{$|h_{\sigma}^{'}(m)| < 1$,\quad for $\sigma\ll 1$ and $m \in [0,\delta)$.}
		\end{equation} 
		Hence, 
		\begin{equation} \label{inequality h 1}
			h_{\sigma}(m) < m,\quad\text{for $\sigma \ll 1$ and $m \in (0,\delta)$}.
		\end{equation}
		The bound \eqref{triangolo} is a consequence of the asymptotic expansions of Lemma \ref{stime asintotiche h}. Indeed, from \eqref{expansion 2 h}, \eqref{expansion 3 h}, \eqref{expansion 4 h} and \eqref{first derivative h} we obtain 
		\begin{equation}
			\begin{split}
				h_{\sigma}^{'}(m)& = \frac{1}{\sigma} \left( \frac{\frac{m^2}{8}+\left( \hat{c}(m)+\frac{2}{4-\frac{m^2}{4}} \right)\sigma + o_m(\sigma) }{2+\left( c(m)+\frac{4}{(4-\frac{m^2}{2})^2} \right)\sigma + o_m(\sigma)} \right) -\frac{1}{\sigma} \left( \frac{\frac{m}{2}+ \bar{c}(m)\sigma + o_m(\sigma)}{2+\left ( c(m)+\frac{4}{(4-\frac{m^2}{2})^2} \right)\sigma + o_m(\sigma)}\right)^2. \notag
			\end{split}
		\end{equation}
		By expanding the square for the second addend, we have 
		\begin{equation}
			\begin{split}
				h_{\sigma}^{'}(m)& = \frac{1}{\sigma} \left( \frac{\frac{m^2}{8}+\left( \hat{c}(m)+\frac{2}{4-\frac{m^2}{4}} \right)\sigma + o_m(\sigma) }{2+\left( c(m)+\frac{4}{(4-\frac{m^2}{2})^2} \right)\sigma + o_m(\sigma)} \right) -\frac{1}{\sigma} \left( \frac{\frac{m^2}{4}+ 2\bar{c}(m)\sigma + o_m(\sigma)}{4+\left ( 2c(m)+\frac{8}{(4-\frac{m^2}{2})^2} \right)\sigma + o_m(\sigma)}\right), \notag
			\end{split}
		\end{equation}
		hence,
		\begin{equation}
			\begin{split}
				h_{\sigma}^{'}(m)
				& =\frac{2\hat{c}(m)-2\bar{c}(m)+\frac{4}{4-\frac{m^2}{4}}+\frac{o_m(\sigma)}{\sigma}} {4+\left( 2c(m)+\frac{8}{(4-\frac{m^2}{2})^2} \right)\sigma + o_m(\sigma)}. \notag
			\end{split}
		\end{equation}
		If we set $\varepsilon,\delta>0$ small enough then there exists a $\hat{\sigma}>0$ such that if $\sigma<\hat{\sigma}$ then $|h_{\sigma}^{'}(m) - \frac{1}{4}| \leq \varepsilon$ for all $m \in [0,\delta)$. This concludes the first part of the proof.\\
		In the remaining part of the proof we are going to show that $h_{\sigma}(m) < m$ for all $m \in [\delta,1]$ provided $\sigma$ is sufficiently small, where $\delta>0$ is as in \eqref{inequality h 1}. To be precise, we are going to prove that $h_{\sigma}(m)$ converges to $\frac{m}{4}$ uniformly over the interval $[\delta,1]$. Indeed, again from Lemma \ref{stime asintotiche h} we have
		\begin{equation}
			\begin{split}
				h_{\sigma}(m) & = \frac{\frac{m}{2}+ \bar{c}(m)\sigma + o_m(\sigma)}{2+ \left( c(m)+\frac{4}{(4-\frac{m^2}{2})^2} \right)\sigma+ o_m(\sigma)},\notag
			\end{split}
		\end{equation}
		with $\frac{o_m(\sigma)}{\sigma} \to 0$ as $\sigma \downarrow 0$ uniformly in $m \in [\delta,1]$. As a consequence, $h_{\sigma}(m) < m$ for all $m \in [\delta,1]$, provided $\sigma$ is sufficiently small. This concludes the proof.
	\end{proof}
	The proof of Step 4 is deferred to Appendix \ref{uniqueness sigmageq1}.

	\section{Proof of Theorem \ref{mild solution global existence}}\label{Well-Posedness McKean_Vlasov}
	
	In this section we study the well-posedness of the problem \eqref{spde Q}.	As we  use a combination of the arguments of e.g.~\cite{burgers} developed for the stochastic Burgers' equation, and of those  used in the McKean-Vlasov PDE literature, in particular \cite{pav-car},  in places we give only essential details.
	\begin{proof}[Proof of Theorem \ref{mild solution global existence}]
		
		The stochastic process $u(t)$ is a mild solution of \eqref{spde Q} (in the sense \eqref{mild solution}) if and only if the process $v(t)=u(t)-W_{A}(t)$  is a mild solution of the following problem
		\begin{equation} \label{sistema normale deterministico}
			\begin{cases}
				\pa_tv= Av+\pa_x \left[ V^{'}\tilde{v}+ ( F^{'} \ast \tilde{v} ) \,\tilde{v} \right],\quad  (0,T) \times \T, & \\ 
				v(t,0)=v(t,2\pi),\quad  t \in [0,T],& \\
				v(0,x)=u_0(x),\quad  x \in \T, &
			\end{cases}
		\end{equation}
		where $\tilde{v}(t):=v(t)+W_{A}(t)$, $t \in [0,T]$. Therefore, to prove Theorem \ref{mild solution global existence}, it is enough to show the global existence and uniqueness of a mild solution to \eqref{sistema normale deterministico}, which is what we do in the following. 
		
		As classical, see e.g. \cite{burgers},  global existence and uniqueness (in mild sense) follow directly by local existence and uniqueness, plus appropriate  a priori estimates for the (mild) solution. We prove local well-posedness in Proposition \ref{local existence mild solution}
		and the a priori estimates in Proposition \ref{stime sistema deterministico}. 
	\end{proof}

	%To do so, we start by recalling the definition and some properties of the stochastic convolution $W_A$ (see Def. \ref{convoluzione stochastica}), after that we provide the definition of mild solution to SPDE \eqref{spde Q} and, by a fixed point type of argument, a local existence and uniqueness result for \eqref{spde Q} is proven. Finally, by a regularization argument and an a priori estimate obtained for a regularized version of equation (\ref{spde Q}) (see Section \ref{auxiliary semi} for further details) we prove that the local solution obtained beforehand can be extended up to any arbitrary time $T>0$.
	%We highlight that the two key ingredients in the proof of Theorem \ref{mild solution global existence} are the existence and uniqueness of a local mild solution of \eqref{sistema normale deterministico} (Proposition \ref{local existence mild solution}) and the a priori estimates for the local mild solution (Proposition \ref{stime sistema deterministico}). 
	
	We recall that a continuous $L^2(\T;\R)$-valued stochastic process $\{v(t)\}_{t \in [0,T]}$, is a mild solution to \eqref{sistema normale deterministico} if the following identity holds for every $t \in [0,T]$, 
	\begin{equation} \label{mild solution deterministico}
		v(t) =e^{tA}u_0+\int_0^t e^{(t-s)A}\partial_x \left[  V^{'}\tilde{v}(s)+(F^{'} \ast \tilde{v})(s)\tilde{v}(s) \right]\,ds ,\qquad  \mathbb{P}-\mathrm{a.s.}. 
	\end{equation} 
	
	Thus, introducing the linear map $P:C([0,T];H^1(\T;\R))\to C([0,T];L^2(\T;\R))$ given by
	\begin{equation} \label{operatore P}
		P[z](t):=\int_0^t e^{(t-s)A}\partial_x z(s)\,ds, \qquad \mbox{for } z\in C([0,T];H^1(\T;\R))\,,
	\end{equation}
	we can rewrite equation \eqref{mild solution deterministico} as
	\begin{equation} \label{mild sol versione abbrieviata}
		v(t)=e^{tA}u_0+P \big [ V^{'}\tilde{v}+(F^{'} \ast \tilde{v})\tilde{v} \big ](t),\qquad \text{$\mathbb{P}$-a.s.}, \, t \in [0,T]\,.
	\end{equation}

	Hence an $L^2(\T;\R)$-valued stochastic process $v$ is said to be a {\em local mild solution} to \eqref{sistema normale deterministico} if there exists a stopping time $T^*$ such that equation \eqref{mild sol versione abbrieviata} is satisfied for all $t \in [0,T^*)$, $\mP$-a.s.

	We will make use of the following technical lemma for the operator $P$, the proof of which is in Appendix \ref{estimate heat kernel}. 
	
	\begin{lemma} \label{stime heat kernel}
		The map $P$ defined in \eqref{operatore P} can be extended to a bounded linear operator over the space $C\left ([0,T];L^2(\T;\R) \right )$,
		%(the extension is denoted by $P$ as well)  
		$$P:C([0,T];L^2(\T;\R)) \to C([0,T];L^2(\T;\R)).$$
		%Moreover, for any given $z \in C([0,T];L^1(\T;\R))$ the following holds for all $t \in [0,T]$
		%\begin{equation}\label{diseguaglianza trequarti}
		%	\left \| P \left[ z \right ](t) \right \|_{L^2(\T;\R)} \leq C_1 \int_0^t \rho(t-s)\|z(s)\|_{L^1(\T;\R)}\,ds,
		%\end{equation} 
		%where $C_1$ is a positive real constant and $ t \in \R_{+} \to \rho(t) \in \R_{+}$ is a continuous function defined as 
		%\begin{equation}
		%\rho(t):=\sqrt{\frac{1}{t}+\frac{1}{t^{\frac{3}{2}}}},\,t \in \R_{+}. \notag
		%\end{equation}
		Furthermore, if $z \in C \left([0,T];L^2(\T;\R)\right)$ then, for all $t \in [0,T]$,
		\begin{equation}\label{diseguaglianza unmezzo}
			\left \| P \left[ z \right ](t) \right \|_{L^2(\T;\R)} \leq C_2 \int_0^t (t-s)^{-\frac{1}{2}}\|z(s)\|_{L^2(\T;\R)}\,ds,
		\end{equation} 
		where $C_2>0$ is a positive constant.
	\end{lemma}
	The above result is well-known in similar settings, see e.g. \cite[Lemma 14.2.1]{ergodicity}  or \cite[Lemma 5.2 and Lemma 5.4]{dapra04}, though we could not find it for the specific setup in which we work, so we include the proof in Appendix \ref{appendix: proofs of sec 7}.

	\begin{proposition}\label{local existence mild solution}
		For any initial datum $u_0 \in L^2(\T;\R)$ and for a.e. $\omega \in \Omega$, there exists a stopping time $T^*=T^*(\omega)$ such that  equation \eqref{sistema normale deterministico} has a unique local mild solution (in the sense defined above) up to time $T^*$. 
	\end{proposition}
	\begin{proof}
		We study \eqref{mild solution deterministico},  or equivalently \eqref{mild sol versione abbrieviata}, pathwise for any given $\omega\in\Omega'$, where $\Omega^{'} \subset \Omega$ is the set 
		\begin{equation}\label{omegaprimo}
			\Omega^{'}:= \left \{ \omega \in \Omega\,|\, W_A(\omega) \in C\left ( [0,T];H^1(\T;\R) \right ) \right \}.
		\end{equation}
		The set $\Omega^{'}$ is measurable with $\mP(\Omega^{'})=1$  (as $W_{A}$ is a continuous $H^1(\T;\R)$-valued process, see Lemma \ref{lemma irriducibilità}). 
		
		So, we fix $u_0 \in L^2(\T;\R)$ and $\omega\in\Omega'$, and choose an $m \in \R_+$ such that  $m>\|u_0\|_{L^2(\T;\R)}$. We want to use a fixed point argument on the space
		$$\Sigma(m,T^{*}) := \big \{v \in C([0,T^{*}];L^2(\T;\R)) \;\big |\; \|v(t)\|_{L^2(\T;\R)} \leq m,\; \forall t \in [0, T^{*}]\big \}$$
		applied to the map $\cG$, defined as 
		\begin{equation}
			\begin{split}
				\mathcal{G}(v)(t) :=e^{tA}u_0+P\left [V^{'}(v+W_{A})\right ](t) +P\left [(F^{'} \ast (v+W_{A}))(v+W_{A})\right ](t), \notag
			\end{split}
		\end{equation}
		for any $t \in [0,T^{*}]$ and $v \in \Sigma(m,T^{*})$.  We therefore need to  show 
		that, provided $T^{*}$  is small enough,  the space $\Sigma(m,T^{*})$ is invariant under $\cG$  and $\cG$ is a contraction on $\Sigma(m,T^{*})$. 
		
		To show that there exists $T^*>0$ small enough such that $\Sigma(m,T^{*})$ is invariant under $\cG$,   let $v \in \Sigma(m,T^{*})$; then, 
		\begin{equation}\label{Gv}
			\begin{split}
				\|\mathcal{G}(v)(t)\|_{L^2(\T;\mathbb{R})}  \leq\,& \|e^{t A}u_0\|_{L^2(\T;\R)}  +\left \| P \left [ V^{'}(v+W_{A})\right ](t) \right \|_{L^2(\T;\R)} \\
				& + \left \| P \left [(F^{'} \ast (v+W_{A}))(v+W_{A}) \right ](t) \right \|_{L^2(\T;\R)}, 
			\end{split}
		\end{equation}
		for any $t \in [0,T^*]$. Since $e^{tA}$ is a contraction semigroup, we have $$\|e^{t A}u_0\|_{L^2(\T;\R)} \leq \|u_0\|_{L^2(\T;\R)} < m,$$
		for any $t>0$.
		Moreover, by Lemma \ref{stime heat kernel}, the second and third addends on the RHS of \eqref{Gv} can be bounded by
		\begin{equation}
			\begin{split}
				&\left \| P \left [ V^{'}(v+W_{A})\right ](t) \right \|_{L^2(\T;\R)}\\
				&\qquad \leq C \int_{0}^t (t-s)^{-\frac{1}{2}}\left \|V^{'}(v(s)+W_{A}(s)) \right \|_{L^{2}(\T;\R)}\,ds \\
				&\qquad \leq C \|V^{'}\|_{L^{\infty}(\T;\R)}\int_{0}^t (t-s)^{-\frac{1}{2}}\left(\|v(s)\|_{L^{2}(\T;\R)} + \|W_{A}(s)\|_{L^{2}(\T;\R)} \right )\,ds \\
				%& \leq C\|V^{'}\|_{L^{\infty}(\T;\R)}(m+\mu_2)\int_{0}^t (t-s)^{-\frac{1}{2}}\,ds \\
				&\qquad \leq C\|V^{'}\|_{L^{\infty}(\T;\R)}(m+\mu_2)\left(T^*\right)^{\frac{1}{2}}, \notag
			\end{split}
		\end{equation}
		and 
		\begin{equation}
			\begin{split}
				\left \| P \left [(F^{'} \ast (v+W_{A}))(v+W_{A}) \right ](t) \right \|_{L^2(\T;\R)}
				%\,& \leq C \|F^{'}\|_{L^{\infty}(\T;\R)} \int_0^t (t-s)^{-\frac{1}{2}} \|v(s)+W_{A}(s) \|_{L^2(\T;\R)}^2 \,ds \\
				\leq C \|F^{'}\|_{L^{\infty}(\T;\R)} (m^2+\mu_2^2)\left(T^*\right)^{\frac{1}{2}}, \notag
			\end{split}
		\end{equation}
		respectively,  where in the above  $\mu_2:=\sup_{t\in[0,T]}\|W_A\|_{L^2(\T;\R)}$ and $C$ is a generic positive constant (the value of which may change from line to line), independent of $m$. Therefore, we have the estimate
		\begin{equation*}
			\|\mathcal{G}(v)(t)\|_{L^2(\T;\mathbb{R})} \leq \|u_0\|_{L^2(\T;\R)} + C(m+\mu_2+m^2+\mu_2^2)\left(T^*\right)^{\frac{1}{2}}.
		\end{equation*}
		If we choose $T^*$ sufficiently small, then the operator $\mathcal{G}$ maps $\Sigma(m,T^*)$ into itself, i.e.~the map $\mathcal{G}:\Sigma(m,T^*) \to \Sigma(m,T^*)$ is well-defined. We now want to show that such a map is a contraction. To this end, let $v_1,v_2 \in \Sigma(m,T^{*})$;  then 
		\begin{equation}
			\begin{split}
				\left \|\mathcal{G}(v_1)(t)-\mathcal{G}(v_2)(t) \right \|_{L^2(\T;\R)}  \leq\,& \left \| P \left [ V^{'}(v_1-v_2) \right](t) \right \|_{L^2(\T;\R)} \\
				& +  \left \|P  \left[(F^{'} \ast \tilde{v}_1)\tilde{v}_1-(F^{'} \ast \tilde{v}_2)\tilde{v}_2 \right](t) \right \|_{L^2(\T;\R)},\notag
			\end{split}
		\end{equation}
		for every $t \in [0,T^*]$, where $\tilde{v}_1:=v_1+W_{A}$, $\tilde{v}_2:=v_2+W_{A}$.
		The first addend on the RHS  can be bounded using analogous calculations to those we have done in the above. As for the second addend,  applying again Lemma \ref{stime heat kernel} and using Young's inequality for convolutions, we have 
		\begin{equation}
			\begin{split}
				&\left \|P  \left[(F^{'} \ast \tilde{v}_1)\tilde{v}_1-(F^{'} \ast \tilde{v}_2)\tilde{v}_2 \right](t) \right \|_{L^2(\T;\R)}\\ %\leq\,& C \int_0^t (t-s)^{-\frac{1}{2}} \left \| \left ( F^{'} \ast \tilde{v}_1 \right)(s)\tilde{v}_1(s)- \left( F^{'} \ast \tilde{v}_2 \right )(s)\tilde{v}_2(s) \right \|_{L^2(\T;\R)}\,ds \\
				%&\qquad\qquad\leq\, C \int_0^t (t-s)^{-\frac{1}{2}} \left \| \left ( F^{'} \ast \tilde{v}_1 \right )(s)(v_1(s)-v_2(s)) \right \|_{L^2(\T;\R)}\,ds \\
				%&\qquad\qquad\quad\, +C \int_0^t (t-s)^{-\frac{1}{2}} \left \| \left ( F^{'} \ast (v_1(s)-v_2(s) )\right )\tilde{v}_2(s) \right \|_{L^2(\T;\R)} \,ds ,\\
				&\qquad\qquad\leq\, C \|F^{'}\|_{L^{\infty}(\T;\R)}\int_0^t (t-s)^{-\frac{1}{2}} \|\tilde{v}_1(s)\|_{L^2(\T;\R)}\|v_1(s)-v_2(s)\|_{L^2(\T;\R)}\,ds \\  
				&\qquad\qquad\quad\, +C \|F^{'}\|_{L^{\infty}(\T;\R)}\int_0^t (t-s)^{-\frac{1}{2}} \|\tilde{v}_2(s)\|_{L^2(\T;\R)}\|v_1(s)-v_2(s)\|_{L^2(\T;\R)}\,ds \\
				&\qquad\qquad\leq\, C\|F^{'}\|_{L^{\infty}(\T;\R)} \left(T^*\right)^{\frac{1}{2}}\left ( m+\mu_2 \right)\|v_1-v_2 \|_{C([0,T^*];L^2(\T;\R))}\,.\notag
			\end{split}
		\end{equation}
		Thus, we obtain 
		\begin{equation}
			\|\mathcal{G}(v_1)(t)-\mathcal{G}(v_2)(t)\|_{L^2(\T;\R)}
			\leq C(1+m+\mu_2)\left( T^* \right)^{\frac{1}{2}} \|v_1-v_2\|_{C([0,T^{*}];L^2(\T;\R))}, \notag
		\end{equation}
		for all $t \in [0,T^*]$ and for any given $v_1,v_2\in \Sigma(m,T^*)$. %We point out that the constant $C$ in the above inequality depends on the $L^{\infty}$-norm of $V^{'}$ and $F^{'}$.
		Choosing $T^*$ sufficiently small, the conclusion follows from the Banach fixed point theorem.  Note that $T^*$ is a stopping time as, by construction, it depends on $\omega$ only through  $\mu_2$. 
	\end{proof}

	To study a priori estimates for the mild solution to \eqref{sistema normale deterministico},  we start by recalling  the following technical lemma, the proof of which can be found e.g. in  \cite[Lemma A.2]{kruse}.
	
	\begin{lemma}[Generalized Gronwall's inequality]\label{Lemma di Gronwall generalizzato}
		Let $T >0$, $C_1,C_2 \geq 0$ and $q:[0,T] \to \mathbb{R}$ be a non-negative
		and continuous function and let $r > 0$. If
		\begin{equation}\label{dis gronwall variazione}
			q(t) \leq C_1+C_2 \int_0^t (t-s)^{-1+r}q(s)\,ds,\quad \text{for all $t \in [0,T]$}, 
		\end{equation}
		then there exists a non-negative, increasing and continuous function $\tilde{c}(t)=\tilde{c}(t,C_2,r)\geq 0$, $t \geq 0$ such that 
		\begin{equation}
			q(t) \leq C_1\tilde{c}(t),\quad  \text{for all $t \in [0,T]$}. \notag
		\end{equation}
	\end{lemma}
	 In the next proposition we still work pathwise. After the proof of Proposition \ref{stime sistema deterministico} we explain how to extend the solution up to a time $T$ independent of $\omega \in \Omega$. 
	
	\begin{proposition}[A priori estimates]\label{stime sistema deterministico}
		Let  $\tT=\tT(\omega)$ be such that a mild  solution $v$ to \eqref{sistema normale deterministico} exists up to time $\tT$ and   $v \in C([0,\tT];L^2(\T;\mathbb{R}))$.   Then,
		\begin{align} 	\|v(t)\|_{L^2(\T;\R)}^2  \leq a(u_0,W_A)(t)\cdot\exp\left(C\int_0^t b(u_0,W_A)(s)\,ds\right),\quad t \in [0,\tT],\,\text{$\mathbb{P}$-a.s.}, \label{L2 norm WA}
		\end{align}
		where the non-negative functions $a(v_0,\varphi)(t)$ and $ b(v_0,\varphi)(t)$ are given  by
		\begin{align*}
			a(v_0,\varphi)(t) :=\, & \|v_0\|^2_{L^2(\T;\R)}+C\int_0^t \|\varphi(s)\|^4_{L^2(\T;\R)}\,ds, \\	b(v_0,\varphi)(t):=\,&\|\varphi(t)\|_{L^2(\T;\R)}^2\\
			&+\left(\|v_0\|_{L^1(\T;\R)}+\int_0^t \left(1+ \|\varphi(s)\|_{H^1(\T;\R)}^2\right)\,ds\right)e^{C\int_0^t \|\varphi(s)\|_{H^1(\T;\R)}ds},
		\end{align*}
		for any $t \in [0,T]$, $v_0\in L^2(\T;\R)$ and $\varphi\in C([0,T];H^1(\T;\R))$, where $C>0$ is a constant depending on $V$ and $F$.
	\end{proposition}

	\begin{proof}[Proof of Proposition \ref{stime sistema deterministico}]
		In what follows, unless otherwise specified, $C$ denotes a generic deterministic positive constant, which may change from line to line. Note also that the time $\tT$ in the statement of the proposition does exist (for each $\omega$, just take some $\tT$ smaller than the time $T^*$ of the local result). 
		
		We start with considering a regularised version of the system \eqref{sistema normale deterministico}, where we replace the stochastic convolution $W_A$ with a {$C^{\infty}([0,T]\times \R)$-valued random variable $\varphi$}. Namely, we consider {the following random evolution}
		\begin{align} \label{pde regolare}
			\begin{cases}
				\partial_t v =  Av+\partial_x \left[ V^{'}(v+\varphi)+F^{'}*(v+\varphi)(v+\varphi) \right] \\
				v(t,0)=v(t,2\pi),\quad t \in [0,T] , \\
				v(0,x)= u_0(x),\quad  x \in \T, 
			\end{cases}
		\end{align}
		with $u_0 \in L^2(\T;\R)$. We prove in Proposition \ref{regularised PDE} that, {$\mP$-a.s.},  there exists  a unique global solution $v$ to \eqref{pde regolare} and that such a solution satisfies the following a priori estimates for the $L^2$ norm,
		\begin{align} 	\|v(t)\|_{L^2(\T;\R)}^2  \leq a(u_0,\varphi)(t)\cdot\exp\left(\int_0^t b(u_0,\varphi)(s)\,ds\right), \quad t \in [0,\tilde T],\,\,\mP-\text{a.s.}  \label{L2 norm}
		\end{align}
		where the functions $a(u_0,\varphi)(t)$ and $ b(u_0,\varphi)(t)$ are defined as in the statement of Proposition \ref{stime sistema deterministico}. To extend these estimates to the dynamics \eqref{sistema normale deterministico}, we consider a family $\{\varphi_n\}_{n \in \N}$ of $C^{\infty}([0,\tT]\times \R)$-valued random variables such that
		\begin{align}\label{phito W}
			\varphi_n \to W_{A}\quad\mathrm{in}\;C([0,\tT];H^1(\T;\R))\;\mP-\mathrm{a.s.},\;\mathrm{as}\;n\to\infty.
		\end{align}
		For any $n\in\N$, we denote by $v_n$ the unique mild solution to \eqref{pde regolare} with $\varphi$ replaced by $\varphi_n$. Moreover, we denote by $v\in C([0,\tT];L^2(\T;\R))$ the mild solution to \eqref{sistema normale deterministico} up to time $\tT$. We want to show that $v_n$ converges to $v$ in $C([0,\tT];L^2(\T;\R))$,  $\mathbb{P}$-a.s.
		
		Recall that $v$ and the sequence $\{v_n\}_{n \in \N}$ satisfy 
		\begin{align*}
			v(t)\,&=e^{tA}u_0+P \left [ V^{'}\tilde{v}+(F^{'} \ast \tilde{v})\tilde{v} \right ](t),\\
			v_n(t)\,&=e^{tA}u_0+P \left[ V^{'}\tilde{v}_n+(F^{'} \ast \tilde{v}_n)\tilde{v}_n \right ](t),
		\end{align*}
		respectively,  with $P$ defined in \eqref{operatore P}, having set $\tilde{v}=v+W_{A}$, $\tilde{v}_n=v_n+\varphi_n$, $n \in \N$.
		Then, by Lemma \ref{stime heat kernel} and with calculations completely analogous to those in the proof of Proposition \ref{local existence mild solution}, we have  
		\begin{align*}			
			\|v(t)-v_n(t)\|_{L^2(\T;\R)} & \leq  C \tT^{\frac{1}{2}} \left \|W_{A}-\varphi_n \right \|_{C([0,T];L^2(\T;\R))}\\
			&+C\int_0^t (t-s)^{-\frac{1}{2}} \left \|v(s)-v_n(s) \right \|_{L^2(\T;\R)}\,ds\\
			&+\int_0^t (t-s)^{-\frac{1}{2}} \left \| \left ( F^{'} \ast \tilde{v} \right )(s)\left(\tilde{v}(s)- \tilde{v}_n(s)\right) \right \|_{L^2(\T;\R)}\,ds \\
			& + \int_0^t (t-s)^{-\frac{1}{2}} \left \| \left ( F^{'} \ast (\tilde{v}-\tilde{v}_n) \right )(s)\tilde{v}_n(s) \right \|_{L^2(\T;\R)} \,ds\,.			
		\end{align*}
		Since  $v \in C \left ([0,\tT];L^2(\T;\R) \right )$ $\mP$-a.s., there exists a non-negative random variable $M_{1,\tT}=M_{1,\tT}(\omega) <+\infty$ $\mP$-a.s. such that $\|v(t)\|_{L^2(\T;\R)} \leq M_{1,\tT}$, for all $t \in [0,\tT]$, $\mP$-a.s (but we don't know the exact dependence of $M_{1,\tT}$ on $\tT$). Hence,
		\begin{align*}
			\left \| \left ( F^{'} \ast \tilde{v} \right )\left(\tilde{v}- \tilde{v}_n\right)(s) \right \|_{L^2(\T;\R)}
			\leq C&\|\tilde{v}(s)\|_{L^2(\T;\R)}\|\tilde{v}(s)-\tilde{v}_n(s)\|_{L^2(\T;\R)}\\
			\leq C&M_{1,\tT}\left(\|\varphi_n(s)-W_{A}(s)\|_{L^2(\T;\R)}+\|v(s)-v_n(s)\|_{L^2(\T;\R)}\right),
		\end{align*}
		and, similarly,
		\begin{align*}
			\left\| \left ( F^{'} \ast (\tilde{v}-\tilde{v}_n) \right )(s)\tilde{v}_n(s) \right \|_{L^2(\T;\R)}\,\leq\,CM_{2,\tT}\|v(s)-v_n(s)\|_{L^2(\T;\R)},
		\end{align*}
		for all $s \in [0,\tT]$, where $M_{2,\tT}=M_{2,\tT}(\omega)$ is a non-negative random variable with $M_{2,\tT} < + \infty$, $\mP$-a.s. such that $\|\tilde{v}_n\|_{C([0,\tT];L^2(\T;\R))} \leq M_{2,\tT}$ uniformly in $n \in \mathbb{N}$, $\mP$-a.s.; such a random variable  exists by \eqref{L2 norm}. 
		%Since the sequence $\{\varphi_n \}_{n \in \N}$ is uniformly bounded in $C \left ( [0,T];H^1(\T;\R) \right)$, $\mP$-a.s. This condition is guaranteed by assuming that $\varphi_n \xrightarrow[n \to +\infty]{} W_A$ in $C \left ([0,T];H^1(\T;\R) \right )$, $\mP$-a.s.\\
		From the above, we then have 
		\begin{equation}
			\begin{split}
				\|v(t)-v_n(t)\|_{L^2(\T;\R)}  \leq\,& C\|W_{A}-\varphi_n\|_{C([0,T];L^2(\T;\R))} \\
				& + C(M_{1,\tT}+M_{2,\tT})\int_0^t (t-s)^{-\frac{1}{2}}\|v(s)-v_n(s)\|_{L^2(\T;\R)}\,ds, \notag
			\end{split}
		\end{equation}
		so that  using Lemma \ref{Lemma di Gronwall generalizzato} finally gives
		$$\|v(t)-v_n(t)\|_{L^2(\T;\R)} \leq \tilde{c}_{\tT} \left( \|W_{A}-\varphi_n\|_{C([0,T];L^2(\T;\R))} \right),$$
		for all $t \in [0,\tT]$, where $\tilde{c}_{\tT}>0$ depends on  $\|V^{'}\|_{L^{\infty}(\T;\R)}$, $\|F^{'}\|_{L^{\infty}(\T;\R)}$, $M_{1,\tT}$ and $M_{2,\tT}$. Hence  $v_n$ converges to $v$ in $C([0,T];L^2(\T;\R))$ $\mathbb{P}$-a.s., as $n\to\infty$. 
		Finally, applying \eqref{L2 norm} to $\{v_n\}_{n\in\N}$ and noting that by definition of the functions $a$ and $b$, $a(u_0,\varphi_n) \xrightarrow{n\to\infty} a(u_0,W_A)$ and $b(u_0,\varphi_n) \xrightarrow{n\to\infty} b(u_0,W_A)$ $\mP$-a.s.~uniformly in $[0,\tT]$ (since $\varphi_n \xrightarrow{n\to\infty} W_A$ in $C([0,T];H^1(\T;\R))$ $\mP$-a.s.), we obtain the a priori estimate \eqref{L2 norm WA} for the $L^2$ norm of the solution $v$ to \eqref{sistema normale deterministico}. It is important to note that the constants $M_{1,\tT}$,  $M_{2,\tT}$ do not appear i the definitions of $a$ and $b$, they only appear in the estimates used to show the convergence of $v_n$ to $v$. 
	\end{proof}
	%\begin{remark}\label{remark 12}
	%	We note that the above line of argument works path-wise, i.e.~for any fixed $\omega \in \Omega$ for which $W_A(\omega) \in C \left ([0,T];H^1(\T;\R) \right )$.
	%	As a consequence, we deduce that the a priori estimate \eqref{L2 norm WA} is valid for all $\omega \in \Omega^{'}$, where $\Omega^{'}$ is defined by \eqref{omegaprimo} and such that $\mP(\Omega^{'})=1$.
	%\end{remark}
	{ Because of the form of the a priori estimate \eqref{L2 norm WA}, from a classical argument (see e.g. \cite{burgers})  it follows that the solution $v$ to \eqref{sistema normale deterministico} can be extended, for almost every $\omega \in \Omega$,  up to  a time $T=T(\omega)$.  \footnote{A difference between \cite{ergodicity} and our setting is that in \cite[Theorem 14.2.4]{ergodicity} estimates independent of the regularity of $\varphi$ are needed because there the authors consider cylindrical Wiener noise. In our case this further difficulty is not present. } To show that such a time can be taken to be independent of $\omega$, referring to the construction in the proof of the above Proposition \ref{stime sistema deterministico}  we observe (see  \cite[Theorem 14.2.4]{ergodicity}) the following: first,  by Proposition \ref{regularised PDE} the solution of \eqref{pde regolare} can be defined up to a time $T>0$ fixed a priori and independent of $\omega$. Hence all the $v_n$'s exist on an interval $[0,T]$, for any $T>0$, independent of $\omega$. This, combined with the fact that the sequence $\varphi_n$ can be taken so that the convergence \eqref{phito W} is on $C([0,T]; H^1)$, allows one to show that the a priori estimates of Proposition \ref{stime sistema deterministico} are in fact valid for any deterministic time interval $[0,T]$. 
	
	\section{Proof of point i) of Theorem \ref{theorem:existence and uniqueness inv measure}: irreducibility}\label{sec: irreducibility}	
	In this section we prove irreducibility of the dynamics \eqref{spde Q} using the methods of  \cite[Chapter 5]{dapra04}, \cite[Chapter 14]{ergodicity}. We recall that throughout this section $Q$ is assumed to satisfy \eqref{Qactsonbasis} and \eqref{eqn:Qtrace-class}.
	
	We start by considering the control system associated with \eqref{spde Q}
	\begin{equation}
		\begin{cases} \label{controlled system}
			\pa_t y=Ay+\pa_x \left [ V^{'}y+ \left( F^{'}*y \right ) y \right]+Q^{\frac{1}{2}}f,\quad t \in [0,T], & \\
			y(t,0)=y(t,2\pi),\quad  t \in [0,T], & \\
			y(0,x)=y_0(x),\quad  x \in \T, &
		\end{cases}
	\end{equation}
	with initial datum $y_0 \in L^2(\T;\R)$ and obtained from \eqref{spde Q} by replacing the stochastic forcing $\pa_t W$ with a {\em deterministic} control $f \in L^2 \left ([0,T];L^2(\T;\R) \right )$.  When we wish to emphasize the dependence of the solution of \eqref{controlled system} on the initial datum $y_0$ and on the control $f$ we use the notation $y(t;y_0,f)$. 
	
	In the same fashion as the proof of Proposition \ref{regularised PDE}, we can show that  system \eqref{controlled system} admits a unique global mild solution $y=y(t)$ and that  the following estimate (analogous to those found in the stochastic case)  holds:
	%\begin{align} \label{L1 yf norm}
		%\|y(t)\|_{L^1(\T;\R)} &\leq \|y_0\|_{L^1(\T;\R)}+\int_0^t \left \| Q^{\frac{1}{2}} f(s) \right \|_{L^1(\T;\R)}\,ds, \\
	%\label{L2 yf norm}
		%\|y(t)\|_{L^2(\T;\R)}^2& \leq \ell (y_0,f)(t) e^{\frac{1}{2} \left( \left \|V^{''} \right \|_{L^{\infty}(\T;\R)}t+\left \|F^{''} \right \|_{L^{\infty}(\T;\R)} \int_0^t \| y(s) \|_{L^{1}(\T;\R)}\,ds+t \right )} ,
	%\end{align}
    \begin{equation}\label{L2 yf norm}
		\|y(t)\|_{L^2(\T;\R)}^2 \leq \ell_2 (y_0,f)(t) e^{C\left(t+ \int_0^t \ell_1 (y_0,f)(s)\,ds\right)} ,
    \end{equation}
	for all $ t \geq 0$, for some constant $C>0$ depending on $V$ and $F$, and where $t \to \ell_1(y_0,f)(t)$, $t \to \ell_2(y_0,f)(t)$  are time-continuous increasing functions, namely 
    \begin{align*}
    &\ell_1(y_0,f)(t):=\|y_0\|_{L^1(\T;\R)}+\int_0^t \left \| Q^{\frac{1}{2}} f(s) \right \|_{L^1(\T;\R)}\,ds,\\
    &\ell_2(y_0,f)(t):= \|y_0\|_{L^2(\T;\R)}+\frac{1}{2}\int_0^t \|Q^{\frac{1}{2}}  f(s)\|_{L^2(\T;\R)}^2\,ds,
    \end{align*}
    for any $t\geq 0$.
	
	We want to show that system \eqref{controlled system} is approximately controllable. We recall that system \eqref{controlled system} is {\em approximately controllable in} $L^2(\T;\R)$ {\em at time $T>0$} via an $L^{2} \left ([0,T];L^2(\T;\R) \right )$-control if for any  $y_0, y_1 \in L^2(\T;\R)$ and  for all $\varepsilon>0$ there exists $f \in L^2 \left ([0,T];L^2(\T;\R) \right )$ such that the solution $y=y(t;y_0,f)$ of \eqref{controlled system} satisfies 
	\begin{equation}
		\|y(T;y_0,f)-y_1\|_{L^2(\T;\R)} \leq \varepsilon .\notag
	\end{equation}
	If the dynamics \eqref{controlled system} is approximately controllable in $L^2(\T;\R)$ at time $T$ for any $T>0$ then we simply say that \eqref{controlled system} is approximately controllable in $L^2(\T;\R)$.
	
	To show that  \eqref{controlled system} is approximately controllable in $L^2(\T;\R)$ via an  $L^2 \left ([0,T];L^2(\T;\R) \right )-$ control,  we will first show that it is approximately controllable in $C^2(\T;\R)$ via a $C([0,T]; C(\T;\R))$-control, see Lemma \ref{approxiamte lemma}. Then the smoothing properties of the deterministic part of the equation allow one to conclude the desired approximate controllability in $L^2(\T;\R)$ (via an  $L^2 \left ([0,T];L^2(\T;\R) \right )$-control), see Proposition \ref{approximate proposition}.
	Finally, irreducibility (in $L^2(\T;\R)$) of the
	semigroup $\{ \cP_t \}_{t \geq 0}$ associated with \eqref{spde Q} is deduced once we show that mild solutions of the SPDE \eqref{spde Q} can be approximated by solutions of the deterministic problem \eqref{controlled system}. This is the content of  Lemma \ref{lemma irriducibilità} and Theorem \ref{irreducibility of the semigroup}, the latter being the main result of this section. 
	We now begin to carry out the programme described above. 
	\begin{lemma}\label{approxiamte lemma}
	Let $Q$ satisfy \eqref{Qactsonbasis}-\eqref{eqn:Qtrace-class}. 
		For any $T>0$, $y_0,y_1 \in C^2(\T;\R)$ and $\varepsilon>0$ there exists $f \in C \left ( [0,T];L^2(\T;\R) \right )$ such that 
		\begin{equation} \label{inequality y1}
			\|y(T;y_0, f)-y_1\|_{L^2(\T;\R)} \leq \varepsilon.
		\end{equation}
	\end{lemma}
	\begin{proof}
		Let $y_0, y_1$ be two arbitrary but fixed points in  $C^2(\T;\R)$ and let  $\alpha_{y_0,y_1} = \alpha_{y_0,y_1}(t)$ be a continuous  path in $C^2(\T;\R)$, joining  $y_0$ and $y_1$, i.e.
		\begin{equation}\label{alpha pallino}
			\alpha_{y_0,y_1}(0)=y_0,\quad\alpha_{y_0,y_1}(T)=y_1.
		\end{equation}
		To fix ideas,  we will take 
		\begin{equation}
			\alpha_{y_0,y_1}(t):=\frac{T-t}{T}y_0+\frac{t}{T}y_1,\quad\quad t \in [0,T]\,, \notag
		\end{equation}
		and note that $\alpha_{y_0, y_1}$ is a function of time and space, but we omit the dependence on the space variable when not needed. Because $y_0, y_1 \in C^2$,  they belong to the domain of $A$, hence we can define the path $\beta_{y_0,y_1} = \beta_{y_0,y_1}(t) \subset C(\T;\R)$ as follows: 
		\begin{equation}
			\begin{split}
				\beta_{y_0,y_1}(t)&:=\pa_t \alpha_{y_0,y_1}(t)-A\alpha_{y_0,y_1}(t)\\
				& -\pa_x \left [ V^{'}\alpha_{y_0,y_1}(t)+ \left (F^{'}*\alpha_{y_0,y_1} \right )(t)\alpha_{y_0,y_1}(t)\right ], \quad t \in [0,T] \,. \notag
			\end{split}
		\end{equation}
		Then, by definition, $\alpha_{y_0,y_1}$  is a classical solution to the PDE 
		\begin{equation}\label{auxiliary PDE}
			\begin{cases} 
				\pa_t \alpha_{y_0,y_1}=A\alpha_{y_0,y_1}+\pa_x \left [ V^{'}\alpha_{y_0,y_1}+ \left( F^{'}*\alpha_{y_0,y_1} \right ) \alpha_{y_0,y_1} \right]+\beta_{y_0,y_1}, \quad t \in [0,T],& \\
				\alpha_{y_0,y_1}(t,0)=\alpha_{y_0,y_1}(t,2\pi), \quad  t \in [0,T], & \\
				\alpha_{y_0,y_1}(0,x)=y_0(x),\, \quad x \in \T ,&
			\end{cases}
		\end{equation}
		such that \eqref{alpha pallino} holds. If we prove that for every $\epsilon>0$  there exists $f \in C([0,T]; L^2(\T;\R))$ such that  the following holds
		\begin{equation}\label{distance L2 norm}
			\|y(t,y_0;f)-\alpha_{y_0,y_1}(t)\|_{L^2(\T;\R)} \leq C \varepsilon,\quad \text{for every $t \in [0,T]$},
		\end{equation}
		for some $C>0$ (and independent of $\epsilon$) then the proof is concluded. Indeed \eqref{inequality y1} readily follows from  \eqref{distance L2 norm} and \eqref{alpha pallino}.

		We will show that  if $f$ is any function  in $C\left ([0,T];L^2(\T;\R) \right )$ such that 
		\begin{equation}\label{distance betaf}
			\|\beta_{y_0,y_1}(t)- Q^{\frac{1}{2}} f(t) \|_{L^2(\T;\R)} \leq \varepsilon,\, \quad t \in [0,T], 
		\end{equation}
		then \eqref{distance L2 norm} holds.\footnote{The existence of at least one such function is obvious, as $\beta_{y_0, y_1}$ is a continuous function.} So, let $f$ be such that \eqref{distance betaf} holds and set  $w(t):=y(t;y_0, f) -\alpha_{y_0,y_1}(t)$.   Since $\alpha_{y_0,y_1}$ is a (mild) solution to \eqref{auxiliary PDE} and $y(t;y_0,f)$  is a mild solution to \eqref{controlled system}, we have 
		\begin{align} \label{non linear term} \notag
			w(t) & = \int_0^t e^{(t-s)A} \pa_x \left [ V^{'}w(s) \right ]\,ds\\
			& + \int_0^t e^{(t-s)A} \pa_x \left [ \left( F^{'}*w \right )(s) \alpha_{y_0,y_1}(s) - \left( F^{'}*y \right ) (s) w(s) \right ]ds \\ 
			& +\int_0^t e^{(t-s)A}  Q^{\frac{1}{2}}  f(s)\,ds-\int_0^t e^{(t-s)A}\beta_{y_0,y_1}(s)\,ds. \notag   
		\end{align}
		Let now  $\kappa>0$ be such that  
		\begin{equation}\label{deltabound}
			\sup \limits_{t \in [0,T]} \left \|y(t;y_0, f) \right \|_{L^2(\T;\R)} +\sup \limits_{t \in [0,T]} \|\alpha_{y_0,y_1}(t) \|_{L^2(\T;\R)} \leq \kappa \, .
		\end{equation}
		Such a $\kappa$ exists by definition of $\alpha_{y_0,y_1}$ and because $y \in C([0,T]; L^2)$ if $f \in C([0,T]; L^2)$.\footnote{To be thorough, note also that $\kappa$ can be chosen independently of $\varepsilon \in (0,1)$, as  
			\begin{equation}\label{betaf}
				\begin{split}
					\left \|\left( -B \right )^{-\frac{\gamma}{2}} f(t) \right \|_{L^2(\T;\R)} \leq \left \|\beta_{y_0,y_1}(t) \right \|_{L^2(\T;\R)}+1,\quad \text{for all $t \in [0,T]$},
				\end{split}
			\end{equation} \label{footnote 1} }  
		From \eqref{diseguaglianza unmezzo} and the $L^2$-contraction property of $\{e^{tA}\}_{t \geq 0} $ we obtain  
		\begin{equation}
			\begin{split}
				\| w(t)\|_{L^2(\T;\R)} & \leq C \int_0^t (t-s)^{-\frac{1}{2}}\, \left  \|V^{'}w(s) \right \|_{L^2(\T;\R)}\,ds \\
				& +C \int_0^t (t-s)^{-\frac{1}{2}} \left\|( F^{'}*w )(s) \alpha_{y_0,y_1}(s) \right \|_{L^2(\T;\R)}\,ds \\
				& +C \int_0^t (t-s)^{-\frac{1}{2}} \left\|( F^{'}*y )(s) w(s) \right \|_{L^2(\T;\R)}\,ds \\
				& + \int_0^t \left \| \beta_{y_0,y_1}(s)- Q^{\frac{1}{2}}  f(s)\right \|_{L^2(\T;\R)} \,ds, \notag
			\end{split}
		\end{equation}
		where   $C>0$ is a generic constant (independent of  $\varepsilon \in (0,1)$ but possibly dependent on $T$). Since for every $s \in [0,t]$ we have  
		\begin{equation}
			\begin{split}\label{manipulation}
					\left\|( F^{'}*w )(s) \alpha_{y_0,y_1}(s) \right \|_{L^2(\T;\R)}& \leq \left\|( F^{'}*w )(s)\right\|_{L^{\infty}(\T;\R)}  \left\| \alpha_{y_0,y_1}(s) \right \|_{L^2(\T;\R)}
				\\ 
				& \leq \left\| F'\right\|_{L^{\infty}(\T;\R)}\left\|w(s)\right\|_{L^2(\T;\R)}  \left\| \alpha_{y_0,y_1}(s) \right \|_{L^2(\T;\R)}
			\end{split}
		\end{equation}
		(and acting similarly on the term $(F'*y)w$),  from \eqref{deltabound} we deduce
		\begin{equation}
			\begin{split}
				\left \| w(t) \right \|_{L^2(\T;\R)} & \leq \bar{C} \left ( 1+\kappa \right ) \int_0^t \left( t-s \right)^{-\frac{1}{2}}\left \|w(s) \right \|_{L^2(\T;\R)} \,ds \\
				& + \int_0^t  \left \| \beta_{y_0,y_1}(s)- Q^{\frac{1}{2}}  f(s)\right \|_{L^2(\T;\R)}\,ds, \notag
			\end{split}
		\end{equation}
		where $\bar{C}$ is a positive constant depending on $\|V^{'}\|_{L^{\infty}(\T;\R)}$,  $\|F^{'}\|_{L^{\infty}(\T;\R)}$ and $T>0$.		The conclusion now follows from \eqref{distance betaf} and from  the generalized Gronwall inequality (see Lemma \ref{Lemma di Gronwall generalizzato}). 
		%	\begin{equation}
		%		\begin{split}
		%			\|w(t)\|_{L^2(\T;\R)} & \leq %\hat{C} \int \limits_0^T \left \| %\beta_{y_0,y_1}(s)-\left( -B \right %)^{-\frac{\gamma}{2}} f(s)\right %\|_{L^2(\T;\R)} \,ds \leq \hat{C} T %\varepsilon. \notag
		%			\end{split}
		%	\end{equation}
		%Since $\varepsilon$ can be made arbitrarily small and $\hat{C}$ is independent of $\varepsilon \in (0,1)$ we deduce \eqref{distance L2 norm} which implies the assertion.
	\end{proof}
	\begin{remark}
		By slightly modifying the above proof, it is easy to see that one can always take the control $f$ in $C([0,T]; C^{\infty}(\T;\R))$ -- just take $f$ smooth such that \eqref{distance betaf} holds,  then consider the $L^{\infty}$ bound in  \eqref{deltabound} and finally adapt the manipulations in \eqref{manipulation} accordingly. However, we don't need this regularity in our proofs so we simply consider $f$ in $C([0,T]; L^2(\T;\R))$. 
	\end{remark}
	\begin{proposition}\label{approximate proposition}
		Let $Q$ satisfy \eqref{Qactsonbasis} and \eqref{eqn:Qtrace-class}. Then the control system \eqref{controlled system} is approximately controllable in $L^2(\T;\R)$ via an $L^{2} \left([0,T];L^2(\T;\R) \right)$-control. That is, for any  $T>0$, $\varepsilon>0$ and $y_0,y_1 \in L^2(\T;\R)$ there exists $f \in L^{2} \left([0,T];L^2(\T;\R) \right)$ such that 
		\begin{equation}
			\|y(T;y_0,f) - y_1 \|_{L^2(\T;\R)} \leq \varepsilon.\notag
		\end{equation}
	\end{proposition}
	%	Furthermore, the control $f$ can be taken so that it is $dt$-a.e $L^2$-continuous on $[0,T]$.{\color{red}?????}
	\begin{proof}
		From the previous lemma we know that \eqref{controlled system} is approximately controllable provided the initial datum and the endpoint are in $C^2(\T;\R)$. Moreover we recall that, when $f=0$, \eqref{controlled system} has smoothing properties (see e.g. \cite[Theorem 2.2]{pavliotis}); namely, if $y_0 \in L^2(\T;\R)$ then $y(t;y_0,0) \in C^{1}\left( (0,+\infty);C^2(\T;\R) \right)$ for any $t>0$. With this premise, let $\bar{y}_1$ be any point in $C^2(\T;\R)$ and    $\bar{t} > 0$ be any positive time. Using the mentioned smoothing properties and  Lemma \ref{approxiamte lemma},  there exists a function $\bar f \in L^{2} \left([\bar t,T];L^2(\T;\R) \right)$ such that the control 
		\begin{equation}\label{control f}
			f(t):=
			\begin{cases}
				0,\, & t \in [0,\bar{t}) , \\
				\bar{f}(t),\, \quad & t \in [\bar{t},T] ,
			\end{cases}
		\end{equation}
		will drive system \eqref{controlled system} from $y_0$ to an arbitrarily small neighbourhood of $\bar{y}_1$ in time $T$ (more precisely it will drive \eqref{controlled system} first from $y_0$ to a point in $C^2$ and then from such a point to an arbitrarily small neighbourhood of $\bar{y}_1$). That is,   for any  $\bar{y}_1 \in C^2(\T;\R)$ and $\varepsilon >0$ there exists $f  \in L^{2} \left([0,T];L^2(\T;\R) \right)$ such that 
		\begin{equation}
			\left \| y(T;y_0,f)-\bar{y_1} \right \|_{L^2(\T;\R)} \leq \varepsilon. \notag
		\end{equation}
		Since $C^2(\T;\R)$ is dense in $L^2(\T;\R)$, to conclude the argument we choose  $\bar{y}_1 \in C^2(\T;\R)$ such that $\left \|\bar{y}_1-y_1 \right \|_{L^2(\T;\R)} \leq \varepsilon$;  then from the triangle inequality we obtain 
		\begin{equation}
			\left \| y(T;y_0,f)-y_1 \right \|_{L^2(\T;\R)} 
			\leq \left \| y(T,y_0;f)-\bar{y_1} \right \|_{L^2(\T;\R)}+\left \| \bar{y_1}-y_1 \right \|_{L^2(\T;\R)} \leq 2\varepsilon. \nonumber 
		\end{equation}
	\end{proof}
	Before proving the irreducibility of the semigroup $\{ \cP_t \}_{t \geq 0}$ we recall the following elementary fact. 
	%%%%%%%%%%%%%%%%%%%
	%%%%%%%%%%%%%%%%%%%%%%
	%%%%%%%%%%%%%%

	%%%%%%%%%%%%%%%%%%%%%%%%%%%%%%%%%
	%%%%%%%%%%%%%%%%%%%%%%%%%%%%%%%

	\begin{lemma}\label{lemma irriducibilità} 
		Let $Q$ satisfy \eqref{Qactsonbasis} and \eqref{eqn:Qtrace-class}. For  
		$f \in L^2([0,T];L^2(\T;\R))$,  define 
		\begin{equation} \label{deterministic convolution}
			f_A(t):=\int_0^t e^{(t-s)A} {Q}^{\frac{1}{2}}  f(s)\,ds,\quad  t \in [0,T].
		\end{equation}
		Then $f_A \in C([0,T];H^1(\T;\R))$ and its weak derivative is given by 
		\begin{equation}\label{deterministic derivative}
			\pa_x f_A(t)= \sum \limits_{k \in \Z} |k| \lambda_k \left( \int_0^t e^{-(t-s)k^2}f_k(s)\,ds \right )\,e_{-k},\quad  t \in [0,T]. 
		\end{equation}
		where $f_k(s):=\langle f(s),e_k \rangle_{L^2(\T;\R)}$, $s \in [0,T]$, $k \in \Z$. 
		As a consequence, the 
		stochastic convolution $W_A$ belongs to $C([0,T];H^1(\T;\R))$. 
		
		Moreover,  for any function $f$ in $L^2\left( [0,T]; L^2(\T;\R) \right)$ and  for any $\varepsilon >0$ the following holds
		\begin{equation}\label{che ne se}
			\mP \left( \sup \limits_{t \in [0,T]} \| W_A(t)-f_A(t) \|_{H^1(\T;\R)} \leq \varepsilon \right) >0 \, . 
		\end{equation}
	\end{lemma}
	We give a brief proof of the above lemma in Appendix \ref{appendix: proofs of sec 7}.

	\begin{theorem}\label{irreducibility of the semigroup}
		Let $Q$ satisfy \eqref{Qactsonbasis} and \eqref{eqn:Qtrace-class}. Then, the semigroup $\{\cP_t\}_{t \geq 0}$ generated by SPDE \eqref{spde Q} is irreducible in $L^2(\T;\R)$, that is, 
		for any  $u_0 \in L^2(\T;\R)$,  
		$\left (\cP_T1_{G} \right )(u_0) > 0$ for all open sets $G \subset L^2(\T;\R)$ and every $T>0$. 
	\end{theorem}
	\begin{proof}
		Let $u=u(t;u_0)$, $t\geq 0$, be the solution to \eqref{spde Q} with initial condition $u_0 \in L^2(\T;\R)$. Throughout the proof $u_0$ will be fixed but arbitrary and we don't repeat this in every statement. Proving the assertion is equivalent to showing that for any  $T>0$,  $u_0 \in L^2(\T;\R)$ and for all $y_1$ in a dense set of $L^2(\T;\R)$,  the following holds 
		\begin{equation} \label{obiettivo}
			\mP \left( \left \| u(T;u_0)-y_1 \right \|_{L^2(\T;\R)} \leq \varepsilon \right) > 0, \quad {\mbox{for every }} \epsilon >0\,.
		\end{equation}
		To this end, let us fix $T>0$ for the remaining part of the proof and  note that due to Proposition \ref{approximate proposition} we know that the set of reachable points at time $T$ of the control system \eqref{controlled system} is a dense subset of $L^2(\T;\R)$. In other words, the set 
		\begin{equation}
			\mathfrak{F}:= \left \{ y(T; u_0,f) \,|\, f \in L^{2} \left ([0,T];L^2(\T;\R) \right ) \right \},
		\end{equation}
		is dense in $L^2(\T;\R)$.
		Hence, it suffices to prove that \eqref{obiettivo} holds for any $y_1 \in \mathfrak{F}$;  in particular  we will show that
		%if the following held for all $\varepsilon>0$ and for any given $u_0 \in L^2(\T;\R)$ and $f \in L^{\infty} \left ([0,T];L^2(\T;\R) \right )$
		\begin{equation} \label{obiettivo 2}
			\mP \left ( \sup \limits_{t \in [0,T]} \|u(t;u_0)-y(t;u_0,f)\|_{L^2(\T;\R)} \leq \varepsilon \right ) >0,
		\end{equation}
		for every $\epsilon>0$ and $f \in L^2\left ([0,T];L^2(\T;\R)\right )$ as in the above. To do so, we follow the method adopted in Lemma \ref{approxiamte lemma} and we start by estimating the difference in the $L^2(\T;\R)$-norm between $u(t;u_0)$ and $y(t;u_0,f)$. In turn, because of the non-linearity, this requires a bound similar to \eqref{deltabound}. Since in this case $u(t;u_0)$ is random (hence, the analogous bound to \eqref{deltabound} will not hold for every $\omega$) we proceed as follows. Using the estimates \eqref{L2 norm WA} (this estimate is for $v$, to get the one for $u$ it suffices to recall that $u=v+W_A$) and \eqref{L2 yf norm}, we can see that there exists $\bar{\kappa}$ such that  
		\begin{equation} \label{bound on eta}
			\sup \limits_{t \in [0,T]} \left \|u(t) \right \|_{L^2(\T;\R)}+\sup \limits_{t \in [0,T]} \left \|y(t) \right \|_{L^2(\T;\R)} \leq \bar{\kappa},\quad\text{for all $\omega \in S_{\varepsilon}$ and $ \varepsilon \in (0,1)$, }
		\end{equation}
		where the set $S_{\varepsilon}$ is defined for all $\varepsilon \in (0,1)$ as 
		\begin{eqnarray}\label{threshold constant}
			&& S_{\varepsilon}:=\left \{ \omega \in \Omega\,:\, \sup \limits_{t \in [0,T]} \|W_A(t)\|_{H^1(\T;\R)} \leq \hat{\kappa}_{\varepsilon}\right \},\quad \text{where} \label{threshold set}\\
			&& \hat{\kappa}_{\varepsilon}:=\sup \limits_{t \in [0,T]} \left \| f_A(t) \right \|_{H^1(\T;\R)}+\varepsilon . 
		\end{eqnarray}
		Similarly to footnote \ref{footnote 1}, $\bar{\kappa}$ can be chosen independently of $\varepsilon \in (0,1)$.
		In what follows, we write in short $u(t)$ in place of $u(t;u_0)$ and $y(t)$ in place of $y(t,u_0;f)$. With this notation in mind, from the mild formulation of $u$ and $y$ we have that
		\begin{equation}
			\begin{split}
				u(t)-y(t) & =\int_0^t e^{(t-s)A} \pa_x \left[ V^{'} (u(s)-y(s)) \right ]\,ds + W_A(t)-f_A(t) \\
				& + \int_0^t e^{(t-s)A} \pa_x \left [ \left (F^{'} * u \right )(s) u(s)- \left( F^{'} * y \right )(s) y(s) \right ]\,ds, \notag
			\end{split}
		\end{equation}
		where the above equality holds for all $t \in [0,T]$, $\mP$-a.s.
		Now, we first restrict to the realizations $\omega \in S_{\varepsilon}$ (see \eqref{threshold set} for its definition). Hence, from \eqref{bound on eta} and \eqref{diseguaglianza unmezzo} we obtain 
		\begin{equation}
			\begin{split}
				\|u(t)-y(t)\|_{L^2(\T;\R)} & \leq C \|V^{'}\|_{L^{\infty}(\T;\R)} \int_0^t (t-s)^{-\frac{1}{2}} \|   u(s)-y(s)\|_{L^2(\T;\R)}\,ds \\
				& +C\|F^{'}\|_{L^{\infty}(\T;\R)} \bar{\kappa} \int_0^t (t-s)^{-\frac{1}{2}} \|u(s)-y(s)\|_{L^2(\T;\R)}\,ds \\
				&+ \sup \limits_{t \in [0,T]} \|W_A(t)-f_A(t) \|_{L^2(\T;\R)},\notag
			\end{split}
		\end{equation}
		where the above inequality holds for all $t \in [0,T]$, $\omega \in S_{\varepsilon}$ and for all $\varepsilon \in (0,1)$.
		Now, by applying the generalized Gronwall's inequality we obtain that there exists a deterministic constant $C_{\bar{\kappa},T} \in \R_{+}$ such that 
		\begin{equation} \label{bound on difference}
			\begin{split}
				\|u(t)-y(t)\|_{L^2(\T;\R)} \leq C_{\bar{\kappa},T} \sup \limits_{t \in [0,T]} \|W_A(t)-f_A(t) \|_{L^2(\T;\R)},
			\end{split}
		\end{equation}
		and the above inequality holds for all $t \in [0,T]$, $\omega \in S_{\varepsilon}$ and for all $\varepsilon \in (0,1)$. 
		The proof is concluded by using \eqref{bound on difference} and noting that, since $\|W_A(t)\|_{H^1(\T;\R)}-\|f_A(t)\|_{H^1(\T;\R)} \leq \|W_A(t)-f_A(t)\|_{H^1(\T;\R)}$ for all $t \in [0,T]$, the sets
		\begin{eqnarray} \notag
			&& S_{\varepsilon}^{'}:=S_{\varepsilon} \cap \left \{ \omega \,:\, \sup \limits_{T \in [0,T]} \left \| W_A(t)-f_A(t) \right \|_{H^1(\T;\R)} < \varepsilon \right \} \\ \notag
			&& S_{\varepsilon}^{''}:=\left \{ \omega \,:\, \sup \limits_{T \in [0,T]} \left \| W_A(t)-f_A(t) \right \|_{H^1(\T;\R)} < \varepsilon \right \}
		\end{eqnarray}
		have the same probability, which we know to be strictly positive due to \eqref{che ne se}, i.e. $\mP\left ( S_{\varepsilon}^{'} \right )=\mP \left ( S_{\varepsilon}^{''} \right )>0$ for all $\varepsilon \in (0,1)$. 
	\end{proof}

	%%%%%%%%%%%%%%%%%%
	%%%%%%%%%%%%%%%%%%%%%
	%%%%%%%%%%%%%%%%%%%%%%%%%%%%%%%%%%%%%%%%

	\section{Proof of part ii) of Theorem \ref{theorem:existence and uniqueness inv measure}: Strong Feller Property}\label{sec:sec7}
	%This semigroup should be denoted $\cP_t^{\gamma}$ as it depends on $\gamma$ through the covariance operator $Q$, but we omit this dependence in the notation for simplicity. 
	
	From Proposition \ref{stime sistema deterministico}, we readily deduce that the semigroup $\left\{ \cP_t \right \}_{t \geq 0}$ associated with \eqref{spde Q} (defined in \eqref{semigroup non linear spde}) is a Feller semigroup on $L^2(\T;\R)$. 
	The purpose of this section is to prove that the semigroup $\cP_t$ is strong Feller as well, i.e. to prove Theorem \ref{thm:strong feller for fully nonlinear semigroup}  below. 
	We recall that throughout this section we work under Assumption \ref{ass:QstrongFeller}; we explain where this assumption is used in Note \ref{nota doppio pallino} below and in the comments before Lemma \ref{interpolation formula}.

	The proof of Theorem \ref{thm:strong feller for fully nonlinear semigroup}  requires showing the Strong Feller property for a (class of) SPDE with Lipshitz non-linearity, see \eqref{damped spde} and  \eqref{strong feller for PtR}  below. This result  is used in the proof of the main theorem and then proved in  Subsection \ref{sec: sec 7.2} - more comments on this matter can be found at the beginning of that subsection as well. 
	
	\begin{theorem}\label{thm:strong feller for fully nonlinear semigroup}
		Let Assumption \ref{ass:QstrongFeller} hold. Then the semigroup $\{\cP_t\}_{t \geq 0}$  generated by SPDE \eqref{spde Q}  is strong Feller in $L^2$, i.e. $\cP_t$ maps $\cB_b(L^2(\T;\R);\R)$ into $C_b(L^2(\T;\R);\R)$, for any $t > 0$.
	\end{theorem} 
	\begin{proof}[Proof of Theorem \ref{thm:strong feller for fully nonlinear semigroup}]	
		The strategy is inspired by  \cite[Chap.14]{ergodicity}.   Let  $\{\xi_{R}\}_{R>0} \subset C^{\infty}(\R;\R)$ be a family of smooth cutoff functions such that
		\begin{equation}\label{smooth indicator functions}
			\xi_R(x)=
			\begin{cases}
				1,\quad  |x| \leq R, & \\
				0,\quad  |x| \geq R+1  . &
			\end{cases}
		\end{equation} 
		We consider a `damped' version of the SPDE \eqref{spde Q} where we replace the non-linearity $(F'\ast u)u$ with the truncated operator
		\begin{equation} \label{damped non linearity}
			\mathcal{M}_R(u):=\left ( F^{'}*u \right )\,u\,\xi_R \left ( \|u\|_{L^2(\T;\R)}^2 \right ), \quad u \in L^2(\T;\R). 
		\end{equation}
		Namely, we consider the family of SPDEs
		\begin{equation}\label{damped spde}
			\begin{cases}
				\pa_t u_R=Au_R+\pa_x \left[ V^{'}u_R+\mathcal{M}_R(u_R) \right ]+{Q}^{\frac{1}{2}} \pa_t W,\quad (0,T) \times \T, & \\
				u_R(t,0)=u_R(t,2\pi),\quad  t \in [0,T], & \\
				u_R(0,x)=u_0(x), \quad  x \in \T, &
			\end{cases}
		\end{equation}
		with $R>0$. The nonlinear term $\mathcal{M}_R(u)$ is globally Lipschitz-continuous, as opposed to the non-linearity in the original system \eqref{spde Q}. 
		Hence a standard application of the Banach fixed point theorem gives the well-posedness in mild sense (in $L^2$) of \eqref{damped spde} for any initial datum  $u_0 \in L^2$. Therefore we can define the (Feller) semigroup associated to \eqref{damped spde}, and we denote it by $\cP_t^R$. Because the non-linearity in \eqref{damped spde} is globally Lipschitz, it is easier to prove the Strong Feller property for $\cP_t^R$  rather than for $\cP_t$ directly. In particular, if we prove the following two facts 
		\begin{itemize}
			\item[(i)] 	 for each $R>0$, the semigroup  $\{\cP^R_t\}_{R>0}$  is Strong Feller in $L^2(\T;\R)$; 
			\item[(ii)] for any $t>0$ and any $\psi \in \mathcal{B}_b\left( L^2(\T;\R);\R \right)$, $\cP_t^R\psi$ converges to $\cP_t\psi$, as $R \to +\infty$, locally uniformly in $L^2(\T;\R)$ (i.e. uniformly on any bounded subset $X \subset L^2(\T;\R)$); 
		\end{itemize}
		then $\{\cP_t\}_{t \geq 0}$ is Strong Feller in $L^2(\T;\R)$. Indeed, let $\{u_0^n\}_{n \in \N}, u_0 \subset L^2(\T;\R)$ be such that $u_0^n \stackrel{n \to +\infty}{\longrightarrow u_0}$ in $L^2(\T;\R)$ and let $X$ be a bounded subset of $L^2(\T;\R)$ such that $\{u_0^n\}_{n \in \N}, u_0 \subset X$. First, we write
		\begin{equation}
			\begin{split}
				\left | (\cP_t\psi)(u_0^n)-\left (\cP_t \psi\right) (u_0)\right | & \leq \left | (\cP_t\psi)(u_0^n)-\left (\cP_t^R \psi \right) (u_0^n)\right |  \\
				& + \left | (\cP_t^R\psi)(u_0^n)-\left (\cP_t^R \psi \right) (u_0) \right | \\
				& + \left | (\cP_t^R\psi)(u_0)-\left (\cP_t \psi \right)(u_0) \right | \\
				& \leq 2\sup \limits_{u_0 \in X} \left | (\cP_t\psi)(u_0)-\left (\cP_t^R \psi \right) (u_0)\right |  \\
				& + \left | (\cP_t^R\psi)(u_0^n)-\left (\cP_t^R \psi \right) (u_0) \right |. \notag 
			\end{split}
		\end{equation} 
		Letting $n \to +\infty$, by (i) we obtain 
		\begin{equation}
			\limsup \limits_{n \to +\infty} \left | (\cP_t\psi)(u_0^n)-\left (\cP_t \psi\right) (u_0)\right | \leq 2\sup \limits_{u_0 \in X} \left | (\cP_t\psi)(u_0)-\left (\cP_t^R \psi \right) (u_0)\right | \, .\notag 
		\end{equation}
		Hence, the conclusion follows by letting $R \to +\infty$ and using (ii). 
		
		Statement (i), i.e. the strong Feller property for $\{\cP^R_t\}_{t \geq 0} $, is proved in Subsection \ref{sec: sec 7.2} under Assumption \ref{ass:QstrongFeller}. More precisely, (i) follows directly from the bound
		\begin{equation} \label{strong feller for PtR}
			\left |\left ( \cP_t^{R}\psi \right )(u_0) - \left ( \cP_t^{R}\psi \right )(v_0) \right |^2 \leq \frac{(1+t)^\gamma}{t^{1+\gamma}}e^{CL^2_R t}\,\|u_0-v_0\|_{L^2(\T;\R)}^2\, ,
		\end{equation}
		where $C>0$ is some constant (independent of $R$) and  $L_R$ is the Lipschitz constant associated to $V'+\mathcal{M}_R$. The above bound is proved in Proposition \ref{prop:strong feller Lipshitz nonlinearities}.  
		
		To prove Statement (ii), i.e. to show that the following limit holds
		\begin{equation}\label{locally uniformly}
			\sup \limits_{u_0 \in X} \left |\cP_t\psi(u_0)-\cP_t^R\psi(u_0) \right | \to 0, \qquad \mbox{as } R\to +\infty, 
		\end{equation}
		 for any given bounded set $X \subset L^2(\T;\R)$ and any fixed $t \in [0,T]$, we introduce the family of stopping times $\{ \tau_{u_0}^R \}_{R>0}$ given by  
		$$ \tau_{u_0}^R:= \inf \{ t \geq 0 \,:\, \|u(t,u_0)\|_{L^2(\T;\R)} \geq R \} \wedge T.$$
		Let $X$ be a bounded subset of $L^2(\T;\R)$ and fix $t\in[0,T]$. To prove \eqref{locally uniformly}, it is enough to show that
		\begin{equation}\label{stopping time locally uniformly}
			\sup \limits_{u_0 \in X} \mP ( \tau_{u_0}^R < t ) \to 0, \quad \mbox{as } \, R\to+\infty.
		\end{equation}
		Indeed,
		since $u(t;u_0)=u^R(t;u_0)$ for $t\in [0,\tau^R_{u_0}]$, we have
		\begin{equation}
			\begin{split}
				\cP_t\psi(u_0)-\cP_t^R\psi(u_0) =\,& \mE \left( \left( \psi(u(t;u_0))- \psi(u^R(t;u_0) \right) 1_{\{ \tau_{u_0}^R > t \}} \right)\\
				+ \, &\mE \left( \left( \psi(u(t;u_0))- \psi(u^R(t;u_0) \right) 1_{\{ \tau_{u_0}^R < t \}} \right)\\
				= \,& \mE \left( \left( \psi(u(t;u_0))- \psi(u^R(t;u_0) \right) 1_{\{ \tau_{u_0}^R < t \}} \right), \notag
			\end{split}
		\end{equation}
		and the RHS of the above can be bounded by
		\begin{align*}
			\left | \mE \left( \left( \psi(u(t;u_0))- \psi(u^R(t;u_0) \right) 1_{\{ \tau_{u_0}^R < t \}} \right)\right | 
			\leq 2 \| \psi \|_{L^{\infty}(L^2(\T;\R);\R)))} \mP ( \tau_{u_0}^R < t ).
		\end{align*}
		Hence, to conclude, we need to show \eqref{stopping time locally uniformly}. By Proposition \ref{stime sistema deterministico}, we know that there exists an increasing a.s.~continuous random function $t \in [0,T] \to C(t,X) \in \R_{+}$ such that 
		\begin{equation*}
			\left \| u(t; u_0) \right \|_{L^2(\T;\R)} \leq C(t,X),\,\text{for all $t \in [0,T]$ and $u_0 \in X$, $\mP$-a.s},
		\end{equation*}
		uniformly in $R>0$. From this we deduce that \eqref{stopping time locally uniformly} holds and this concludes the proof.
	\end{proof}
	
	Let us now introduce some notation that will be needed in Subsection \ref{sec: sec 7.2}.  Given two Banach spaces $\mX$ and $\mY$ endowed with the norm $\| \cdot \|_{\mX}$ and $\| \cdot \|_{\mY}$ respectively, we denote by $\fL(\mX;\mY)$ the Banach space of linear bounded operators from $\mX$ to $\mY$ endowed with the norm 
	$$\| J \|_{\fL(\mX;\mY)} := \sup \limits_{\|a\|_{\mX} \leq 1} \| Ja\|_{\mY},\quad \text{for any given $J \in \fL(\mX;\mY)$}.$$
	We use the shorthand notation $\fL \left( L^2(\T;\R)\right)=\fL \left( L^2(\T;\R) ; L^2(\T;\R) \right)$ to denote the Banach space of linear bounded operators from $L^2(\T;\R)$ into itself.
	We further introduce the Banach spaces $\cC_T^2$ defined as the set of adapted square-integrable processes with values in $C \left( [0,T] ; L^2(\T;\R)\right )$. 
	We recall  that a map $J:\mX \to \mY$ is {\em Fr\'echet differentiable} at $a_0 \in \mX$ if there exists a bounded linear operator $D_aJ(a_0) \in \fL(\mX;\mY)$ such that
	\begin{equation}
		\lim \limits_{\|h\|_{\mX} \to 0} \frac{\left \| J(a_0+h)-J(a_0)-D_aJ(a_0)h\right \|_{\mY}}{\|h\|_{\mX}}=0,\notag
	\end{equation}
	where $D_aJ(a_0)h$ denotes the operator $D_aJ(a_0)$ applied to $h$. Furthermore, if $J$ is Fr\'echet differentiable at every point of $\mX$ then we simply say that $J$ is Fr\'echet differentiable in $\mX$. Let us clarify that, in what follows, while $D_a$ denotes the Fr\'echet derivative with respect to an element $a$ in some appropriate infinite dimensional space, $\pa_x$ denotes the derivative with respect to $x \in\mathbb T$. 
	
	\begin{remark} See also \cite[Example 3.2.4]{Ricciflow}.\label{gradient of gateaux derivative}
		Consider the special case $\mX=H$ and $\mY=\R$, where $\left ( H, \langle \, , \,\rangle_H \right)$ is a Hilbert space, and let $J:H \to \R$ be a Fr\'echet differentiable map, with Fr\'echet derivative $D_a J\in\fL(H;\R)$. From the Riesz representation theorem  $\fL(H;\R)\simeq H$, hence there exists a unique point $\nabla_a J(a) \in H$ such that 
		\begin{equation}
			D_a J(a)h=\langle \nabla_a J(a),h \rangle_H,\quad\text{for all $h \in H$.}\notag
		\end{equation}
		Then, for any given $h \in H$, with slight abuse of notation we will write $\left \langle D_a J(a),h \right \rangle_H $ instead of $D_a J(a)h$. 
	\end{remark}

	\subsection{Strong Feller property for gradient form Lipschitz non-linearities}\label{sec: sec 7.2}
	In this subsection we consider SPDEs of the form
	\begin{equation}\label{lipschitz spde}
		\begin{cases}
			\pa_t u=Au+\pa_x \left[ \mathcal{F}(u) \right ]+ {Q}^{\frac{1}{2}}\pa_t W,\quad (0,T] \times \T, & \\
			u(t,0)=u(t,2\pi),\quad  t \in [0,T], & \\
			u(0,x)=u_0(x), x \in \T, &
		\end{cases}
	\end{equation}
	where $\mathcal{F} \in C_b^2 \left( L^2(\T;\R);  L^2(\T;\R)\right)$\footnote{We recall that $C_b^2 \left( L^2(\T;\R);  L^2(\T;\R)\right)$ is the space consisting of twice Fr\'echet differentiable functions from $L^2(\T;\R)$ to $L^2(\T;\R)$ with continuous and bounded first and second Fr\'echet derivative.} and $T>0$, and we show that the semigroup $\cP_t^{\mathcal F}$ associated with the above evolution is Strong Feller (see Proposition \ref{non so che scrivere}). Observe that the SPDE \eqref{damped spde} is a particular case of \eqref{lipschitz spde}, when $\mathcal{F}(u)= V'u+\mathcal{M}_R(u)$, hence the results of this section imply the Strong Feller property for the semigroup $\cP_t^R$. Moreover, note that since $\mathcal{F} \in C_b^2 \left( L^2(\T;\R);  L^2(\T;\R)\right)$, the functions $\cF:L^2(\T;\R) \to L^2(\T;\R)$ and $D_{u}\cF:L^2(\T;\R) \to \fL(L^2(\T;\R))$ are both globally Lipschitz continuous (see \cite[Proposition 3.2.7]{Ricciflow}); so the non-linearity in \eqref{lipschitz spde} is the gradient of a globally Lipschitz continuous functional, hence the name of this subsection. We denote by $L_{\cF}$ and $L_{D\cF}$ the Lipschitz constants of $\mathcal{F}$ and $D_{u}\cF$, respectively. 
	
	To show that the semigroup $\cP_t^{\mathcal F}$ is strong Feller we adapt the methods in \cite[Chap.4]{cerrai}, which have been developed to prove smoothing properties of SPDEs with globally Lipschitz non-linearities. In our setting we can't apply the results of \cite[Chap.4]{cerrai} directly, as the type of non-linearity in \eqref{damped spde} is different from the one in \cite{cerrai}. Indeed, in \cite[Chapter 4]{cerrai} the non-linearity is allowed to depend on $x$ and $u$ but not on $\pa_x u $, which is the case here.  However, the general approach of \cite[Chap.4]{cerrai} can still be adapted to our case. We outline the strategy to prove that $\cP_t^{\mathcal F}$ is strong Feller in Note \ref{nota doppio pallino} below.
	
	We recall that $u(t)$, $t \in [0,T]$, is called a mild solution to \eqref{lipschitz spde} if $u(t)$ is a continuous $L^2(\T;\R)$-valued stochastic process such that
	\begin{equation}
		u(t) =e^{tA}u_0+P[\cF(u)](t)+W_{A}(t),\quad \text{for all $t \in [0,T]$, $\mathbb{P}$-a.s.}, \notag
	\end{equation}
	where $P$ is the operator defined in \eqref{operatore P}. We emphasize that only throughout this subsection we denote by $u(t)$ (or $u(t; u_0)$ to stress dependence on initial conditions) the solution to \eqref{lipschitz spde}, rather than the solution to \eqref{spde Q}. 
	
	Consider the map $\cI: L^2(\T;\R) \times \cC_T^2 \to \cC_T^2$ defined as 
	\begin{equation}\label{fixed point Lipschitz}
		\cI(u_0,u)(t):= e^{tA}u_0+P\left [\cF(u) \right](t)+W_A(t),\quad  t \in [0,T].
	\end{equation}
	Since $\cF(u(s)) \in L^2(\T;\R)$ for all $s \in [0,T]$, $\mP$-a.s., we can apply \eqref{diseguaglianza unmezzo}. Hence, in a similar manner of proof of Proposition \ref{local existence mild solution}, we obtain
	%Using again  \eqref{diseguaglianza unmezzo} (which can be applied because $\cF(u(s)) \in L^2(\T;\R)$ for all $s \in [0,T]$, $\mP$-a.s.) in a similar manner to what we have done in the proof of Proposition \ref{local existence mild solution}, we have
	\begin{equation}
		%\begin{split}
		\| \cI(u_0,u)(t)-\cI(u_0,v)(t) \|_{L^2(\T;\R)}^2 
		%&=\left \| P \left [ \cF(u)\right ](t)-P\left [ \cF(v)\right ](t) \right \|_{L^2(\T;\R)}^2 \\
		%& \leq C^2L_{\cF}^2 \left |\int_0^t (t-s)^{-\frac{1}{2}} \|u(s) - v(s)\|_{L^2(\T;\R)}\,ds \right |^2 \\
		\leq C^2L_{\cF}^2T \sup \limits_{t \in [0,T]} \|u-v\|_{L^2(\T;\R)}^2, \quad
		\text{$\mP$-a.s.},\notag
		%	\end{split}
	\end{equation}
	for some constant $C>0$.
	By taking the supremum over $t \in [0,T]$ and then the expectation on both sides of the above, we have 
	$$\| \cI(u_0,u)-\cI(u_0,v) \|_{\cC_T^2}^2 \leq C^2L_{\cF}^2T \left \|u-v \right \|_{\cC_T^2}^2 \, ,$$
	from which (local and then global) in time well-posedness of \eqref{lipschitz spde} follows.   
	
	We can then define the semigroup $\{\cP_t^{\cF}\}_{t \geq 0}$ associated to SPDE \eqref{lipschitz spde}, namely 
	\begin{equation}
		\left (\cP_t^{\cF}\psi \right )(u_0):= \mE \left ( \psi(u(t;u_0)) \right ), \qquad  \psi \in \cB_b(L^2(\T;\R);\R) \,,
	\end{equation}
	where we recall that throughout this section $u(t)$ denotes the solution to \eqref{lipschitz spde}. 
	
	\begin{remark}\label{nota doppio pallino}
		The strategy to show that  $\{\cP_t^{\cF}\}_{t \geq 0}$  is Strong Feller is as follows. By definition, we want to show that $\cP_t^{\cF}\psi$ is continuous if  $\psi$ is bounded and  measurable. We will in fact show that $\cP_t^{\mathcal F}\psi$ is Lipschitz if $\psi$ is bounded and measurable (see Proposition \ref{prop:strong feller Lipshitz nonlinearities}). To prove the Lipschitzianity of $\cP_t^{\mathcal F}\psi$, we will find bounds on the Fr\'echet derivative  $D_{u_0}(\cP_t^{\mathcal F}\psi)(u_0)$  (see Proposition \ref{stima derivata semigruppo}). In turn, in order to find such bounds we will use a Bismut-Elworthy-Li type of formula, which is a representation formula for $D_{u_0}(\cP_t^{\mathcal F}\psi)(u_0)$, see Proposition \ref{prop bismut}. This representation formula is the reason why we impose condition \eqref{cond_lambda} in Assumption \ref{ass:QstrongFeller}. More comments on this before Lemma \ref{interpolation formula}. 
	\end{remark}
	
	\begin{lemma}[Fr\'echet differentiability of the solution]\label{prop:differentiability of u w.r.t u0}
		Let $u(t;u_0)$ denote the solution of \eqref{lipschitz spde} with initial datum $u_0$ and suppose $Q$ satisfies \eqref{Qactsonbasis}-\eqref{eqn:Qtrace-class}. Then the map $u_0 \in L^2(\T;\R) \mapsto u(t;u_0) \in \cC_T^2$
		is Fr\'echet differentiable \footnote{Note that one can also prove that the map $u_0 \in L^2(\T;\R) \to u(t;u_0) \in \cC_T^2$ is twice Fr\'echet differentiable, but we don't need the second Fr\'echet derivative in what follows.} and, for any  $h \in L^2(\T;\R)$, the directional derivative $\eta_h:=D_{u_0}u(t;u_0)h$ of $u$ in the direction $h$ satisfies the following bounds:
		\begin{equation} \label{inequality u}
			\| \eta_h(t) \|_{L^2(\T;\R)}^2 \leq \|h\|_{L^2(\T;\R)}^2 e^{CL_{\cF}^2t},
		\end{equation}
		and 
		\begin{equation}\label{inequality derivata di u}
			\int_0^t \| \pa_x \eta_h(s) \|_{L^2(\T;\R)}^2\,ds \leq \|h\|_{L^2(\T;\R)}^2 e^{CL_{\cF}^2t},
		\end{equation}
		$\mP$-a.s., for $t \geq 0$. As a consequence, the semigroup $\cP_t^{\cF}$ is   Fr\'echet differentiable with respect to $u_0$ and the  Fr\'echet derivative $D_{u_0}(\cP_t^{\cF}\psi)(u_0)$  satisfies the identity
		\begin{equation}
			\left \langle D_{u_0}\left (\cP_t^{\cF}\psi \right )(u_0), h \right \rangle_{L^2(\T;\R)}=\mE \left ( \left \langle D_{u_0} \psi(u(t;u_0)), \eta_h(t) \right \rangle_{L^2(\T;\R)} \right)\, ,  \notag
		\end{equation}
		for every $\psi \in \mathcal{B}_b(L^2(\T;\R);\R)$.
	\end{lemma}
	\begin{proof}
		See Appendix \ref{appendix: proofs of sec 7}.
	\end{proof}
	
	We clarify that in the above statement $D_{u_0}u(t;u_0)h$ denotes the action of  $D_{u_0}u(t;u_0) \in \mathfrak{L}(L^2(\mathbb T; \R))$ on the element $h$ of $L^2(\mathbb T; \R)$ and note that $\eta_h(t)= \eta_h(t,x)$, as $\eta_h(t) \in L^2(\mathbb T; \R)$ for every $t \geq 0$, but we omit the explicit dependence on $x \in \T$ in the notation, as customary.
	Since   $\cP_t^{\cF}\psi: L^2(\mathbb T; \R) \rightarrow \R$, in the last equality of Lemma \ref{prop:differentiability of u w.r.t u0} we have indicated the action of the Fr\'echet derivative  of $\cP_t^{\cF}$ applied to a vector $h$ with the scalar product of $L^2(\T;\R)$ - see Note \ref{gradient of gateaux derivative}.  
	%%%%%%%%%%%%%%%%%%%%%%%%
	%%%%%%%%%%%%%%%%%%%%%%%%%%%%%%%%%%
	%%%%%%%%%%%%%%%%%%%%%%%%%%%%%%%%%%%%%%%%%%%%5

	%%%%%%%%%%%%%%%%%%%%%%%%%%
	%%%%%%%%%%%%%%%%%%%%%%%%%%%%%%%%%%%%%%
	%%%%%%%%%%%%%%%%%

	%\begin{lemma}\label{lemma bismut}
	%	If we let $\psi \in C_b^2 \left ( L^2(\T;\R);\R \right)$ and let $\{ \cP_t^{\cF}\}_{t \geq 0}$ be the semigroup generated by SPDE \eqref{lipschitz spde} then the following equality holds for all $t \geq 0$, $\mP$-a.s.
	%	\begin{equation}
	%	\psi(u(t;u_0))= \cP_t^{\cF}\psi(u_0)+ \int_0^t \left \langle D_{u_0}\left( P_{t-s}^{\cF} \psi \right )(u(s;u_0)),Q^{\frac{1}{2}} dW(s) \right \rangle_{L^2(\T;\R)} \, ,      
	%		 \notag
	%	\end{equation} 
	%	where $u(t;u_0)$, $t \geq 0$, is the solution to SPDE \eqref{lipschitz spde} with initial datum $u_0 \in L^2(\T;\R)$ and the above integral is defined as 
	%	\begin{align} \notag
	%		\int_0^t & \left \langle D_{u_0}\left( P_{t-s}^{\cF} \psi \right )(u(s;u_0)),Q^{\frac{1}{2}} dW(s) \right \rangle_{L^2(\T;\R)}\\
	%		& :=\sum \limits_{k \in \Z} \lambda_k \int_0^t \left \langle D_{u_0}\left( P_{t-s}^{\cF} \psi \right )(u(s;u_0)),e_k \right \rangle_{L^2(\T;\R)} \,d\beta_s^k,\, t \geq 0. \notag
	%	\end{align}
	%\end{lemma}
	
	Let us introduce the following stochastic process $Z_h(t)$, $t \geq 0$, defined as 
	\begin{equation}\label{Zeta}
		Z_h(t):= \int_0^t \left \langle Q^{-\frac{1}{2}}\eta_h(s),dW(s) \right \rangle_{L^2(\T;\R)},\, t \geq 0,
	\end{equation}
	where $\eta_h=D_{u_0}u(t;u_0)h$ is as in the statement of Lemma \ref{prop:differentiability of u w.r.t u0}. Thanks to Lemma \ref{prop:differentiability of u w.r.t u0} and Lemma \ref{interpolation formula} below, $Z_h(t)$, $t \geq 0$, is well-defined as long as \eqref{cond_lambda} holds. 
	
	\begin{lemma}\label{interpolation formula}
		Let Assumption \ref{ass:QstrongFeller} hold and let $f$ be any function in $H^1(\T;\R)$. 
		Then there exists a constant $c>0$ such that
		\begin{equation}
			\|Q^{-\frac{1}{2}}f\|_{L^2(\T;\R)} \leq c\|f\|_{L^2(\T;\R)}^{1-\gamma} \|f\|_{H^1(\T;\R)}^{\gamma} \,.
		\end{equation} 
	\end{lemma}
	\begin{proof}
		Let  $f_k:=\langle f,e_k \rangle_{L^2(\T;\R)}$, $k \in \Z$, where $\{ e_k \}_{k \in \Z}$ is the orthonormal basis defined in \eqref{fourier}. 
		From the assumption on $Q$ and Parseval's identity, for any $\gamma \in \R$ we have
		\begin{equation}
			\|Q^{-\frac{1}{2}}f\|_{L^2(\T;\R)}^2 = \sum \limits_{k \in \Z} \lambda_k^{-2} f_k^2 = \sum \limits_{k \in \Z} |f_k|^{2(1-\gamma)} \cdot \lambda_k^{-2} |f_k|^{2\gamma}  \,.
		\end{equation}
		Choosing $\gamma<1$, we can apply H\"older's inequality with $p=\frac{1}{1-\gamma}$ and $q=\frac{1}{\gamma}$, and obtain 
		\begin{equation}
			\|Q^{-\frac{1}{2}}f\|_{L^2(\T;\R)}^2  \leq \left ( \sum \limits_{k \in \Z} |f_k|^2 \right)^{1-\gamma}\left ( \sum \limits_{k \in \Z} \lambda_k^{-2/\gamma}|f_k|^2 \right)^{\gamma} \leq c\|f\|_{L^2(\T;\R)}^{2(1-\gamma)} \|f\|_{H^1(\T;\R)}^{2\gamma},\notag
		\end{equation}
		where the last inequality follows by the assumption on $\lambda_k$, provided $\gamma<1$. 
	\end{proof}
	
	\begin{proposition}[Bismut-Elworthy-Li formula]\label{prop bismut}
		Let Assumption \ref{ass:QstrongFeller} hold, and let $\{ \cP_t^{\cF}\}_{t \geq 0}$ be the semigroup generated by the SPDE \eqref{lipschitz spde}. Then, the Fr\'echet derivative of $\{ \cP_t^{\cF}\}_{t \geq 0}$ satisfies
		\begin{equation}
			\begin{split}
				\langle D_{u_0} \left ( \cP_t \psi \right )(u_0),h \rangle_{L^2(\T;\R)}= \frac{1}{t} \mE \left[ \psi \left( u(t;u_0) \right) Z_h(t) \right], \notag
			\end{split}
		\end{equation} 
		for any $t > 0$, $\psi \in C_b^2 \left ( L^2(\T;\R);\R \right)$, and any given $h \in L^2(\T;\R)$.
	\end{proposition}
	\begin{proof}[Sketch of the proof] This follows a standard argument (see e.g. \cite[Proposition 4.4.3]{cerrai}), which we summarise for the reader's convenience. From It\^o's formula, we have  
		\begin{equation}
			\begin{split}
				\psi(u(t;u_0))= \cP_t^{\cF}\psi(u_0)+ \int_0^t \left \langle D_{u_0}\left( \cP_{t-s}^{\cF} \psi \right )(u(s;u_0)),Q^{\frac{1}{2}}dW(s) \right \rangle_{L^2(\T;\R)} ,\notag 
			\end{split}
		\end{equation}
		for all $t \geq 0$, $\mP$-a.s.\footnote{The proof of the above is straightforward in finite dimension, see e.g. \cite[]{CrisanOttobre}, but more delicate in infinite dimension, see \cite[Lemma 4.1]{dapra2}.}  Multiplying both sides of the above equality by $Z_h(t)$ and taking the expectation, we obtain 
		\begin{equation}
			\begin{split}
				\mE \left[ \psi(u(t;u_0)) Z_h(t) \right]
				& =\mE \left( \int_0^t \left \langle D_{u_0}\left( \cP_{t-s}^{\cF} \psi \right )(u(s;u_0)),Q^{\frac{1}{2}} dW(s) \right \rangle_{L^2(\T;\R)} Z_h(t)\right) \\
				& = \mE \left (\int_0^t \left \langle Q^{\frac{1}{2}} D_{u_0}\left(\cP_{t-s}^{\cF} \psi \right )(u(s;u_0)),Q^{-\frac{1}{2}}\eta_h(s) \right \rangle_{L^2(\T;\R)}\,ds \right)\\
				& = \mE \left (\int_0^t \left \langle D_{u_0}\left(\cP_{t-s}^{\cF} \psi \right )(u(s;u_0)),\eta_h(s) \right \rangle_{L^2(\T;\R)}\,ds \right) \\
				&= \mE \left( \int_0^t \left \langle D_{u_0}\left[\left(\cP_{t-s}^{\cF} \psi \right )(u(s;u_0))\right],h \right \rangle_{L^2(\T;\R)} ds\right).\notag
			\end{split}
		\end{equation}
		From Fubini-Tonelli's theorem and using the semigroup property, we then conclude
		\begin{equation}
			\begin{split}
				\mE \left[ \psi(u(t;u_0)) Z_h(t) \right]&=  \left \langle D_{u_0} \int_0^t \mE \left [( P_{t-s}^{\cF} \psi )(u(s;u_0)) \right]\,ds , h \right \rangle_{L^2(\T;\R)}\\
				& =t  \left \langle D_{u_0}\left (\cP_t^{\cF}\psi \right )(u_0),h \right \rangle_{L^2(\T;\R)}. \notag
			\end{split}
		\end{equation}
	\end{proof}
	The next step toward the strong Feller property is to obtain an estimate of the $L^2(\T;\R)$-norm of the Fr\'echet derivative $D_{u_0}\cP_t^{\cF}\psi$, for $\psi \in C_b^2(L^2(\T;\R);\R)$. 
	
	\begin{proposition} \label{stima derivata semigruppo}
		Let Assumption \ref{ass:QstrongFeller} hold. Then,
		\begin{equation} 
			\begin{split}
				\sup_{u_0\in L^2(\T;\R)}\left \| D_{u_0}\left (\cP_t^{\cF}\psi \right )(u_0) \right \|_{L^2(\T;\R)}^2 \leq \frac{(1+t)^\gamma}{t^{1+\gamma}}e^{CL_\cF^2 t} \| \psi \|_{L^{\infty}(L^2(\T;\R);\R)}^2,\notag
			\end{split}
		\end{equation}
		for any $\psi \in C_b^2(L^2(\T;\R);\R)$ and $t >0$.
	\end{proposition}
	\begin{proof}
		Let $h \in L^2(\T;\R)$. From Bismut-Elworthy-Li formula (Proposition \ref{prop bismut}) and It\^o's isometry we deduce 
		\begin{equation}
			\begin{split}
				\left |\left \langle D_{u_0}\left (\cP_t^{\cF}\psi \right )(u_0),h \right \rangle_{L^2(\T;\R)} \right |^2  
				& \leq \frac{1}{t^2} \|\psi\|_{L^{\infty}(L^2(\T;\R);\R)}^2 \quad\mE \left | Z(t) \right |^2 \\
				& = \frac{1}{t^2} \|\psi\|_{L^{\infty}(L^2(\T;\R);\R)}^2 \quad\mE \int_0^t \left \|Q^{-\frac{1}{2}}\eta_h(s) \right \|_{L^2(\T;\R)}^2 \,ds. \notag
			\end{split}
		\end{equation}
		Applying Lemma \ref{interpolation formula} and then H\"older's inequality with $p=\frac{1}{1-\gamma}$ and $p^{'}=\frac{1}{\gamma}$, we can write 
		\begin{align*}
			\int_0^t \left \|Q^{-\frac{1}{2}}\eta_h(s) \right \|_{L^2(\T;\R)}^2 \,ds &\leq c \int_0^t  \left \| \eta_h(s) \right \|_{L^2(\T;\R)}^{2(1-\gamma)} \left \|\eta_h(s) \right \|_{H^1(\T;\R)}^{2\gamma}\,ds \\
			& \leq c\left ( \int_0^t  \left \| \eta_h(s) \right \|_{L^2(\T;\R)}^{2} \,ds \right )^{1-\gamma} \left ( \int_0^t  \left \| \eta_h(s) \right \|_{H^1(\T;\R)}^{2} \,ds \right)^{\gamma} \\
			%& \leq \sup_{s \in [0,t]} \| \eta_h(s) \|_{L^2(\T;\R)}^{2(1-\gamma)} \,\,t^{1-\gamma} \left ( \int_0^t  \left \| \eta_h(s) \right \|_{H^1(\T;\R)}^{2} \,ds \right)^{\gamma}\\
			& \leq c \left(\|h\|_{L^2(\T;\R)}^{2(1-\gamma)} \, e^{CL_{\cF}^2t(1-\gamma)}\, t^{1-\gamma}\right)\cdot\left( \|h\|_{L^2(\T;\R)}^{2\gamma} (t+1)^{\gamma}e^{CL_{\cF}^2t\gamma}\right)\\
			& = \|h\|_{L^2(\T;\R)}^{2}\, (1+t)^\gamma e^{CL_\cF^2 t} \, t^{1-\gamma},\quad\text{$\mP$-a.s.}, 
		\end{align*}
		where the last inequality follows from the bounds \eqref{inequality u}-\eqref{inequality derivata di u}.
		Combining the above bounds, 
		\begin{equation*}
			\left |\left \langle D_{u_0}\left (\cP_t^{\cF}\psi \right )(u_0),h \right \rangle_{L^2(\T;\R)} \right |^2 \leq \frac{(1+t)^\gamma}{t^{1+\gamma}} e^{CL_\cF^2 t} \|\psi\|_{L^{\infty}(L^2(\T;\R);\R)}^2 \|h\|_{L^2(\T;\R)}^{2}.
		\end{equation*}
		Choosing $h=D_{u_0}\left( \cP_t^{\cF}\psi\right)(u_0)$ we obtain the statement.
	\end{proof}
	The strong Feller property of $\{ \cP_t^{\cF} \}_{t \geq 0}$ is now a straightforward consequence of the mean value theorem. 
	\begin{proposition}\label{prop:strong feller Lipshitz nonlinearities} 
		Let Assumption \ref{ass:QstrongFeller} hold. 
		If $\psi \in \cB_b(L^2(\T;\R);\R)$ then $\cP_t^{\cF} \psi$ is Lipschitz continuous (and hence Strong Feller) for every $t > 0$ (as a function from $L^2$ to $\R$), i.e.
		\begin{equation} \label{non so che scrivere}
			\left |\left ( \cP_t^{\cF}\psi \right )(u_0) - \left ( \cP_t^{\cF}\psi \right )(v_0) \right |^2 \leq \frac{(1+t)^\gamma}{t^{1+\gamma}} e^{CL_\cF^2 t} \|\psi\|_{L^{\infty}}\|u_0-v_0\|_{L^2(\T;\R)}^2, 
		\end{equation} 
		for any $\psi \in \cB_b(L^2(\T;\R);\R), u_0,v_0 \in L^2(\T;\R)$ and $t > 0$. In particular, the semigroup $(P^R_t\psi)$ associated to \eqref{damped spde} is Lipschitz continuous and satisfies \eqref{strong feller for PtR}, for every $R>0$. 
	\end{proposition}
	\begin{proof}
		%We start by recalling that if $H$ is a Hilbert space and  $\cP_t$ is any Markov semigroup on  $\cB_b(L^2(\T;\R);\R)$ then the inequality
		%$$ \left \| \left( \cP_t \psi \right )(u_0)-\left( \cP_t \psi \right )(v_0) \right \|_{L^2(\T;\R)}  \leq c \| \psi\|_{L^{\infty}(L^2(\T;\R);\R)} \|u_0-v_0\|_{L^2(\T;\R)}, \quad u_0,v_0 \in L^2(\T;\R), c>0, $$
		%holds for all $\psi \in \cB_b(H;\R)$ if and only if it holds for  any  $\psi \in C_b^2(H;\R)$, see \cite[Lemma 2.2]{peszat}. 
		First, recall that if \eqref{non so che scrivere} holds for all $\psi \in C_b^2 \left( L^2(\T;\R);\R\right)$, then it also holds for all $\psi \in \cB_b \left( L^2(\T;\R);\R\right)$, see \cite[Lemma 2.2]{peszat}. Hence,  it is enough to prove \eqref{non so che scrivere} for $\psi \in C_b^2 \left( L^2(\T;\R);\R\right)$. From the (infinite dimensional version of the) mean value theorem, see Proposition \cite[Proposition 3.2.7]{Ricciflow},  and from Proposition \ref{stima derivata semigruppo}, we then have  
		\begin{equation}
			\begin{split}
				\left | \cP_t^{\cF}\psi(u_0)-\cP_t^{\cF}\psi(v_0) \right |^2 &\leq  \sup \limits_{\bar{u} \in [u_0;v_0]} \left  \|D_{u_0}\left (\cP_t^{\cF}\psi \right )(\bar{u})\right \|_{L^2(\T;\R)}^2 \left \|u_0-v_0 \right \|_{L^2(\T;\R)}^2 \\
				& \leq \frac{(1+t)^\gamma}{t^{1+\gamma}} e^{CL_\cF^2 t}\,\|\psi\|_{L^{\infty}(L^2(\T;\R);\R)}^2\,\, \|u_0-v_0\|_{L^2(\T;\R)}^2\,,  \notag
			\end{split}
		\end{equation}
		where in the above $[u_0;v_0]:= \left \{ r u_0+(1-r)v_0\,:\, r \in [0,1]\right \}.$ 
		
		Finally, to show that, for any given $R>0$, the semigroup $(P^R_t\psi)$ generated by \eqref{damped spde} is Lipschitz continuous, it is enough to note that $V^{'}u+\mathcal{M}_R(u)\in C^2_b\left(L^2(\T;\R);\R\right)$.
	\end{proof}

	\begin{appendix}
		\section{Proofs of Section \ref{sec:sec4}}\label{uniqueness sigmageq1}
		\subsection{Characterization of stationary solutions}
		\begin{proof}[Proof of Proposition \ref{stationary proposition}]
			The approach we use is well established, at least since \cite{Dre2010}, so we only give a sketch.   \\
			We first prove that for any given $\mu \in \mathcal{P}_{ac}(\T)$ the solution $\rho_{\mu}$ to the linear equation
			\begin{align}\label{linearised equation}
				0= \sigma \pa_{xx} \rho_{\mu}(x) +\pa_x \left[\left(V'(x) + (F'\ast \mu)(x)\right) \rho_{\mu}(x)\right]
			\end{align}
			is unique and it is given by 
			\begin{equation}\label{}
				\rho_{\mu}(x)=\frac{1}{C_{\sigma}} e^{-\frac{1}{\sigma} \big (\int \limits_0^x (V'(y)+F'\ast \mu(y)\,dy\big)},\quad x \in \T,
			\end{equation}
			with $C_{\sigma}$ normalisation constant.
			Indeed, since $V' \in C^{\infty}(\T;\R)$ and for any fixed $\mu \in \mathcal{P}_{ac}(\T)$ the function $F'\ast \mu$ is smooth, i.e. $F'\ast \mu \in C^{\infty}(\T;\R)$ (see e.g. \cite[Lemma 2.3., p.14]{Tartar} for further details), the (linear) operator 
			$$\mathcal{L}(\rho)(x):=\pa_{xx} \rho(x) +\pa_x \left[\left(V'(x) + (F'\ast \mu)(x)\right) \rho(x)\right],\quad x \in \T,$$ is uniformly elliptic and with smooth coefficients, hence, any weak solution $\rho_{\mu} \in H^1(\T;\R) \cap \cP_{ac}(\T)$ to \eqref{linearised equation} is actually smooth, i.e. $\rho_{\mu} \in C^{\infty}(\T;\R) \cap \cP_{ac}(\T)$, so the derivatives can be intended in the classical sense. 
			Equation \eqref{linearised equation} can be then solved explicitly:
			\begin{equation}
				\begin{split}
					&\pa_x [(V'(x)+F' \ast \mu(x))\rho_{\mu}(x)]+\sigma \pa_{xx} \rho_{\mu}(x)=0 \\
					& \pa_x [(V'(x)+F' \ast \mu(x))\rho_{\mu}(x)+\sigma \pa_{x}\rho_{\mu}(x)]=0 \\
					& (V'(x)+F' \ast \mu(x))\rho_{\mu}(x)+\sigma \pa_{x}\rho_{\mu}(x)=d \\
					& \sigma \pa_{x}\rho_{\mu}(x)=d-(V'(x)+F' \ast \mu(x))\rho_{\mu}(x), \notag
				\end{split}
			\end{equation}
			for some constant $d \in \R$.
			Finally, from the variation of constants formula we deduce that $\rho_{\mu}$ has the following expression
			\begin{equation}
				\rho_{\mu}(x)=\bigg( c+\frac{d}{\sigma}\int \limits_0^x e^{\frac{1}{\sigma} \big(\int \limits_0^y (V'(z)+F' \ast \mu(z)\,dz\big)}\,dy \bigg )e^{-\frac{1}{\sigma} \big(\int \limits_0^x (V'(y)+F'\ast \mu(y)\,dy\big)},\,x \in \T,\notag
			\end{equation}
			where $c$ and $d$ are real constants to be determined later.
			From the periodicity of $\rho_{\mu}$ we know that $\rho_{\mu}(0)=\rho_{\mu}(2\pi)$ which gives 
			$$c=c+\frac{d}{\sigma}\int \limits_0^{2\pi} e^{\frac{1}{\sigma} \big(\int \limits_0^y (V'(z)+F' \ast \mu(z)\,dz\big)}\,dy, $$
			where the above equality is a consequence of the periodicity of $V$ and $F$; indeed
			\begin{equation}
				\begin{split}
					& \int \limits_0^{2\pi} \left(V'(y)+F'\ast \mu(y) \right)\,dy=V(2\pi)-V(0)+F \ast \mu(2\pi)-F \ast \mu(0) \\ 
					&=  \int_{\T} F(2\pi-x) \mu(x)\,dx-F \ast \mu(0)=\int_{\T} F(-x) \mu(x)\,dx-F \ast \mu(0) \\
					&= F \ast \mu(0)-F \ast \mu(0)=0, \notag
				\end{split}
			\end{equation}
			therefore, we have $ e^{-\frac{1}{\sigma} \big(\int \limits_0^{2\pi} (V'(y)+F'\ast \mu(y)\,dy\big)}=1$. Hence, $d\int \limits_0^{2\pi} e^{\frac{1}{\sigma} \big(\int \limits_0^y (V'(z)+F' \ast \mu(z)\,dz\big)}\,dy=0$, which implies $d=0$. The constant $c$ is now determined by renormalization.
			We omit the rest of the argument and just recall that if we consider the map 
			$K:\mathcal{P}_{ac}(\T) \to \mathcal{P}_{ac}(\T)$ defined as 
			$K(\mu):=\rho_{\mu}$ from the above we then have that a solution to the non-linear problem \eqref{stationary equation} must be of the form \eqref{fixed point}.
		\end{proof}
		\subsection{Step 4 of the proof of Theorem \ref{thm:mainthm3-inv-meas}} \label{uniqueness h}
		We restrict to the case $\sigma \geq \frac{1}{2}$, the approach adopted is similar to what we have done for $\bar{g}_{\sigma}$. Namely, let $\xi_{\sigma}:\R \to \R$ be the map defined as 
		\begin{equation} \label{xi}
			\xi_{\sigma}(m):= \int_{\T} (\cos x-m)e^{-\frac{1}{\sigma}\cos(2x)+\frac{m}{\sigma}\cos x}\,dx,\quad m \in \R,
		\end{equation}
		and note that $m \in \R$ is a fixed point of $h_{\sigma}$ if and only if $m$ is a zero of $\xi_{\sigma}$, i.e. $\xi_{\sigma}(m)=0$. Let us also introduce the sequence $\{c_k\}_{k \in \N}$ defined as 
		\begin{equation}
			c_k:=\int_{\T}(\cos x)^k\,e^{-\frac{1}{\sigma}\cos(2x)}\,dx,\quad k \in \N.
		\end{equation}
		Since $\cos x$, $x \in \T$, is an anti-symmetric function and $\cos(2x)$, $x \in \T$, is a symmetric function  with respect to $x=\frac{\pi}{2}$ and $(\cos x)^2 < 1$ for all $x \in (0,2\pi)$ we have
		\begin{eqnarray}
			&& c_{2k+1}=0,\quad \text{for all $k \in \N$}, \label{zero on odd numbers cosine} \\
			&& c_{2k+2} < c_{2k},\quad \text{for all $k \in \N$} \label{decreases on even numbers cosine}.
		\end{eqnarray}
		In Proposition \ref{series expansion} we provide a power series 
		expansion of $\xi_{\sigma}$; similarly to what we have done in Section \ref{More invmeas} this power expansion will then allow us to prove the following result. 
		\begin{proposition}\cite[cfr. Step 1, Theorem 2.1]{tugaut}\label{series expansion xi prop}
			The function $\xi_{\sigma}:\R \to \R$ admits the following series expansion 
			\begin{equation}\label{series expansion xi}
				\xi_{\sigma}(m)= \sum \limits_{k \in \N} \frac{1}{(2k)!} \left( \frac{m}{\sigma}\right )^{2k+1} c_{2k}
				\iota_k(\sigma),
			\end{equation}
			where $\{\iota_k(\sigma)\}_{k \in \N}$ is the sequence defined as 
			\begin{equation}\label{iota}
				\iota_k(\sigma):=\frac{c_{2k+2}}{(2k+1)c_{2k}}-\sigma,\quad k \in \N,\, \sigma >0.
			\end{equation}
		\end{proposition}
		We omit the proof of the above result as it can be done with calculations similar to Lemma \ref{series expansion}.
		\begin{theorem}
			If $\sigma \geq \frac{1}{2}$ then $\xi_{\sigma}:\R \to \R$ admits a unique zero which is $m=0$. This implies that $m=0$ is the unique fixed point of $h_{\sigma}:\R \to \R$ for $\sigma \geq \frac{1}{2}$.
		\end{theorem}
		\begin{proof}
			Let us note that $\frac{c_{2k+2}}{(2k+1)c_{2k}} \leq \frac{1}{3}$ for all $k \geq 1$. With this in mind, from \eqref{series expansion xi} we can write 
			\begin{equation}
				\xi_{\sigma}(m)=\left( \frac{m}{\sigma} \right) \left( c_2-\sigma c_0 \right ) + \sum \limits_{k=1}^{+\infty} \frac{1}{(2k)!} \left( \frac{m}{\sigma}\right )^{2k+1} c_{2k}\iota_k(\sigma). \notag
			\end{equation}
			By using the identity $\cos(2x)=2\cos^2 x-1$ the factor $c_2-\sigma c_0$  can be rearranged into the following form:
			\begin{equation}\label{first derivaitve rearranged cosine}
				\begin{split}
					c_2-\sigma c_0 & = \int_{\T}  \cos^2 x \,e^{-\frac{1}{\sigma}\cos(2x)}\,dx-\sigma \int_{\T} e^{-\frac{1}{\sigma}\cos(2x)}\,dx \\
					& =\left( \frac{1}{2}-\sigma \right )\int_{\T} e^{-\frac{1}{\sigma}\cos(2x)}\,dx+\frac{1}{2}\int_{\T}\cos(2x)e^{-\frac{1}{\sigma}\cos(2x)}\,dx. \\
				\end{split}
			\end{equation}
			Using the modified Bessel functions of first kind (defined in \eqref{besselI}) we can recast \eqref{first derivaitve rearranged cosine} into the following form 
			\begin{align}\notag
				c_2-\sigma c_0 & = \left( \frac{1}{2}-\sigma \right) I_0 \left( -\frac{1}{\sigma} \right )+\frac{1}{2}I_1 \left (-\frac{1}{\sigma} \right ) \\
				&= \left( \frac{1}{2}-\sigma \right) I_0 \left( \frac{1}{\sigma} \right )-\frac{1}{2}I_1 \left (\frac{1}{\sigma} \right )< \left ( \frac{1}{2}-\sigma \right )I_0 \left (\frac{1}{\sigma} \right ).\label{58}
			\end{align} 
			Recalling that the functions $I_0(z) >0$, for all $z \in \R$ and $I_1(z)>0$ for all $z>0$, if $\sigma \geq \frac{1}{2}$ then $c_2-\sigma c_0 <0$ and, moreover, $\iota_k(\sigma) <0$, for all $k \geq 1$. Hence, since all the coefficients of the power series expansion of $\xi$ are strictly negative, we readily obtain that $\xi_{\sigma}(m)=0$ if and only if $m=0$. This concludes the proof.
		\end{proof}
		
		\begin{proof}[Proof of Lemma \ref{lemma asymptotic expansion}]
			We begin with proving formula \eqref{expansion 1}. We want to apply Proposition \ref{asymptotic 
				expansion} with $U(x):=\cos(2x)$, $ x \in \T$, and $G=0$. 
			To this end, we have to look for the minimum points of the function $U$ on the torus $\T$. Clearly, 
			$U$ admits two global minima $x_1=\frac{\pi}{2}$ and $x_2=\frac{3\pi}{2}$. Therefore, in order to 
			apply Lemma \ref{asymptotic expansion} we split the integral into two parts. Namely,
			\begin{equation}
				\int_{\T} e^{-\frac{1}{\sigma} \cos(2x)}\,dx=\int_{0}^{\pi} e^{-\frac{1}{\sigma} \cos(2x)}\,dx+\int_{\pi}^{2\pi} e^{-\frac{1}{\sigma} \cos(2x)}\,dx. \notag   
			\end{equation}
			Using the $\pi$-periodicity of $\cos(2x)$ we reduce to a single integral i.e. 
			\begin{equation}
				\int_{\T} e^{-\frac{1}{\sigma} \cos(2x)}\,dx=2\int_{0}^{\pi} e^{-\frac{1}{\sigma} \cos(2x)}\,dx=2I.\notag
			\end{equation} 
			On the interval $[0,\pi]$ the function $U$ admits a unique global minimum at $x=\frac{\pi}{2}$. Hence, we can apply Lemma \ref{asymptotic expansion} to $I$ and obtain the desired result
			\begin{eqnarray}
				I &=& \sqrt{\frac{\pi \sigma}{2}} e^{\frac{1}{\sigma}} \left(1+o(1)  \right). \label{espansione per I}
			\end{eqnarray}
			Formula \eqref{expansion 2} is obtained with a similar reasoning. 
		\end{proof}

		\begin{proof}[Proof of Lemma \ref{stime asintotiche h}]
			We want to apply Lemma \ref{asymptotic expansion} with $U(x):=\cos(2x)$ and $G(x):=-\cos x$ so, we consider the function 
			$$U_m(x):=\cos(2x)-m\cos x,\quad x \in \T,\quad m \in [0,1].$$
			By direct calculation one can see  
			$x_{1,m}=\arccos\left(\frac{m}{4}\right)$ and $x_{2,m}=2\pi-\arccos\left(\frac{m}{4}\right)$ are points of global minimum for $U_m$ as $m$ ranges in $[0,1]$. 
			If we now let $f \in C^{3}(\T;\R)$ then by applying Lemma \ref{asymptotic expansion} we obtain that the following asymptotic expansion holds
			\begin{equation}
				\begin{split}
					& \int_{\T} f(x)e^{-\frac{1}{\sigma}(\cos(2x)-m\cos x)}\,dx =\sqrt{\frac{2\pi \sigma}{4-\frac{m^2}{4}}} e^{\frac{1}{\sigma} \left( 1+ \frac{m^2}{8} \right)} \left( f \left( x_{1,m} \right) +f \left( x_{2,m} \right) + \left(\gamma_{1,m}^f+\gamma_{2,m}^f \right)\sigma + o_m(\sigma) \right), \notag
				\end{split}
			\end{equation}
			where $\gamma_{1,m},\gamma_{2,m}$ are constants defined as in \eqref{gamma constant}.
			In our case of interest, since the second, third and fourth derivative of $U_m$ at $x_{1,m}$ and $x_{2,m}$ are respectively 
			\begin{eqnarray}
				&& U_m^{(2)}(x_{1,m})=U_m^{(2)}(x_{2,m})=4-\frac{m^2}{4},\notag \\
				&& U_m^{(3)}(x_{1,m})=-U_m^{(3)}(x_{2,m})=12 m \sqrt{1-\frac{m^2}{16}},\notag \\
				&& U_m^{(4)}(x_{1,m})=U_m^{(4)}(x_{2,m})=\frac{7}{4}m^2-16,\notag
			\end{eqnarray}
			if we set $f=1$ we have 
			\begin{equation}
				\begin{split}
					& \int_{\T} e^{-\frac{1}{\sigma}(\cos(2x)-m\cos x)}\,dx  =\sqrt{\frac{2\pi \sigma}{4-\frac{m^2}{4}}} e^{\frac{1}{\sigma} \left( 1+ \frac{m^2}{8} \right)} \left(2 + \left(\gamma_{1,m}+\gamma_{2,m} \right)\sigma + o_m(\sigma) \right). \notag
				\end{split}
			\end{equation}
			Moreover, since $f^{'}=f^{''}=0$ we have $\gamma_{1,m}=\gamma_{2,m}$ and by recalling that 
			$\mathcal{U}_k=\frac{d^k}{dx^k}(\cos(2x)-m\sin x)|_{x=x_{1,m}}$ we obtain
			$$\gamma_{1,m}=\frac{5\mathcal{U}_{3}^2}{24\mathcal{U}_{2}^3}-\frac{\mathcal{U}_{4}}{8\mathcal{U}_{2}^2}=\frac{c(m)}{2}+\frac{2}{(4-\frac{m^2}{2})^2},$$
			where $[0,1] \ni m \to c(m) \in \R$ is a continuous function such that $c(m) \to 0$ as $m \downarrow 0$.\\
			The expansions \eqref{expansion 3 h} and \eqref{expansion 4 h} are obtained analogously from \eqref{expansion f} by using $f(x)=\cos x$ and $f(x)=\cos^2 x$, respectively. For the reader who would like to check the details we point out that if $f(x)=\cos x$ then 
			since $f^{'}(x)=\sin x$ it follows that 
			\begin{equation}
				f^{'}(x_{1,m})=-f^{'}(x_{2,m}),\,\,\,f^{'}(x_{1,m})U^{'''}(x_{1,m})=f^{'}(x_{2,m})U^{'''}(x_{2,m}),
			\end{equation}
			hence, we obtain $\gamma_{1,m}^{(\cos)}=\gamma_{2,m}^{(\cos)}$ and
			$$ \gamma_{1,m}^{(\cos)}=f(x_{1,m}) \left ( \frac{5\mathcal{U}_{3}^2}{24\mathcal{U}_{2}^3}-\frac{\mathcal{U}_{4}}{8\mathcal{U}_{2}^2} \right) -f^{'}(x_{1,m}) \frac{\mathcal{U}_{3}}{2\mathcal{U}_{2}^2}+\frac{f^{''}(x_{1,m})}{2 \mathcal{U}_{2}}=\frac{\bar{c}(m)}{2},$$
			where $[0,1] \ni m \to \bar{c}(m) \in \R$ is a continuous function such that $\bar{c}(m) \to 0$ as $m \downarrow 0$. 
			If we set $f(x)=\cos^2 x$ then 
			\begin{equation}
				f^{'}(x)=-2\cos x \sin x,\,\,\,f^{'}(x_{1,m})U^{'''}(x_{1,m})=f^{'}(x_{2,m})U^{'''}(x_{2,m}),
			\end{equation}
			hence, we obtain $\gamma_{1,m}^{(\cos^2)}=\gamma_{2,m}^{(\cos^2)}$ and
			\begin{equation}
				\begin{split}
					\gamma_{1,m}^{(\cos^2)} & =f(x_{1,m}) \left ( \frac{5\mathcal{U}_{3}^2}{24\mathcal{U}_{2}^3}-\frac{\mathcal{U}_{4}}{8\mathcal{U}_{2}^2} \right) -f^{'}(x_{1,m}) \frac{\mathcal{U}_{3}}{2\mathcal{U}_{2}^2}+\frac{f^{''}(x_{1,m})}{2 \mathcal{U}_{2}} =\frac{\hat{c}(m)}{2}+\frac{1}{4-\frac{m^2}{4}}.\notag
				\end{split}
			\end{equation}
			where $[0,1] \ni m \to \hat{c}(m) \in \R$ is a continuous function such that $\hat{c}(m) \to 0$ as $m \downarrow 0$. 
		\end{proof}

		\section{Proofs of Section \ref{Well-Posedness McKean_Vlasov} to Section \ref{sec:sec7}} \label{estimate heat kernel}\label{stochastic convolution}\label{appendix: proofs of sec 7}
		We recall that if $G_t$, $t > 0$, is the heat kernel on $\R$, i.e. 
		\begin{equation}
			G_t(x):=\frac{1}{\sqrt{4\pi t}}e^{-\frac{x^2}{4t}},\qquad x \in \R,\, t>0, \notag
		\end{equation}
		the periodic heat kernel $G_t^{per}$ is defined as
		\begin{equation} \label{heat kernel on the torus}
			G_t^{per}(x):=\sum \limits_{k \in \Z} G_t(x+2k\pi),\qquad x \in \R,\, t > 0.
		\end{equation}
		Clearly, $G_t^{per} \in C^{\infty}((0,+\infty) \times \T)$ and has $L^1$-norm equal to 1. 
		
		We also recall that, since
		\begin{align} \label{estimate pax G}
			|\pa_x G_t(x)| \leq \frac{C}{\sqrt{t}}G_{2t}(x),
		\end{align}
		one has
		\begin{align}\label{estimate pax Gper}
			|\pa_x G_t^{per}(x)| \leq \frac{C}{\sqrt{t}}G_{2t}^{per}(x),
		\end{align}
		for all $x \in \R$, $t>0$, with $C=\sqrt{2}e^{-\frac{1}{2}}$. Hence, 
		\begin{equation}
			\label{pax Gper L1}
			\left \|\pa_x G_t^{per} \right \|_{L^1(\T;\R)} \leq \frac{C_1}{\sqrt{t}},
		\end{equation}
		for some constant $C_1>0$. We briefly recall that \eqref{estimate pax G} follows from writing
		\begin{equation}
			\left |\pa_x G_t(x) \right |= h_t(x)e^{-\frac{x^2}{8t}},\notag
		\end{equation} 
		where $x \in \R \mapsto h_t(x) \in \R $ is the even function defined as 
		$h_t(x):=\frac{|x|}{4\sqrt{\pi t^3}} e^{-\frac{x^2}{8t}},\quad x \in \R,\; t>0.$
		The maximum of $h_t$ is attained at $x=\pm 2\sqrt{t}$, from which \eqref{estimate pax G} follows. The bounds \eqref{estimate pax Gper} and \eqref{pax Gper L1} are then obvious.
		
		\begin{proof}[Proof of Lemma \ref{stime heat kernel}]
			We begin with showing that the bounded linear operator $P:C \left( [0,T] ; H^1(\T;\R) \right ) \to C \left( [0,T];L^2(\T;\R) \right)$ defined by \eqref{operatore P} can be extended to a bounded linear operator from $C \left( [0,T] ; L^2(\T;\R) \right)$ into $C \left( [0,T];L^2(\T;\R) \right)$ satisfying \eqref{diseguaglianza unmezzo}.
			To this end, first consider $z \in C \left( [0,T] ; H^1(\T;\R) \right )$ and $\psi \in L^2(\T;\R)$. In this case, we obtain
			\begin{align}	\left|\left\langle P[z](t),\psi \right\rangle\right|& = \left |\int_0^t \left \langle e^{(t-s)A} \partial_x z(s),\psi \right \rangle_{L^2(\T;\R)}\,ds \right |\nonumber\\
				& \leq \int_0^t \left |\left \langle  \partial_x z(s),e^{(t-s)A}\psi \right \rangle_{L^2(\T;\R)} \right | \,ds\nonumber\\ %= \int_0^t \left | \left \langle z(s),\partial_x \left ( e^{(t-s)A}\psi \right ) \right   \rangle_{L^2(\T;\R)} \right | \,ds\\
				&\leq \int_0^t \|z(s)\|_{L^2(\T;\R)} \left \|\partial_x \left (e^{(t-s)A}\psi \right ) \right \|_{L^2(\T;\R)}\,ds. \label{Pz_psi}
			\end{align}
			From Young's inequality for convolutions and \eqref{pax Gper L1}, for any $\psi \in L^2(\T;\R)$ the heat semigroup satisfies
			\begin{equation*}
				\left \| \partial_x \left ( e^{tA} \psi \right ) \right \|_{L^{2}(\T;\R)} \leq \left \| \pa_x G_t^{per}\right \|_{L^1(\T;\R)} \|\psi\|_{L^2(\T;\R)} \leq \frac{C_1}{t^{\frac{1}{2}}} \| \psi \|_{L^2(\T;\R)},
			\end{equation*}
			for all $t >0$. Thus,
			\begin{equation*} \label{inequality eq perunmezzo}
				\int_0^t \|z(s)\|_{L^2(\T;\R)} \left \|\partial_x \left ( e^{(t-s)A}\psi \right ) \right \|_{L^{2}(\T;\R)}\,ds  \leq C_1\|\psi \|_{L^{2}(\T;\R)}\int_0^t (t-s)^{-\frac{1}{2}}\|z(s)\|_{L^2(\T;\R)}\,ds,
			\end{equation*}
			for all $\psi \in L^2(\T;\R)$ and $t \in [0,T]$. Setting $\psi=P[z](t)$, we have
			\begin{eqnarray} \label{diseguaglianza prima}
				\left \| P [z](t) \right \|_{L^2(\T;\R)} && \leq C_1 \int_0^t (t-s)^{-\frac{1}{2}}\|z(s)\|_{L^2(\T;\R)}\,ds \\
				&& \leq C_1 T^{\frac{1}{2}} \left \| z \right \|_{C \left( [0,T] ; L^2(\T;\R)\right) }, \label{dis boh}
			\end{eqnarray}
			for all $t \in [0,T]$. Therefore, inequality \eqref{diseguaglianza unmezzo} holds for all $z \in C \left ([0,T];H^1(\T;\R)\right )$.
			
			Now, let $z \in C\left ([0,T];L^2(\T;\R)\right )$ and let $\{z_n\}_{n \in \N} \subset C \left ([0,T];H^1(\T;\R) \right )$ be such that $z_n \to z$ in $C \left ([0,T];L^2(\T;\R) \right )$ as $n\to+\infty$. Then from \eqref{diseguaglianza prima} we know that 
			$$\left \| P[z_n](t) \right \|_{L^2(\T;\R)} \leq C_1 \int_0^t (t-s)^{-\frac{1}{2}}\|z_n(s)\|_{L^2(\T;\R)}\,ds, 
			$$
			for all $n \in \N$ and $t \in [0,T]$. Moreover, from the linearity of $P$ and inequality \eqref{dis boh} we obtain that $\left \{P[z_n] \right \}_{n \in \N}$ is a Cauchy sequence in $C \left ([0,T];L^2(\T;\R) \right )$.
			Therefore, if we define $P[z]$ as the limit in $C([0,T];L^2(\T;\R))$ of $P[z_n]$ as $n \to +\infty$ (due to \eqref{dis boh} such limit is independent from the choice of the approximating sequence $\{z_n\}_{n \in \N}$) then we obtain that inequality \eqref{diseguaglianza unmezzo} holds for all $z \in C \left ( [0,T];L^2(\T;\R)\right )$ and for any $t\in[0,T]$. 
		\end{proof}
		
		We now prove global well-posedness and $L^2$ bounds for the solution of the PDE \eqref{pde regolare} used in the proof of Proposition \ref{stime sistema deterministico}.  
		We denote by $W^{1,\infty}(\T;\R)$ the Sobolev  space of bounded  functions with weak derivative in $L^{\infty}(\T;\R)$,  endowed with the norm 
		$$\|f\|_{W^{1,\infty}(\T;\R)}=\|f\|_{L^{\infty}(\T;\R)}+\|\pa_xf\|_{L^{\infty}(\T;\R)},\quad \text{for all $f \in W^{1,\infty}(\T;\R)$}.$$
		
		\begin{proposition}\label{regularised PDE}
			For any  $T>0$ (independent of $\omega$) {the random} PDE \eqref{pde regolare} admits a {$\mP$-a.s.} continuous $L^2(\T;\R)$-valued mild solution $v$ on the interval $[0,T]$; moreover such a solution is a classical solution and satisfies the a priori estimates \eqref{L2 norm}.
			%Furthermore, $a(v_0,\varphi)$ depends continuously on the functions $v_0$ and $\varphi$ in the sense that if we let $\{v_0^n \}_{n \in \N}, v_0 \subset L^2(\T;\R)$ such that $v_0^n \xrightarrow[n \to +\infty]{} v_0$ in $L^2(\T;\R)$ and $\{\varphi_n \}_{n \in \N},\varphi \subset C([0,T];H^1(\T;\R))$ such that $\varphi_n \xrightarrow[n \to +\infty]{} \varphi$ in $C([0,T];H^1(\T;\R))$ then it follows that $a(v_0^n,\varphi_n) \xrightarrow[n \to +\infty]{} a(v_0,\varphi)$ uniformly in $[0,T]$. Moreover,
			
			%where $t \to b(v_0,\varphi)(t)$ and $t \to c(\varphi)(t)$ are non-negative time-continuous functions defined as, respectively, 
			%$$b(v_0,\varphi)(t):=\|v_0\|_{L^2(\T;\R)}^2+2\pi\|F^{'}\|_{L^{\infty}(\T;\R)}^2\int_0^t \|\varphi(s)\|_{L^2(\T;\R)}^4\,ds,\,t \in [0,T],$$
			%$$c(\varphi)(t):= 4\pi\|F^{'}\|_{L^{\infty}(\T;\R)}^2\|\varphi(t)\|_{L^2(\T;\R)}^2, t \in [0,T].$$
			%Moreover, $b(v_0,\varphi)$ and $c(\varphi)$ depend continuously on $v_0$ and $\varphi$ in the same way like $a(v_0,\varphi)$ does. 
		\end{proposition}

		\begin{proof}
			The local existence of a { $\mP$-a.s.} continuous $L^2(\T;\R)$-valued mild 
			solution can be proven exactly as in Proposition \ref{local existence mild solution} 
			with $\varphi$ in place of $W_A$. We denote such a solution by  $v$. We also note that,
			due to the smoothing properties of $A$, $v$ is a smooth solution {$\mP$-a.s.} as long as $v$ does not blow up, since the coefficients $V,F$ and the 
			external forcing term $\varphi$ are smooth. In other words, $v$ is a classic solution 
			to \eqref{pde regolare} defined up to a time $T^{*}={T^{*}(\omega)}>0$ small enough. In the same fashion 
			of proof of Theorem \ref{mild solution global existence}, to prove the global existence
			of $v$, it is enough to show that if $v$ is a solution up to time 
			$\tilde{T}>0$, then the estimate \eqref{L2 norm} is satisfied for all $t\in[0,\tilde{T}]$. 
			
			To ease the presentation, we denote $v_{\varphi}:=v+\varphi$ and we omit the dependence on time $t$ for $v$ and $\varphi$, i.e. we write $v$ and $\varphi$ in place of $v(t)$ and $\varphi(t)$, respectively. In what follows, if not further specified, $C$ denotes a generic deterministic positive constant, the value of which may change from line to line. Because of the non-linear term, to estimate the $L^2$-norm we need to start by estimating the $L^1$-norm. To do so, we analyse the derivative $\frac{d}{dt} \|v\|_{L^1(\T;\R)}$. To differentiate the $L^1(\T;\R)$-norm of $v$, since the function $x \to |x|$ is not smooth, we need to approximate it with a family of regular functions. To this end, let us consider the following convex $C^2(\R;\R)$-approximation $\{ \chi_{\varepsilon}\}_{\varepsilon >0}$ of the absolute value
			\begin{equation} \label{approssimazione valore assoluto}
				\chi_{\varepsilon}(r)=
				\begin{cases}
					|r|,\quad &|r| > \varepsilon, \\
					-\frac{r^4}{8\varepsilon^3}+\frac{3r^2}{4\varepsilon}+\frac{3\varepsilon}{8},\quad& |r| \leq \varepsilon. 
				\end{cases}
			\end{equation}
			Then, we have 
			\begin{equation}
				\begin{split}
					\frac{d}{dt}\int_{\T}\chi_{\varepsilon}(v)\,dx =\,&\int_{\T} \chi_{\varepsilon}^{'}(v) \left \{\pa_{xx}v+\partial_x \left[ V^{'}v_{\varphi} + (F^{'}*v_{\varphi}) v_{\varphi} \right]\right \}\,dx\\
					=\,& \int_{\T}  \chi_{\varepsilon}^{'}(v) \partial_x \left[ V^{'}\varphi+(F^{'}*v_{\varphi})\,\varphi \right]dx\\
					&- \int_{\T}  \chi_{\varepsilon}^{''}(v)\pa_xv \left[ \pa_x v +V^{'}v+(F^{'}*v_{\varphi})\,v\right]\,dx,\quad {\mP-a.s.} 
					%\\=:\,&I_{\varepsilon}^{'}+I_{\varepsilon}^{''}
					\label{step1-estimates}
				\end{split}
			\end{equation}
			Since $|\chi_{\varepsilon}^{'}(r)| \leq 1$ for all $r>0$ and all $\varepsilon>0$, the first addend on the RHS of \eqref{step1-estimates} can be estimated by
			\begin{equation}
				\begin{split}
					\int_{\T}  \chi_{\varepsilon}^{'}(v)\partial_x &\left[ V^{'}\varphi+(F^{'}*v_{\varphi})\,\varphi \right]dx\\
					\leq\,& \int_{\T} |V^{'}+F^{'}*v_{\varphi}||\pa_x\varphi|\,dx + \int_{\T} |V^{''}+F^{''}*v_{\varphi}||\varphi|\,dx \\
					%\leq\,& \left[\|V^{'}\|_{L^{\infty}(\T;\R)}+\|F^{'}\|_{L^{\infty}(\T;\R)}\left(\|v\|_{L^1(\T;\R)}+\|\varphi\|_{L^1(\T;\R)} \right)\right]\|\pa_x \varphi\|_{L^1(\T;\R)} \\
					%& + \left[\|V^{''}\|_{L^{\infty}(\T;\R)}+\|F^{''}\|_{L^{\infty}(\T;\R)}\left(\|v\|_{L^1(\T;\R)}+\|\varphi\|_{L^1(\T;\R)} \right)\right]\|\varphi\|_{L^1(\T;\R)}\\
					% \leq\,& \left (\|V^{'}\|_{L^{\infty}(\T;\R)}+\|V^{''}\|_{L^{\infty}(\T;\R)} \right ) \left( \|\varphi\|_{L^1(\T;\R)}+\|\pa_x\varphi\|_{L^1(\T;\R)}\right)\\
					%& + \left( \|F^{'}\|_{L^{\infty}(\T;\R)}+\|F^{''}\|_{L^{\infty}(\T;\R)} \right)\|v\|_{L^1(\T;\R)}\left(\|\varphi\|_{L^1(\T;\R)}+\|\pa_x\varphi\|_{L^1(\T;\R)}\right)\\ 
					%& + \left( \|F^{'}\|_{L^{\infty}(\T;\R)}+\|F^{''}\|_{L^{\infty}(\T;\R)} \right) \left(\|\varphi\|_{L^1(\T;\R)}+\|\pa_x\varphi\|_{L^1(\T;\R)}\right)^2 .
					%\\ & \leq \sqrt{2\pi}\|V^{'}\|_{W^{1,\infty}(\T;\R)} \|\varphi\|_{H^1(\T;\R)}\\
					%& + \sqrt{2\pi}\left( \|F^{'}\|_{L^{\infty}(\T;\R)}+\|F^{''}\|_{L^{\infty}(\T;\R)} \right)\|v\|_{L^1(\T;\R)}\|\varphi\|_{H^1(\T;\R)}\\ 
					%& + 2\pi\left( \|F^{'}\|_{L^{\infty}(\T;\R)}+\|F^{''}\|_{L^{\infty}(\T;\R)} \right) \|\varphi\|_{H^1(\T;\R)}^2.\\
					\leq\,& \sqrt{2\pi}\|V^{'}\|_{W^{1,\infty}(\T;\R)} \|\varphi\|_{H^1(\T;\R)}+ 2\pi\|F^{'}\|_{W^{1,\infty}(\T;\R)} \|\varphi\|_{H^1(\T;\R)}^2 \\
					& +\sqrt{2\pi}\|F^{'}\|_{W^{1,\infty}(\T;\R)}\|v\|_{L^1(\T;\R)}\|\varphi\|_{H^1(\T;\R)},\quad {\mP-a.s.}.\label{mi serve dopo 1}
				\end{split}
			\end{equation}
			As for the last addend on the RHS of \eqref{step1-estimates}, using the fact that $\chi_{\varepsilon}^{''} \geq 0$ and applying Young's inequality, we have 
			\begin{equation}
				\begin{split}
					- \int_{\T}  \chi_{\varepsilon}^{''}(v)\pa_xv & \left[ \pa_x v +V^{'}v+(F^{'}*v_{\varphi})\,v\right]\,dx\\ \leq\,& -\int_{\T}  \chi_{\varepsilon}^{''}(v)|\pa_xv|^2\,dx+ \frac{1}{2}\int_{\T} \chi_{\varepsilon}^{''}(v)|\pa_xv|^2\,dx+ \frac{1}{2} \int_{\T} \chi_{\varepsilon}^{''}(v)|V^{'}|^2|v|^2\,dx \\
					& + \frac{1}{2}\int_{\T} \chi_{\varepsilon}^{''}(v)|\pa_xv|^2\,dx+ \frac{1}{2} \int_{\T} \chi_{\varepsilon}^{''}(v)|F^{'}*v_{\varphi}|^2|v|^2\,dx \\
					\leq\,&	\frac{\|V^{'}\|_{L^{\infty}(\T;\R)}^2}{2} \int_{\T} \chi_{\varepsilon}^{''}(v)|v|^2\,dx+\frac{\|F^{'}*v_{\varphi}\|_{L^{\infty}(\T;\R)}^2}{2} \int_{\T} \chi_{\varepsilon}^{''}(v)|v|^2\,dx, \quad {\mP-a.s.}\notag
				\end{split}
			\end{equation}
			By \eqref{approssimazione valore assoluto}, we note that
			\begin{equation}
				\chi_{\varepsilon}^{''}(r)=
				\begin{cases}
					0,\quad&|r| > \varepsilon,  \\
					-\frac{3r^2}{2\varepsilon^3}+\frac{3}{2\varepsilon},\quad& |r| \leq \varepsilon, \notag
				\end{cases}
			\end{equation}
			is a non-negative continuous function and $r \to \chi_{\varepsilon}^{''}(r)|r|^2$ is bounded from above by the constant $\frac{3\varepsilon}{2}$. Hence, we deduce
			\begin{equation} \label{mi serve dopo 2}
				\begin{split}
					- \int_{\T}  \chi_{\varepsilon}^{''}(v)\pa_xv & \left[ \pa_x v +V^{'}v+(F^{'}*v_{\varphi})\,v\right]\,dx\\ &\leq \frac{3\pi\|V^{'}\|_{L^{\infty}(\T;\R)}^2\varepsilon}{2}+3\pi\|F^{'}\|_{L^{\infty}(\T;\R)}^2\varepsilon \left( \|v\|_{L^1(\T;\R)}^2+\|\varphi\|_{L^1(\T;\R)}^2 \right), \quad {\mP-a.s.}
				\end{split}
			\end{equation}
			Putting together \eqref{mi serve dopo 1}, \eqref{mi serve dopo 2} and \eqref{step1-estimates}, we have
			\begin{equation}
				\begin{split}
					\frac{d}{dt} \int_{\T} \chi_{\varepsilon}(v) dx  \leq\,& \sqrt{2\pi}\|V^{'}\|_{W^{1,\infty}(\T;\R)} \|\varphi\|_{H^1(\T;\R)}+2\pi\|F^{'}\|_{W^{1,\infty}(\T;\R)} \|\varphi\|_{H^1(\T;\R)}^2 \\
					& +\sqrt{2\pi}\|F^{'}\|_{W^{1,\infty}(\T;\R)}\|v\|_{L^1(\T;\R)}\|\varphi\|_{H^1(\T;\R)}+\frac{3\pi\|V^{'}\|_{L^{\infty}(\T;\R)}^2\varepsilon}{2}\\
					&+3\pi\|F^{'}\|_{L^{\infty}(\T;\R)}^2\varepsilon \left( \|v\|_{L^1(\T;\R)}^2+\|\varphi\|_{L^1(\T;\R)}^2 \right), \quad {\mP-a.s.} \notag
				\end{split}
			\end{equation}
			%By integrating with respect to time we obtain 
			%\begin{equation}
			%	\begin{split}
			%		\int_{\T} \chi_{\varepsilon}(v) \,dx  \leq\,& \int_{\T} \chi_{\varepsilon}(v_0)\,dx+ \sqrt{2\pi}\|V^{'}\|_{W^{1,\infty}(\T;\R)} \int_0^t \|\varphi\|_{H^1(\T;\R)}\,ds\\
			%		&+ 2\pi\|F^{'}\|_{W^{1,\infty}(\T;\R)} \int_0^t \|\varphi\|_{H^1(\T;\R)}^2 \,ds +\frac{3\pi\|V^{'}\|_{L^{\infty}(\T;\R)}^2\varepsilon t}{2}\\ &+\sqrt{2\pi}\|F^{'}\|_{W^{1,\infty}(\T;\R)}\int_0^t \|v\|_{L^1(\T;\R)} \|\varphi\|_{H^1(\T;\R)}\,ds \\
			%		&+3\pi\|F^{'}\|_{L^{\infty}(\T;\R)}^2\varepsilon \int_0^t \left( \|v\|_{L^1(\T;\R)}^2+\|\varphi\|_{L^1(\T;\R)}^2 \right)\,ds.\notag
			%	\end{split}
			%\end{equation}
			Finally, integrating with respect to time, letting $\varepsilon\to 0$ and applying Gronwall's lemma, we obtain
			\begin{equation}\label{L1 norm v}
				\|v(t)\|_{L^1(\T;\R)} \leq \left(\|u_0\|_{L^1(\T;\R)}+C\int_0^t \left(1+\|\varphi(s)\|_{H^1(\T;\R)}^2\right)ds\right)\,e^{C\int_0^t \|\varphi(s)\|_{H^1(\T;\R)}\,ds},
			\end{equation}
			for any $t\in[0,\tilde{T}]$, { $\mP$-a.s.}
			
			We can now estimate the $L^2(\T;\R)$-norm of $v$. To this end, by differentiating the $L^2(\T;\R)$-norm with respect to time and integrating by parts, we obtain 
			\begin{equation} \label{derivata in L2}
				\begin{split}
					\frac{1}{2}\frac{d}{dt} \|v\|_{L^2(\T;\R)}^2  =\, -\int_{\T} |\pa_x v|^2\,dx-\int_{\T} V^{'}v_{\varphi} \pa_xv\,dx  -\int_{\T} (F^{'} * v_{\varphi}) v_{\varphi}\pa_x v\,dx, \quad {\mP-a.s}.
				\end{split}
			\end{equation}
			By Young's inequality and integration by parts, the second addend on the RHS of \eqref{derivata in L2} can be estimated as 
			\begin{equation}
				\begin{split}
					-\int_{\T} V^{'}v_{\varphi}\pa_xv\,dx =\,&-\int_{\T} V^{'}(v+\varphi)\pa_xv
					= \frac{1}{2}\int_{\T} V^{''}v^2\,dx-\int_{\T}V^{'}\varphi \pa_x v\,dx \\
					\leq\,& \frac{\|V^{''}\|_{L^{\infty}(\T;\R)}}{2}\|v\|_{L^2(\T;\R)}^2 +\|V^{'}\|_{L^{\infty}(\T;\R)}^2\|\varphi\|_{L^2(\T;\R)}^2+\frac{1}{4} \| \pa_x v\|_{L^2(\T;\R)}^2,\quad {\mP-a.s.} \notag
				\end{split}
			\end{equation}
			As for the third addend on the RHS of \eqref{derivata in L2}, we proceed similarly and we are going to use the $L^1(\T;\R)$-norm estimate of $v$ obtained beforehand: 
			\begin{equation}
				\begin{split}
					-\int_{\T} ( F^{'} * v_{\varphi} )v_{\varphi}\pa_xv\,dx =\,& -\int_{\T} (F^{'}*v)v \partial_x v\,dx-\int_{\T} (F^{'}*v)\varphi \partial_x v\,dx\\
					& -\int_{\T} (F^{'}*\varphi)v\partial_x v\,dx-\int_{\T} (F^{'}*\varphi) \varphi \partial_x v\,dx, \quad { \mP-a.s.}
					%\\ =\,&I_{1}+I_{2}+I_{3}+I_{4}.
					\label{Fv_phi}  
				\end{split}
			\end{equation}
			Since $F^{''} \in L^{\infty}(\T;\R)$, by Young's inequality for convolutions, the first two terms in \eqref{Fv_phi} are bounded respectively by
			\begin{equation}
				\begin{split}
					-\int_{\T} (F^{'}*v)v \partial_x v\,dx=\frac{1}{2}\int_{\T} (F^{''}*v) |v|^2 \,dx \leq \frac{1}{2}\|F^{''}\|_{L^{\infty}(\T;\R)}\|v\|_{L^1(\T;\R)}\|v\|_{L^2(\T;\R)}^2, \quad { \mP-a.s.},\notag
				\end{split}
			\end{equation}
			and
			\begin{equation}
				\begin{split}
					-\int_{\T} (F^{'}*v)\varphi \partial_x v\,dx  \leq \frac{1}{4}\|\pa_xv\|_{L^2(\T;\R)}^2+2\pi\|F^{'}\|_{L^{\infty}(\T;\R)}^2\|v\|_{L^2(\T;\R)}^2\|\varphi\|_{L^2(\T;\R)}^2, \quad {\mP-a.s.} \notag
				\end{split}
			\end{equation}
			As for the latter two addends in \eqref{Fv_phi}, by similar arguments we have
			\begin{equation}
				\begin{split}
					\left|\int_{\T} (F^{'}*\varphi)v\partial_x v\,dx\right| & \leq \sqrt{2\pi}\|F^{'}\|_{L^{\infty}(\T;\R)}\|\varphi\|_{L^2(\T;\R)}\int_{\T}|v| |\partial_x v| \,dx \\
					%& \leq  \sqrt{2\pi}\|F^{'}\|_{L^{\infty}(\T;\R)}\|v\|_{L^2(\T;\R)} \|\varphi\|_{L^2(\T;\R)}\|\partial_x v\|_{L^2(\T;\R)} \\
					& \leq 2\pi\|F^{'}\|_{L^{\infty}(\T;\R)}^2\|v\|_{L^2(\T;\R)}^2 \|\varphi\|_{L^{2}(\T;\R)}^2+\frac{1}{4}\|\pa_xv\|_{L^2(\T;\R)}^2, \quad { \mP-a.s.},\notag 
				\end{split}
			\end{equation}
			and
			\begin{equation}
				\begin{split}
					\left|\int_{\T} (F^{'}*\varphi) \varphi \partial_x v\,dx\right| %& \leq \frac{1}{4}\|\partial_xv\|_{L^2(\T;\R)}^2+\|F^{'}*\varphi\|_{L^{\infty}(\T;\R)}^2\|\varphi\|_{L^2(\T;\R)}^2 \\
					%& 
					\leq \frac{1}{4}\|\partial_xv\|_{L^2(\T;\R)}^2+2\pi\|F^{'}\|_{L^{\infty}(\T;\R)}^2\|\varphi\|_{L^2(\T;\R)}^4, \quad { \mP-a.s.} \notag
				\end{split}
			\end{equation}
			Thus, by Gronwall's lemma and using \eqref{L1 norm v} to estimate $\|v\|_{L^1(\T;\R)}$, we obtain the estimate \eqref{L2 norm} for any $t\in[0,\tilde{T}]$. In particular, from these estimates it follows that the solution $v$ to \eqref{pde regolare} does not blow-up in $L^2(\T;\R)$ and, therefore, it can be extended up to time $T$ \quad {$\mP$-a.s.}
		\end{proof}

		\begin{proof}[Proof of Lemma \ref{lemma irriducibilità}]
			Let $f\in L^2([0,T];L^2(\T;\R))$. Writing the deterministic convolution $f_A$ \eqref{deterministic convolution} in Fourier basis, i.e.
			\begin{equation}
				f_A(t)= \sum \limits_{k \in \Z} \lambda_k\left ( \int_0^t e^{-(t-s)k^2}f_k(s)\, ds\right)e_{k},\quad t \in [0,T],\notag
			\end{equation}
			we can see that, for any $\psi \in C^{\infty}(\T;\R)$,
			\begin{equation*}
				\begin{split}
					\langle f_A,\pa_x \psi \rangle_{L^2(\T;\R)}
					& = \sum \limits_{k \in \Z} \lambda_k \int_0^t e^{-(t-s)k^2} f_k(s)\,ds \int_{\T} e_k\pa_x\psi\,dx \\
					&= -\sum \limits_{k \in \Z} |k|\lambda_k\int_0^t e^{-(t-s)k^2} f_k(s)\,ds \int_{\T}e_{-k}\psi\,dx,
				\end{split}
			\end{equation*}
			where we have used the identity $ \pa_x e_k(x) = |k| \, e_{-k}(x)$, $k\in\Z$. From this, we can see that the weak derivative of $f_A$ is given by \eqref{deterministic derivative}. From \eqref{deterministic derivative}, since $Q$ is trace-class, it is easy to see that $\pa_x f_A(t)$ belongs to $L^2(\T;\R)$, for every $t>0$.
			
			Applying the same reasoning, the stochastic convolution $W_{A}$ belongs to $L^2\left(\Omega;C([0,T];H^1(\T;\R))\right)$, thus in a similar fashion to $f_A$ we can show that the weak derivative of $W_A$ is given by 
			\begin{equation*}
				\pa_x W_A(t)= \sum \limits_{k \in \Z} |k|\lambda_k \left( \int_0^t e^{-(t-s)k^2}\,d\beta_s^k \right )e_{-k},
			\end{equation*}
			$dx$-a.e. in $\T$ for all $t \in [0,T]$, $\mP$-a.s.
		\end{proof}

		%%%%%%%%%%%%%%%%%%
		%%%%%%%%%%%%%%%
		
		\begin{proof}[Proof of Lemma \ref{prop:differentiability of u w.r.t u0}] The  part of the statement which is lengthiest to prove is the differentiability of the solution $u(t; u_0)$ of \eqref{lipschitz spde} with respect to the initial datum $u_0$. To do so one starts by considering the so-called {\em first variation} equation, namely the equation
			\begin{equation} \label{first variation PDE}
				\begin{cases}
					\pa_t z(t)=A z(t)+\pa_x[(D\cF)\left( u(t;u_0) \right)z(t)],\quad t \in (0,T], & \\
					z(0)=h,\,&
				\end{cases}
			\end{equation}
			for the unknown $z(t) \in L^2$. We clarify that in the above $(D\cF)\left( u(t;u_0) 
			\right)$ denotes the Fr\'echet derivative of $\cF=\cF(u)$ (with respect to $u$), calculated at 
			the point $u(t;u_0)$. At this point there are (at least) two possible approaches. One 
			approach, which is the one taken in  \cite[Chapter 4]{cerrai}, is to observe that the 
			solution $u(t;u_0)$ is a fixed point of the map $\mathcal I$ defined in \eqref{fixed point
				Lipschitz} and then apply standard results that allow one to deduce differentiability of
			the fixed point from the regularity properties of the fixed point map ($\mathcal I$, in 
			our case), see \cite[Appendix C]{cerrai}. Once the desired differentiability of $u$ is 
			obtained, one observes that $\eta_h=D_{u_0}u(t;u_0) h$ needs to satisfy the first 
			variation equation; from this observation, the estimates \eqref{inequality u} and 
			\eqref{inequality derivata di u} are easily obtained (as we will explain below).  This 
			approach is lengthy but it works in general circumstances. Applied to our case, it allows 
			one to obtain that $u$ is once Fr\'echet differentiable and (at least) twice Gateaux 
			differentiable. We don't take this approach here to contain the length of the paper and 
			because, strictly speaking, we only need one Fr\'echet derivative of the solution, but 
			\cite{Martinthesis} will contain the details of how to use this approach in our case.  
			The approach we take here is the one  of \cite[Theorem 2]{manca}, namely: one first 
			observes that the first variation equation admits a mild solution (by standard contraction
			mapping arguments). Using this fact, it is easy to show that the following inequality 
			holds:
			\begin{equation}\label{in che serve dopo}
				\begin{split}
					\| z(t) \|_{L^2(\T;\R)}^2 + \int_0^t \|\pa_x z(s)\|_{L^2(\T;\R)}^2\,ds \leq \|h\|_{L^2(\T;\R)}^2 + CL_{\cF}^2\int_0^t \|z(s)\|_{L^2(\T;\R)}^2\,ds, \quad  \mP-a.s.\, ,
				\end{split}
			\end{equation}
			where $L_{\cF}$ is the Lipschitz continuity constant of $\cF$ and $C$ is a positive (deterministic) constant, see e.g. \cite[Lemma 5.8 and Prop. 5.9]{dapra04}, from which \eqref{inequality u} and \eqref{inequality derivata di u} are then easily deduced. At this point one shows that there exist a constant $C>0$ and a function $\nu_T(h): L^2 \rightarrow \R_+ $ such that $\nu_T(h)\rightarrow 0$ as $\|h\|_L^2 \rightarrow 0$ and 
			$$
			\|u(t;u_0+h)-u(t;u_0) - z(t)\|_L^2 \leq C \nu_T(h) \|h\|_{L^2},\quad \mathbb P \,\, a.s.
			$$
			Hence $z(t)$ coincides with the Fr\'echet derivative $D_{u_0}u(t;u_0)h$. The proof of the above follows the lines of \cite[Theorem 2]{manca}, with calculations similar to those we have shown so far, so we don't repeat it here.

			The differentiability of the semigroup $\cP_t^{\mathcal F}$ now follows from the differentiability of $u(t;u_0)$. Indeed,  
			as a result of 
			the Banach fixed point theorem we know that the continuous dependence with respect to the initial datum 
			holds, i.e. if $\{u_0^n\}_{n \in \N} \subset L^2(\T;\R)$ such that  $u_0^n \to u_0$ as $n \to +\infty$ in 
			$L^2(\T;\R)$ then $u(t;u_0^n) \to u(t;u_0)$ as $n \to +\infty$ in $L^2 \left( \T;\R \right)$ for all fixed 
			$t \geq 0$. Consequently, if $\psi \in C_b(L^2(\T;\R);\R)$ then from the dominated convergence theorem it 
			follows that $\cP_t^{\cF}\psi \in C_b(L^2(\T;\R);\R)$ for all $t >0$. Hence, $\{ \cP_t^{\cF} 
			\}_{t \geq 0}$ is a Feller semigroup. Furthermore, if $\psi \in 
			C_b^2 \left ( L^2(\T;\R);\R\right )$ then from the differentiation under the integral sign and the fact 
			that $u(t;u_0)$, $t \geq 0$, is Fr\'echet differentiable in $L^2(\T;\R)$ we deduce that $\cP_t^{\cF} 
			\psi$ is Fr\'echet differentiable in $L^2(\T;\R)$ as well (to be precise, since $u(t;u_0)$, $t \geq 0$ is Fr\'echet differentiable and twice G\^ateaux differentiable in $L^2(\T;\R)$ we obtain that $\{\cP_t^{\cF}\}_{t \geq 0}$ is Fr\'echet differentiable and twice G\^ateaux differentiable in $L^2(\T;\R)$).
			
		\end{proof}

	\end{appendix}
\section*{Acknowledgments.} L.A. and M.O. have been supported by the Leverhulme grant  RPG–2020–09. J.B. acknowledges support by the project
RETENU ANR-20-CE40-0005-01 of the French National Research Agency (ANR).

\bibliographystyle{abbrvnat}
\bibliography{bibliography}

\begin{thebibliography}{43}
\providecommand{\natexlab}[1]{#1}
\providecommand{\url}[1]{\texttt{#1}}
\expandafter\ifx\csname urlstyle\endcsname\relax
  \providecommand{\doi}[1]{doi: #1}\else
  \providecommand{\doi}{doi: \begingroup \urlstyle{rm}\Url}\fi

\bibitem[Amos(1974)]{Amo74}
D.~E. Amos.
\newblock Computation of modified {B}essel functions and their ratios.
\newblock \emph{Mathematics of computation}, 28\penalty0 (125):\penalty0
  239--251, 1974.

\bibitem[B{\'e}nyi and Oh(2013)]{Oh}
{\'A}.~B{\'e}nyi and T.~Oh.
\newblock The {S}obolev inequality on the torus revisited.
\newblock \emph{Publicationes Mathematicae Debrecen}, 83\penalty0 (3):\penalty0
  359, 2013.

\bibitem[Berglund(2019)]{NilsBerglundbook}
N.~Berglund.
\newblock An introduction to singular stochastic {PDE}s: {A}llen-{C}ahn
  equations, metastability and regularity structures.
\newblock \emph{arXiv preprint arXiv:1901.07420}, 2019.

\bibitem[Bertini et~al.(2014)Bertini, Giacomin, and
  Poquet]{BertiniGiacominPoquet}
L.~Bertini, G.~Giacomin, and C.~Poquet.
\newblock Synchronization and random long time dynamics for mean-field plane
  rotators.
\newblock \emph{Probability Theory and Related Fields}, 160\penalty0
  (3):\penalty0 593--653, 2014.

\bibitem[Carrillo et~al.(2020)Carrillo, Gvalani, Pavliotis, and
  Schlichting]{pavliotis}
J.~Carrillo, R.~Gvalani, G.~Pavliotis, and A.~Schlichting.
\newblock Long-time behaviour and phase transitions for the
  {M}c{K}ean--{V}lasov equation on the torus.
\newblock \emph{Archive for Rational Mechanics and Analysis}, 235\penalty0
  (1):\penalty0 635--690, 2020.

\bibitem[Carrillo et~al.(2003)Carrillo, McCann, and
  Villani]{CarrilloMcCannVillani}
J.~A. Carrillo, R.~J. McCann, and C.~Villani.
\newblock Kinetic equilibration rates for granular media and related equations:
  entropy dissipation and mass transportation estimates.
\newblock \emph{Revista Matematica Iberoamericana}, 19\penalty0 (3):\penalty0
  971--1018, 2003.

\bibitem[Cerrai(2001)]{cerrai}
S.~Cerrai.
\newblock \emph{Second order {P}{D}{E}'s in finite and infinite dimension: a
  probabilistic approach}, volume 1762.
\newblock Springer Science \& Business Media, 2001.

\bibitem[Chazelle et~al.(2017)Chazelle, Jiu, Li, and Wang]{pav-car}
B.~Chazelle, Q.~Jiu, Q.~Li, and C.~Wang.
\newblock Well-posedness of the limiting equation of a noisy consensus model in
  opinion dynamics.
\newblock \emph{Journal of Differential Equations}, 263\penalty0 (1):\penalty0
  365--397, 2017.

\bibitem[Constantin and Vukadinovic(2004)]{constantin}
P.~Constantin and J.~Vukadinovic.
\newblock Note on the number of steady states for a two-dimensional
  {S}moluchowski equation.
\newblock \emph{Nonlinearity}, 18\penalty0 (1):\penalty0 441, 2004.

\bibitem[Constantin et~al.(2004)Constantin, Kevrekidis, and Titi]{constantin2}
P.~Constantin, I.~Kevrekidis, and E.~S. Titi.
\newblock Remarks on a {S}moluchowski equation.
\newblock \emph{Discrete \& Continuous Dynamical Systems}, 11\penalty0
  (1):\penalty0 101, 2004.

\bibitem[Crisan et~al.(2021)Crisan, Dobson, and Ottobre]{CrisanOttobre}
D.~Crisan, P.~Dobson, and M.~Ottobre.
\newblock Uniform in time estimates for the weak error of the euler method for
  sdes and a pathwise approach to derivative estimates for diffusion
  semigroups.
\newblock \emph{Transactions of the American Mathematical Society},
  374\penalty0 (5):\penalty0 3289--3330, 2021.

\bibitem[Da~Prato(2004)]{dapra04}
G.~Da~Prato.
\newblock \emph{Kolmogorov equations for stochastic {P}{D}{E}s}.
\newblock Springer Science \& Business Media, 2004.

\bibitem[Da~Prato and Gatarek(1995)]{dapra2}
G.~Da~Prato and D.~Gatarek.
\newblock Stochastic {B}urgers equation with correlated noise.
\newblock \emph{Stochastics: An International Journal of Probability and
  Stochastic Processes}, 52\penalty0 (1-2):\penalty0 29--41, 1995.

\bibitem[Da~Prato and Zabczyk(1996)]{ergodicity}
G.~Da~Prato and J.~Zabczyk.
\newblock Ergodicity for infinite dimensional systems.
\newblock \emph{Cambridge University}, 180, 1996.

\bibitem[Da~Prato et~al.(1994)Da~Prato, Debussche, and Temam]{burgers}
G.~Da~Prato, A.~Debussche, and R.~Temam.
\newblock Stochastic {B}urgers’ equation, non-linear differential equations
  and application, 1994.

\bibitem[Dawson(1983)]{Dawson}
D.~A. Dawson.
\newblock Critical dynamics and fluctuations for a mean-field model of
  cooperative behavior.
\newblock \emph{Journal of Statistical Physics}, 31\penalty0 (1):\penalty0
  29--85, 1983.

\bibitem[Dr{\'a}bek and Milota(2007)]{Ricciflow}
P.~Dr{\'a}bek and J.~Milota.
\newblock \emph{Methods of nonlinear analysis: applications to differential
  equations}.
\newblock Springer Science \& Business Media, 2007.

\bibitem[Dressler and Neunzert(1987)]{Dre2010}
K.~Dressler and H.~Neunzert.
\newblock Stationary solutions of the {V}lasov-{F}okker-{P}lanck equation.
\newblock \emph{Mathematical methods in the applied sciences}, 9\penalty0
  (1):\penalty0 169--176, 1987.

\bibitem[Duong and Tugaut(2018)]{DuTu}
M.~H. Duong and J.~Tugaut.
\newblock The {V}lasov-{F}okker-{P}lanck equation in non-convex landscapes:
  convergence to equilibrium.
\newblock \emph{Electronic Communications in Probability}, 23:\penalty0 1--10,
  2018.

\bibitem[Flandoli(1994)]{FlandoliNavier}
F.~Flandoli.
\newblock Dissipativity and invariant measures for stochastic {N}avier-{S}tokes
  equations.
\newblock \emph{Nonlinear Differential Equations and Applications NoDEA},
  1\penalty0 (4):\penalty0 403--423, 1994.

\bibitem[Flandoli and Maslowski(1995)]{FIMas}
F.~Flandoli and B.~Maslowski.
\newblock Ergodicity of the 2-{D} {N}avier-{S}tokes equation under random
  perturbations.
\newblock \emph{Communications in mathematical physics}, 172\penalty0
  (1):\penalty0 119--141, 1995.

\bibitem[Giacomin et~al.(2012)Giacomin, Pakdaman, and
  Pellegrin]{GiacominPakPel}
G.~Giacomin, K.~Pakdaman, and X.~Pellegrin.
\newblock Global attractor and asymptotic dynamics in the {K}uramoto model for
  coupled noisy phase oscillators.
\newblock \emph{Nonlinearity}, 25\penalty0 (5):\penalty0 1247, 2012.

\bibitem[Graham et~al.(2006)Graham, Kurtz, M{\'e}l{\'e}ard, Protter, and
  Pulvirenti]{Pulvirenti}
C.~Graham, T.~G. Kurtz, S.~M{\'e}l{\'e}ard, P.~Protter, and M.~Pulvirenti.
\newblock \emph{Probabilistic {M}odels for {N}onlinear {P}artial {D}ifferential
  {E}quations: {L}ectures {G}iven at the 1st {S}ession of the {C}entro
  {I}nternazionale {M}atematico {E}stivo ({C}{I}{M}{E}) {H}eld in {M}ontecatini
  {T}erme, {I}taly, {M}ay 22-30, 1995}.
\newblock Springer, 2006.

\bibitem[Hairer and Mattingly(2011)]{MattinglyHairer}
M.~Hairer and J.~Mattingly.
\newblock A theory of hypoellipticity and unique ergodicity for semilinear
  stochastic {P}{D}{E}s.
\newblock \emph{Electronic Journal of Probability}, 16:\penalty0 658--738,
  2011.

\bibitem[Herrmann and Tugaut(2010)]{hermann}
S.~Herrmann and J.~Tugaut.
\newblock Non-uniqueness of stationary measures for self-stabilizing processes.
\newblock \emph{Stochastic Processes and their Applications}, 120\penalty0
  (7):\penalty0 1215--1246, 2010.

\bibitem[Herrmann et~al.(2008)Herrmann, Imkeller, and
  Peithmann]{ImkellerPeithmann}
S.~Herrmann, P.~Imkeller, and D.~Peithmann.
\newblock Large deviations and a {K}ramers’ type law for self-stabilizing
  diffusions.
\newblock \emph{The Annals of Applied Probability}, 18\penalty0 (4):\penalty0
  1379--1423, 2008.

\bibitem[Hong and Liu(2021)]{vlasovwellposedness2}
W.~Hong and W.~Liu.
\newblock Distribution dependent stochastic porous media type equations on
  general measure spaces.
\newblock \emph{arXiv preprint arXiv:2103.10135}, 2021.

\bibitem[Hong et~al.(2022)Hong, Li, and Liu]{vlasovwellposedness}
W.~Hong, S.~Li, and W.~Liu.
\newblock Strong convergence rates in averaging principle for slow-fast
  {M}ckean-{V}lasov {S}{P}{D}{E}s.
\newblock \emph{Journal of Differential Equations}, 316:\penalty0 94--135,
  2022.

\bibitem[Kipnis and Landim(1998)]{KipnisLandim}
C.~Kipnis and C.~Landim.
\newblock \emph{Scaling limits of interacting particle systems}, volume 320.
\newblock Springer Science \& Business Media, 1998.

\bibitem[Kolodziejczyk(2022)]{Martinthesis}
M.~Kolodziejczyk.
\newblock \emph{Invariant measures for {M}c{K}ean-{V}lasov {S}{P}{D}{E}s,
  {P}h{D} {T}hesis}.
\newblock in preparation, 2022.

\bibitem[Kruse(2014)]{kruse}
R.~Kruse.
\newblock \emph{Strong and weak approximation of semilinear stochastic
  evolution equations}.
\newblock Springer, 2014.

\bibitem[Kuksin and Shirikyan(2002)]{Kuksin}
S.~Kuksin and A.~Shirikyan.
\newblock Coupling approach to white-forced nonlinear {P}{D}{E}s.
\newblock \emph{Journal de math{\'e}matiques pures et appliqu{\'e}es},
  81\penalty0 (6):\penalty0 567--602, 2002.

\bibitem[Kuramoto(1975)]{Kuramoto2}
Y.~Kuramoto.
\newblock Self-entrainment of a population of coupled non-linear oscillators.
\newblock In \emph{International symposium on mathematical problems in
  theoretical physics}, pages 420--422. Springer, 1975.

\bibitem[Kuramoto(1981)]{Kuramoto}
Y.~Kuramoto.
\newblock Rhythms and turbulence in populations of chemical oscillators.
\newblock \emph{Physica A: Statistical Mechanics and its Applications},
  106\penalty0 (1-2):\penalty0 128--143, 1981.

\bibitem[Lacker(2018)]{Lacker}
D.~Lacker.
\newblock Mean field games and interacting particle systems.
\newblock \emph{Preprint}, 2018.

\bibitem[Manca(2006)]{manca}
L.~Manca.
\newblock On a class of stochastic semilinear {P}{D}{E}s.
\newblock \emph{Stochastic analysis and applications}, 24\penalty0
  (2):\penalty0 399--426, 2006.

\bibitem[M{\'e}l{\'e}ard and Bansaye(2015)]{Meleard}
S.~M{\'e}l{\'e}ard and V.~Bansaye.
\newblock Some stochastic models for structured populations: scaling limits and
  long time behavior.
\newblock \emph{arXiv preprint arXiv:1506.04165}, 2015.

\bibitem[Pareschi and Toscani(2013)]{PareschiToscani}
L.~Pareschi and G.~Toscani.
\newblock \emph{Interacting multiagent systems: kinetic equations and {M}onte
  {C}arlo methods}.
\newblock OUP Oxford, 2013.

\bibitem[Peszat and Zabczyk(1995)]{peszat}
S.~Peszat and J.~Zabczyk.
\newblock Strong {F}eller property and irreducibility for diffusions on
  {H}ilbert spaces.
\newblock \emph{The Annals of Probability}, pages 157--172, 1995.

\bibitem[Tartar(2007)]{Tartar}
L.~Tartar.
\newblock \emph{An introduction to Sobolev spaces and interpolation spaces},
  volume~3.
\newblock Springer Science \& Business Media, 2007.

\bibitem[Tugaut(2014)]{tugaut}
J.~Tugaut.
\newblock Phase transitions of {M}c{K}ean--{V}lasov processes in double-wells
  landscape.
\newblock \emph{Stochastics An International Journal of Probability and
  Stochastic Processes}, 86\penalty0 (2):\penalty0 257--284, 2014.

\bibitem[Vukadinovic(2009)]{Vukadinovic}
J.~Vukadinovic.
\newblock Inertial manifolds for a {S}moluchowski equation on the unit sphere.
\newblock \emph{Communications in mathematical physics}, 285\penalty0
  (3):\penalty0 975--990, 2009.

\bibitem[Weinan et~al.(2001)Weinan, Mattingly, and Sinai]{Weinan}
E.~Weinan, J.~C. Mattingly, and Y.~Sinai.
\newblock Gibbsian dynamics and ergodicity for the stochastically forced
  {N}avier-{S}tokes equation.
\newblock \emph{Comm. Math. Phys}, 224\penalty0 (1):\penalty0 83--106, 2001.

\end{thebibliography}

\end{document}